\documentclass[12pt,oneside,reqno,fleqn]{amsart}
\usepackage{amssymb}
\usepackage{bbm}
\usepackage{cases}
\usepackage{amsmath,mathtools}
\usepackage{graphicx}
\usepackage{mathrsfs}
\usepackage{stmaryrd}
\usepackage{color}
\usepackage{soul}
\usepackage[dvipsnames]{xcolor}
\usepackage{amsfonts}
\usepackage{amsmath,amssymb,amsthm}
\usepackage{thmtools}
\usepackage{libertine}
\usepackage[T1]{fontenc}
\usepackage[libertine,varbb]{newtxmath}
\usepackage[colorlinks,linkcolor=Red,citecolor=Red,urlcolor=Red]{hyperref}
\usepackage[shortlabels] {enumitem}
\usepackage[nameinlink,capitalize]{cleveref}
\usepackage[spacing=true,kerning=true,tracking=true,letterspace=50]{microtype}
\usepackage{orcidlink}
\usepackage{ulem}
\usepackage[]{todonotes}

\hypersetup{pdftitle={Martingale problem}, pdfauthor={Fabio Bugini, Michele Coghi, Khoa L\^e and Chengcheng Ling}}
\numberwithin{equation}{section}
\newcommand{\be}{\begin{eqnarray}}
\newcommand{\ee}{\end{eqnarray}}
\newcommand{\ce}{\begin{eqnarray*}}
\newcommand{\de}{\end{eqnarray*}}
\declaretheorem[name=Theorem,numberwithin=section]{theorem}
\declaretheorem[name=Lemma,sibling=theorem]{lemma}
\declaretheorem[name=Proposition,sibling=theorem]{proposition}
\declaretheorem[name=Corollary,sibling=theorem]{corollary}
\declaretheorem[name=Assumption,sibling=theorem,style=remark]{assumption}

\declaretheorem[name=Example,sibling=theorem,style=remark]{example}
\declaretheorem[name=Remark,sibling=theorem,style=remark]{remark}
\declaretheorem[name=Definition,sibling=theorem,style=remark]{definition}

\crefname{theorem}{Theorem}{Theorems}
\crefname{lemma}{Lemma}{Lemmas}
\crefname{proposition}{Proposition}{Propositions}
\crefname{corollary}{Corollary}{Corollaries}
\crefname{definition}{Definition}{Definitions}
\crefname{remark}{Remark}{Remarks}
\crefname{example}{Example}{Examples}
\crefname{condition}{Condition}{Conditions}
\crefname{eqn}{Equation}{Equations}
\crefname{assumption}{Assumption}{Assumptions}
\crefname{innercustomthm}{Condition}{Conditions}
\crefrangelabelformat{innercustomthm}{#3#1#4-#5#2#6}

\def\bb2{{\boldsymbol{2}}}

\def\={&\!\!=\!\!&}

\def\<{{\langle}}
\def\>{{\rangle}}

\def\bx{{\mathbf{x}}}

\newcommand{\I}{{\mathcal{I}}}

\def\min{{\mathord{{\rm min}}}}

\def\={&\!\!=\!\!&}
\def\bt{\begin{theorem}}
\def\et{\end{theorem}}
\def\bl{\begin{lemma}}
\def\el{\end{lemma}}
\def\br{\begin{remark}}
\def\er{\end{remark}}
\def\bd{\begin{definition}}
\def\ed{\end{definition}}
\def\bp{\begin{proposition}}
\def\ep{\end{proposition}}
\def\bc{\begin{corollary}}
\def\ec{\end{corollary}}
\def\bx{\begin{example}}
\def\ex{\end{example}}

\def\cB{{\mathcal B}}

\def\cF{{\mathcal F}}

\def\cM{{\mathcal M}}

\def\cN{{\mathcal N}}

\def\cP{{\mathcal P}}

\def\cS{{\mathcal S}}

\def\mC{{\mathbb C}}

\def\mE{{\mathbb E}}

\def\mN{{\mathbb N}}

\def\mP{{\mathbb P}}

\def\mR{{\mathbb R}}

\def\mZ{{\mathbb Z}}
\def\bP{{\mathbf P}}
\def\bE{{\mathbf E}}

\def\sN{{\mathscr N}}

\def\geq{\geqslant}
\def\leq{\leqslant}
\def\les{\lesssim}

\def\bP{{\mathbf P}}

\newcommand{\R}{{\mathbb R}}
\newcommand{\Rd}{{\R^d}}
\newcommand{\tand}{\quad\text{and}\quad}
\newcommand{\LL}{{\mathbb{L}}}
\newcommand{\LLb}{{\widebar\LL}}
\newcommand{\Lb}{{\widebar L}}

\newcommand{\1}{{\mathbf 1}}
\newcommand{\caa}{{\mathcal A}}
\newcommand{\sol}{{\mathbb{S}}}
\newcommand{\solb}{{\widebar\sol}}
\newcommand{\sone}{{\sol_{1,p}}}
\newcommand{\stwo}{{\sol_{2,p}}}

\newcommand{\ival}{{\mathcal{T}}}

\newcommand{\ky}{{\it KY }}
\newcommand{\norm}[1]{{\left\vert\kern-0.25ex\left\vert\kern-0.25ex\left\vert #1 
    \right\vert\kern-0.25ex\right\vert\kern-0.25ex\right\vert}}

\renewcommand{\le}{\leq}
\renewcommand{\ge}{\geq}
\allowdisplaybreaks
\begin{document}
    \title{Young Parabolic PDEs and the Martingale Problem for Young SDEs}
	\date{\today}
	
\author{Fabio Bugini, Michele Coghi,
Khoa L{\^e} 
\and Chengcheng Ling 
}
\address{Fabio Bugini: (formerly)
Technische Universit\"at Berlin, Institute of Mathematics,  10623 Berlin, Germany}
\address{Michele Coghi:
University of Trento, Department of Mathematics, 38123 Povo, Italy}
\email{michele.coghi@unitn.it}
\address{Khoa L{\^e}:
University of Leeds,
 School of Mathematics, Leeds LS29JT, U.K.}
\email{ k.le@leeds.ac.uk}
\address{Chengcheng Ling:
University of Augsburg, Institut für Mathematik, Universitätsstraße 14, 86159 Augsburg,
 Germany  }
 \email{ chengcheng.ling@uni-a.de}
	\begin{abstract}
   Under mild regularity assumptions, we prove that the martingale problem associated with the hybrid Young--Lyons--It\^o differential equation admits a unique solution in the natural class of solutions satisfying Krylov-type estimates, thereby establishing probabilistic weak well-posedness. Our proof relies on an analysis of the associated Kolmogorov equations, which are Young-type parabolic PDEs with irregular coefficients. The resulting theory for such Young-type parabolic PDEs is also of independent interest.
\bigskip
		

		
	\end{abstract}
	
	\maketitle
	
\section{Introduction}
It has long been recognized that the addition of stochastic terms to the basic governing equations
can be used to model an intrinsic presence of randomness as well as to account for other numerical, empirical or physical uncertainties.  Consequently, the field of stochastic 
equations has gained enormous importance over the past decades. Starting with It\^o integration theory, it is nowadays well-known that we can give a probabilistic meaning to problems that are analytically ill-posed due to the low regularity of trajectories of the driving stochastic processes. For It\^o-type equations, the pathwise strong solution theory characterizes solutions through explicit sample-path constructions but typically requires strong regularity of the coefficients; by contrast, the weak solution framework relaxes the requirement of constructing solutions on a fixed probability space by allowing both the process and the probability space to vary, thus permitting existence results under milder assumptions while still retaining a probabilistic construction; and the Stroock--Varadhan martingale problem formulation goes further by characterizing solutions in terms of the generator of the diffusion, focusing directly on the laws of processes rather than individual paths, which allows one to establish existence and uniqueness under minimal regularity conditions. These results are well established, and we refer to the books \cite{ok, ry, ka, SV} for details.   

However, it turns out that, in general, the It\^o solution map, which sends a trajectory of the driving signal to a trajectory of the solution, is not continuous; moreover, this methodology is restricted to the semimartingale/Markov setting. This loss
of robustness was overcome by the theory of rough paths introduced by Lyons \cite{TL}, which allows for a pathwise analysis of stochastic equations.  More precisely, Lyons singled out the appropriate
topology on the space of rough paths which renders the corresponding Itô–Lyons solution map
continuous as a function of a suitably enhanced driving path. In contrast to It\^o's theory, rough path theory allows one
to consider drivers beyond the martingale setting, such as general Gaussian or Markov processes \cite{FH}.

Recently, using stochastic sewing \cite{le20, le23}, the authors of \cite{FHL21} provided a unified framework for studying randomly perturbed ODEs which combines It\^o integration and Lyons' rough path theory. These equations are known as rough SDEs \cite{FHL21, BCN, BFS, FLZ} and in the Young regime, they take the form 
\begin{align}
    \label{eq:rough-SDE}
dX_t=b(t,X_t)dt+g(t,X_t)d\gamma_t+\sigma(t,X_t)dW_t
\end{align}
where $b:[0,T]\times\mathbb{R}^d\rightarrow\mathbb{R}^d$, $g:[0,T]\times\mR^d\rightarrow\mathbb{R}^{d\times d_1} $, $\gamma$ is a $d_1$-dimensional $\beta$-H\"older continuous path, $\sigma:[0,T]\times\mR^d\rightarrow\mathbb{R}^{d\times d_2} $ and $W$
is a standard $d_2$-dimensional Brownian motion. Here $\sigma(t,X_t)dW_t$ is understood in the sense of It\^o integration and 
$g(t,X_t)d\gamma_t$ is defined as a rough integral. In this work we  focus on the Young regime when $\beta > \frac{1}{2}$. For clarity, we assume without loss of generality that $d_1=1$ and $d_2=d$; our analysis, however, applies equally to general $d_1,d_2$.    

 Equations of the type \eqref{eq:rough-SDE} not only have their own theoretical interest from a stochastic analysis perspective, but also have a wide range of applications in fields such as stochastic control, filtering theory and mathematical finance, as discussed in \cite{FHL21,FLZ26}.

Concerning the solution theory, assuming more than Lipschitz regularity of $b$ and $\sigma$, \cite{GN, HL} show existence and uniqueness of strong solutions to \eqref{eq:rough-SDE} for $\beta>\frac{1}{2}$, where $\gamma$ is a fractional Brownian motion with Hurst index larger than $\frac{1}{2}$.  When $\beta > \tfrac{1}{3}$ and the driving signal $\gamma$ can be lifted to a rough path, it has been established in \cite{FHL21} that \eqref{eq:rough-SDE} admits a unique pathwise strong solution if the coefficients are (at least) Lipschitz continuous. A central ingredient of the analysis in \cite{FHL21} is the construction of a suitable metric that simultaneously captures the probabilistic and rough path structures of the driving noise, thereby enabling a fixed-point argument that yields existence and uniqueness of solutions. Since such rough SDEs still involve an Itô-type noise component, which is both Markovian and martingale in nature, it is natural to expect that a weak solution theory for these equations should also be viable. This observation forms the starting point of the present work.

Motivated by the recent advances in the theory of regularization by noise \cite{Zvo, KR, Z, FIR, GaGe, ALL, CD, DGLL}, one may anticipate that the coefficients of rough SDEs can be treated under weaker assumptions. A natural approach, paralleling the weak solution theory for Itô SDEs, is to study the associated Kolmogorov equation, which in this context becomes a rough PDE with singular coefficients. This perspective naturally reduces the problem to the analysis of such rough PDEs. Since our rough SDE is formulated in the Young regime---referred to as the hybrid Young–Lyons–Itô differential equation---the corresponding rough PDE also lies in the Young setting. Our approach to this analysis is discussed in detail in a later subsection.

 By exploiting the regularization effect of the It\^o-type noise, we establish {\it a martingale and weak solution theory for the Young–Lyons–Itô differential equation \eqref{eq:rough-SDE} under mild regularity conditions on the coefficients.} Compared with the existing literature on rough SDEs and rough PDEs, our conditions on $b$ and $g$ are much weaker (no Lipschitz or even continuity requirement on $b$, which may be unbounded and singular). In return, we work in the Young regime $\beta>\frac12$ and require the diffusion coefficient to be uniformly nondegenerate and H\"older continuous in space. Our main hypotheses and results are presented below, after we introduce the notation.
 \subsection{Notation}\label{sub:not}

On finite dimensional vector spaces we always use the Euclidean norm.

For $k\in\{1,\ldots,d\}$ and $f:\mathbb{R}^d\to \mathbb{R}$, we denote $\partial_k f(x):=\frac{\partial f(x)}{\partial x_k}$ for $x\in\mathbb{R}^d$ and $\nabla f(x):=(\partial_if(x))_{1\leq i\leq d}$; derivatives are understood in the weak sense.
For vector-valued $f$ we use the same notation, and
$\nabla^k f$ is defined via $\nabla(\nabla ^{k-1}f)$ iteratively. For a multi-index $k=(k_1,\ldots,k_d)\in\mN^d$, denote $\partial^k f(x):=\frac{\partial^{|k|} f(x)}{\partial x_1^{k_1}\cdots\partial x_d^{k_d}}$. If $k=(0,\ldots,0)$, we use the convention $\partial^kf=f$.

 For each $\nu\in\mZ_+$ and $p\in[1,\infty]$, $L_{\nu,p}(\Rd)
 $ is the usual Sobolev space on $\Rd$ equipped with the norm 
$\|f\|_{L_{\nu, p}(\mathbb{R}^d)}:=\sum_{k=0}^\nu\|\nabla^k f\|_{L_{ p}(\mathbb{R}^d)}$. For $T\in[0,\infty],q\in[1,\infty],$  $\LL^q_{\nu,p}([0,T])$ denotes the space of measurable functions $f:[0,T]\to L_{\nu,p}(\Rd)$ such that $\|f\|_{\LL^q_{\nu,p}([0,T])}$ is finite. Here, for each $s,t\in[0,T]$ satisfying $s\le t$, we denote
 \begin{align*}
  	\|f\|_{\LL^q_{\nu,p}([s,t])}:=\left(\int_s^t\|f(r,\cdot)\|_{L_{\nu,p}(\Rd)}^qdr\right)^{\frac1q}
 \end{align*}
 with the standard modification when $q=\infty$ or $p=\infty$. When $\nu=0$, we simply write $\LL^q_p([0,T])$ instead of $\LL^q_{0,p}([0,T])$.
 In particular, $\LL^q_p([0,T])$ contains Borel measurable functions $f:[0,T]\times\Rd\to\R$ such that $\int_0^T\left[\int_\Rd|f(t,x)|^pdx\right]^{q/p}dt$ is finite.
For $p\in(1,\infty)$ and $\nu\in\R$ we also denote by $L_{\nu,p}(\Rd):=(I-\Delta)^{-\nu/2}\big(L_p(\Rd)\big)$ the Bessel potential space with norm $\|f\|_{L_{\nu,p}(\Rd)}:=\|(I-\Delta)^{\nu/2}f\|_{L_p(\Rd)}$; for $\nu\in\mZ_+$ it coincides with the Sobolev space above, with an equivalent norm \cite{Triebel}. The spaces $\LL^q_{\nu,p}([0,T])$ are defined for real $\nu$ in the same way. We shall use the following standard facts \cite{Triebel}: $\partial_k:L_{\nu,p}\to L_{\nu-1,p}$ is bounded; $L_{\nu,p}\hookrightarrow L_{\nu',p}$ for $\nu'\le\nu$; $L_{\nu,p}\hookrightarrow C_b(\Rd)$ if $\nu>d/p$; the dual of $L_{\nu,p}$ is $L_{-\nu,p'}$ with $p'=\frac p{p-1}$; and $[L_{\nu_0,p},L_{\nu_1,p}]_\lambda=L_{(1-\lambda)\nu_0+\lambda\nu_1,p}$ (complex interpolation), so that in particular $\|f\|_{L_{(1-\lambda)\nu_0+\lambda\nu_1,p}}\lesssim \|f\|_{L_{\nu_0,p}}^{1-\lambda}\|f\|_{L_{\nu_1,p}}^{\lambda}$ for $\lambda\in[0,1]$.
Let $(\Delta_j)_{j\ge0}$ be the Littlewood--Paley blocks associated with a dyadic partition of unity in Fourier space \cite[Section 2.2]{BCD}, so that $\widehat{\Delta_jf}$ is supported in the ball of radius $\frac83$ for $j=0$ and in the annulus $\{\frac34 2^j\le|\xi|\le\frac83 2^j\}$ for $j\ge1$ (in the notation of \cite[Section 2.2]{BCD}, $\Delta_0=\chi(D)+\varphi(D)$ and $\Delta_j=\varphi(2^{-j}D)$ for $j\ge1$). For $\nu\in\R$ and $p\in(1,\infty)$, the Besov space $B^\nu_{p,\infty}=B^\nu_{p,\infty}(\Rd)$ consists of the tempered distributions $f$ with $\|f\|_{B^\nu_{p,\infty}}:=\sup_{j\ge0}2^{j\nu}\|\Delta_jf\|_{L_p}<\infty$. We use the following standard facts \cite{Triebel,BCD}: $L_{\nu+\delta,p}\hookrightarrow B^\nu_{p,\infty}\hookrightarrow L_{\nu-\delta,p}$ for $\delta>0$, $L_{k,p}\hookrightarrow B^k_{p,\infty}$ for $k\in\mZ_+$, and in particular $L_p\hookrightarrow B^0_{p,\infty}$ and $B^\nu_{p,\infty}\hookrightarrow L_p$ for $\nu>0$; $B^\nu_{p,\infty}\hookrightarrow C^{\nu-d/p}(\Rd)\subset C_b(\Rd)$ if $\nu-\frac dp\in(0,1)$; $\partial_k:B^\nu_{p,\infty}\to B^{\nu-1}_{p,\infty}$ is bounded; $\|f\|_{B^{(1-\lambda)\nu_0+\lambda\nu_1}_{p,\infty}}\le\|f\|_{B^{\nu_0}_{p,\infty}}^{1-\lambda}\|f\|_{B^{\nu_1}_{p,\infty}}^\lambda$ for $\lambda\in[0,1]$; and, for $\varrho\in(0,1)$, $\|m\|_{C^\varrho}\approx\sup_j2^{j\varrho}\|\Delta_jm\|_{L_\infty}$. Since $B^\nu_{p,\infty}$ is not separable, we use the following convention: a function $v:[0,T]\to B^\nu_{p,\infty}$ is called measurable if it is strongly measurable as a map into the separable space $B^{\nu-1}_{p,p}\supset B^\nu_{p,\infty}$; $L^\infty([0,T];B^\nu_{p,\infty})$ is the space of such measurable $v$ with $\sup_t\|v_t\|_{B^\nu_{p,\infty}}<\infty$; and time integrals are Bochner integrals in $B^{\nu-1}_{p,p}$. Since $\Delta_j$ commutes with these integrals, $\|\int v_rdr\|_{B^\nu_{p,\infty}}\le\int\|v_r\|_{B^\nu_{p,\infty}}dr$.
We  say that $f\in \LL^{1,q}_{\nu,p}([s,t])$ 
if $f\in\LL^q_{\nu,p}([s,t])$ and $\partial_tf\in\LL^q_{p}([s,t])$. We denote   $\|f\|_{\LL^{1,q}_{\nu,p}([s,t])}:=\|f\|_{\LL^{q}_{\nu,p}([s,t])}+\|\partial_tf\|_{\LL^{q}_{p}([s,t])}$.

 For $\alpha\in(0,1]$,  we set $C^\alpha(\mR^d)$ to be the space of continuous functions such that
\begin{align*}
  \Vert f\Vert_{C^\alpha}:=[f]_{C^\alpha}+\sup_{x\in\mR^d}|f(x)|:=\sup_{x,y\in\mR^d,x\neq y}\frac{|f(x)-f(y)|}{|x-y|^\alpha}+\sup_{x\in\mR^d}|f(x)|<\infty.
\end{align*}
Here, and often below, we write $C^\alpha$ instead of $C^\alpha(\mR^d)$ for simplicity.  
For $\alpha\in(0,\infty)$, we define ${C}^\alpha(\mR^d)$ as the space of all functions $f$ defined on $\mR^d$ having bounded derivatives $\partial^k f$ for multi-indices $k\in\mN^d$ with $|k|\leq \alpha$ so that
\begin{align*}
  \Vert f\Vert_{{C}^\alpha}&:=\|f\|_{{C}^{ \lfloor\alpha\rfloor}}+[f]_{{C}^\alpha}:=\sum_{|k|\leq \alpha}\sup_{x\in\mR^d}|\partial^kf(x)|+\sum_{\alpha-1\leq|k|<\alpha} [\partial^kf]_{{C}^{\alpha-|k|}}<\infty.
\end{align*}   
Analogously we can also define the $\alpha$-H\"older continuous path space $C^\alpha([0,T])$ for any $T\in(0,\infty)$ and the $E$-valued path space $C^\alpha([0,T];E)$, where $(E,\|\cdot\|_E)$ is a normed space.

%
Let $Q:[0,T]^2 \to \R^{d\times d}$ and define 
\begin{align*}
    p^Q_{s,t}(x)=(2\pi)^{-d/2}(\det Q_{s,t})^{-1/2}\exp\left(-\frac12x^\top Q_{s,t}^{-1}x \right)
\end{align*}
and $P^Q_{s,t}f(x)=\int_{\Rd}p^Q_{s,t}(x-y)f(y)dy $ for any measurable function $f$ on $\Rd$. 
Define  $p_t(x):=\frac{1}{\sqrt{2\pi t}^d}e^{-\frac{|x|^2}{2t}}$  and $P_tf(x):=p_t\ast f(x), t\geq0$. It is evident that when $Q_{s,t}=I(t-s) $, $p^Q_{s,t}(x)=p_{t-s}(x)$ and $P^Q_{s,t}=P_{t-s}$.  

We say that a sequence of functions $(f_n)_{n\in \mN}$ converges to a function $f$ in $C^{\alpha-}(\mR^d)$, $\alpha\in(0,\infty)$,  if $\sup_n\|f_n\|_{C^\alpha}<\infty$ and for any $\alpha'<\alpha$ we have $\|f_n-f\|_{C^{\alpha'}}\rightarrow0$ as $n\rightarrow\infty$. We note that such a sequence always exists given $f\in C^\alpha(\mR^d)$: one can simply take $f_n:=P_{1/n}f$. Convergence in $C^{\alpha-}([0,T])$ of paths is defined in the same way; for $\gamma\in C^\alpha([0,T])$, such a sequence of smooth paths is obtained by extending $\gamma$ constantly outside $[0,T]$ and mollifying. 

We adopt the Einstein summation convention over repeated indices.

The following are the main function spaces that we work with throughout the paper.
\begin{definition}\label{def.S}
    We denote 
\begin{align}
  \label{def:sol2p-m} &\sol_{2,p}=\LL^\infty_{2,p}([0,T])\cap C^{\frac{1}{2}}([0,T];L_{1,p}(\mR^d)),
  \\&\|\cdot\|_{ \sol_{2,p}}:=\max\{\|\cdot\|_{\LL^\infty_{2,p}([0,T])}, \|\cdot\|_{C^{\frac{1}{2}}([0,T];L_{1,p}(\mR^d))}\}
 \nonumber
  \\ &\sol_{1,p}=\LL^\infty_{1,p}([0,T])\cap C^{\frac{1}{2}}([0,T];L_{p}(\mR^d))
  , \label{def:sol1p-m}
  \\&\|\cdot\|_{ \sol_{1,p}}:=
  \max\{\|\cdot\|_{\LL^\infty_{1,p}([0,T])}, \|\cdot\|_{ C^{\frac{1}{2}}([0,T];L_{p}(\mR^d))}
  \}\nonumber
\end{align}
\end{definition}
By the continuous embeddings $L_{2,p}(\Rd)\hookrightarrow L_{1,p}(\Rd)\hookrightarrow L_{p}(\Rd)$, we have
\begin{align}\label{eq.Sembed}
    \stwo\hookrightarrow\sone.
\end{align}

For  a given interval $[0, T]$, $[0,T]_\leq^n$ denotes the $n$-th simplex of $[0, T]$.

The notation $a\lesssim b$ means that there exists $C>0$ such that $a\leq C b$, where $C$ depends only on the parameters claimed in the corresponding statement. If the constant depends on any further parameter $c$, we incorporate it in the notation by writing $a\lesssim_c b$.

\subsection{Main Results}

The following conditions on the coefficients of \eqref{eq:rough-SDE} and of its Kolmogorov equation are imposed throughout the paper.  
\begin{assumption}
  \label{ass:main}
  We assume for $p\in(d,\infty)$, 
  \begin{itemize}
    \item[$(H^b)$]   $b\in \LL_q^\infty([0,T])$  for $q\in[p,\infty]$;
    \item[$(H^g)$] $g:[0,T]\times\mathbb{R}^d\rightarrow\mathbb{R}^{d}$ and  $g\in \sol_{1,p}$;
    \item[$(H^\sigma)$] 
    \begin{itemize}
      \item[(i)] $\sigma$ is uniformly bounded and nondegenerate. Namely, there exists a constant $K$ such that
\[
K^{-1} |\xi|^2 \leq |\sigma^\ast (t,x) \xi|^2 \leq K|\xi|^2 \quad
\forall\,\xi\in\R^{d},\,(t,x)\in [0,T] \times \R^d,
\] 
\item[(ii)] $\sigma$ is uniformly continuous on $[0,T]\times\mR^{d}$, and H\"older continuous in space, uniformly in time: for some $\alpha\in(0,1]$, $|\sigma(t,x)-\sigma(t,y)|\le K|x-y|^\alpha$ for all $t\in[0,T]$, $x,y\in\Rd$.
    \end{itemize}
  \end{itemize}
\end{assumption}
We summarize our main results here; the detailed statements are \cref{thm:existence}, \cref{thm:unique} and \cref{thm:Mark} in \cref{sec:mar-weak}, where the proofs are also given.
\begin{theorem}
    \label{thm:main-intro}
        Assume \cref{ass:main} holds and $\gamma\in C^\beta([0,T])$ with $\beta\in(\frac12,1)$. Then, for every starting point, the martingale problem associated with the Young--Lyons--It\^o differential equation \eqref{eq:rough-SDE} has exactly one solution in the class of solutions satisfying the Krylov--Young estimates of \cref{lem:Y-kry}. It is the limit of the laws of the solutions to the equations with mollified $\gamma$. Weak solutions satisfying these estimates exist and are unique in law, and their law is the above martingale solution. Moreover, the space-time process $(t,X_t)$ associated with the unique solution is a time-homogeneous Markov process.
\end{theorem}

\subsection{Idea of the analysis and organization}
Heuristically, we would like to use the regularization from the diffusion $\int_0^t\sigma(s,X_s)dW_s$ to reduce the regularity requirement from \cite{FHL21}. In the classical SDE setting, this is done by studying the corresponding Kolmogorov equation \cite{SV}. Here the Kolmogorov equation of \eqref{eq:rough-SDE} contains a Young differential in time:
 \begin{align}
    \label{main-rPDE}
    d_t u+\Big(\frac{1}{2}a^{ij}\partial^2_{ij}u+b\cdot \nabla u\Big)dt+(g\cdot \nabla u) d \gamma_t=fdt \text{ on } [0,T]\times\mR^d, \quad u(T)=\phi,
\end{align}
where $a=\sigma\sigma^\top$.
Since $\gamma$ is not differentiable, \eqref{main-rPDE} does not fall into the classical $L_p$-theory of parabolic equations \cite{KH, KS, LSU, XXZZ}. Even for constant $a$ and smooth $\gamma$, $u_t=P_{T-t}u_T+\int_t^T P_{r-t}fdr+\int_t^T P_{r-t}g_r d \gamma_r$, and maximal regularity estimates, well understood for the Lebesgue integral $\int_t^T P_{r-t}fdr$, are much less developed for the Young integral $\int_t^T P_{r-t}g_r d \gamma_r$. We therefore first develop a solution theory for \eqref{main-rPDE} (\cref{sec:PDE}), and then show the well posedness of the martingale problem through it (\cref{sec:tightness,sec:mar-weak}).\\

{\it Solution theory for Young PDE.} If $a=a(t)$ does not depend on $x$ and $b=0$, a solution of \eqref{main-rPDE} satisfies the mild formulation
\begin{align*}
    P^a_{s,t}u_t-u_s=-\int_s^tP^a_{s,r}(g_r\cdot\nabla u_r)d\gamma_r+\int_s^tP^a_{s,r}f_rdr,
\end{align*}
where $P^a$ is the Gaussian semigroup generated by $\frac12a^{ij}\partial^2_{ij}$. The Young integral is constructed by the sewing lemma with the germ $A_{r,r'}=P^a_{s,r}(g_r\cdot\nabla u_r)(\gamma_{r'}-\gamma_r)$. The smoothing of $P^a$ compensates the roughness of $\gamma$: if $g\cdot\nabla u$ is bounded in time with values in $B^1_{p,\infty}$ and $\frac12$-H\"older in time with values in $B^0_{p,\infty}$, the Young integral gains one spatial derivative (\cref{lem:youngPdg}). This closes an a priori estimate for $u$ in the space $\sol^0_{2,p}$ of \cref{def:Se}, i.e.\ $u$ bounded in time with values in $B^2_{p,\infty}$ and $\frac12$-H\"older in time with values in $B^1_{p,\infty}$. The Besov scale is essential: the $dt$-part $\int P^a f dr$ does not map $\LL^\infty_p$ into $\LL^\infty_{2,p}$, but it does map it into $L^\infty_tB^2_{p,\infty}$ (see \cref{rem:why-besov}). General coefficients $a(t,x)$ are treated by freezing them at a point $z$; since $\partial^2u$ is only a distribution in $B^0_{p,\infty}$, the commutator is controlled by a pointwise multiplier estimate, which is where the H\"older continuity of $a$ in space enters. The drift $b$ is a lower-order perturbation (\cref{Thm:est-roughPDE,thm:solution-YPDE}).\\

{\it The martingale problem.} Following the theory of singular It\^o-type equations \cite{KR, ZZ18, Zhangzhao, XXZZ}, we single out the solutions satisfying Krylov-type estimates adapted to the Young term, namely
\begin{align}\label{est:in-Kry}
   \mE\Big(\int_s^t|f(r,X_r)|dr+\big|\int_{s}^{t}\rho(r,X_r)d\Upsilon_r\big|\,\Big|\,\cF_s\Big)\lesssim |t-s|^\theta\big(\|f\|_{\LL^\infty_p}+\Vert \Upsilon\Vert_{C^{\beta_0}} \Vert \rho\Vert_{\sol^0_{1,p}}\big)
\end{align}
for smooth paths $\Upsilon$ and some $\theta>\frac12$, where $\beta_0\in(\frac12,\beta)$ is fixed in \eqref{def:beta-eps} (see \cref{lem:Y-kry}). This complements the classical Krylov estimate concerning only the first integral in \eqref{est:in-Kry}.
The solutions of the equations driven by smooth approximations $\gamma^n$ of $\gamma$ satisfy \eqref{est:in-Kry} uniformly in $n$: by It\^o's formula the left-hand side is expressed through the Kolmogorov equation with data $fdt+\rho d\Upsilon$, which is controlled by the Young PDE theory. A tightness argument then gives existence of martingale and weak solutions in this class (\cref{sec:tightness,sec:mar-weak}). For uniqueness, we show that every martingale solution $\bP$ satisfying \eqref{est:in-Kry} has the conditional representation $\bE^\bP(\phi(w_t)|\cB_r)=u(r,w_r)$, where $u$ solves \eqref{main-rPDE} with $f=0$ and $u(t)=\phi$, by applying the martingale problem to the equations with $\gamma^n$ in place of $\gamma$ and controlling the error by \eqref{est:in-Kry}. The representation determines the finite-dimensional distributions and yields the Markov property. Thus \eqref{est:in-Kry} identifies the class within which uniqueness holds, similarly to \cite{Zhangzhao, GP, HZ}; here, however, the singularity comes from the time variable through $\gamma$ rather than from the space variable.

\subsection{Literature discussion}
Starting from \cite{FHL21}, which first established well-posedness of rough SDEs, there has been a growing literature on the subject \cite{BCN, BFS, FLZ, BFS1}, including the McKean--Vlasov setting \cite{FHL25}. The work \cite{BCN} studies distributional/density properties of the solutions constructed in \cite{FHL21}, and \cite{BFS, BFS1} explore the associated Kolmogorov and Fokker--Planck equations and Feynman--Kac representations. All of these results require coefficients that are at least Lipschitz continuous, and solutions are mostly understood in the pathwise strong sense.

Rough path theory also applies in infinite dimensions, but early applications to rough PDEs were limited to mild formulations \cite{GT, DGT}, Feynman--Kac formulas, or transformations that absorb the rough dependence into the flow of characteristics \cite{CF}, with proof strategies far from classical PDE methods \cite{LSU,KH,KS}. Intrinsic formulations of rough PDEs and a priori estimates with a rough Gronwall lemma \cite{GTT, CR, BG, DFS, HH} later allowed existence and uniqueness results bypassing these transformations; see also \cite{LS} for rough PDEs with transport noise. Despite restricting ourselves to the Young regime, \eqref{main-rPDE} is not covered by any of these theories. In a similar spirit to \cite{HH}, we develop an approach to linear parabolic rough PDEs with possibly discontinuous coefficients, so that the assumptions on the standard part are comparable with the classical PDE theory \cite{KH,KS,LSU}. Going below the threshold $\beta=\frac{1}{2}$ raises highly non-trivial issues, which we leave for future work.
\section{Young parabolic PDEs}\label{sec:PDE}
\subsection{Standing parameters and solution spaces}
Throughout this section we fix $T>0$, $p\in(d,\infty)$ and $\beta\in(\frac12,1)$, and we let $\gamma,\bar\gamma\in C^\beta([0,T])$. Let $\alpha\in(0,1]$ be the H\"older exponent in \cref{ass:main-PDE} below, put
\begin{align}\label{def:eps0}
    \epsilon_0:=\min\Big\{\alpha,\ 1-\frac dp\Big\},
\end{align}
and fix $\epsilon\in[0,\epsilon_0)$ and $q_\epsilon:=\frac2\epsilon\in(\frac2{\epsilon_0},\infty]$, with $q_0=\infty$. The main case is $\epsilon=0$: all results of \cref{sec:mar-weak} use only $\epsilon=0$. A small $\epsilon>0$ is used only in \cref{sec:tightness}, to obtain Krylov estimates for test functions which are merely integrable in time (\cref{lem:Krylov}). Unless stated otherwise, the constants in this section may also depend on $\epsilon$.

\begin{remark}\label{rem:why-besov}
    It would be natural to look for solutions in $\LL^\infty_{2,p}([0,T])$. This is not possible: the heat semigroup does not have maximal regularity in $L^\infty$ in time, i.e.\ $\|\int_\cdot^T P_{r-\cdot}h_rdr\|_{\LL^\infty_{2,p}}\lesssim\|h\|_{\LL^\infty_p}$ fails for $p\in(1,\infty)$. Indeed, let $\psi$ be a Schwartz function whose Fourier transform is supported in the unit ball, put $\varphi_k(x):=\cos(2^kx_1)\psi(x)$, and take $h_r=\varphi_k$ for $r\in[4^{-k},2\cdot4^{-k})$, $k=1,\ldots,n$, and $h_r=0$ otherwise. Then $\sup_r\|h_r\|_{L_p}\lesssim1$, while $\Delta\int_0^TP_rh_rdr=2\sum_{k\le n}(P_{2\cdot4^{-k}}-P_{4^{-k}})\varphi_k$, whose $L_p$-norm is of order $\sqrt n$ by the Littlewood--Paley inequality. This reflects a general phenomenon: by a theorem of Baillon \cite{Baillon}, maximal regularity in $C([0,T];X)$ forces the generator to be bounded unless $X$ contains a copy of $c_0$. Maximal regularity does hold, however, in the Besov space $B^2_{p,\infty}$ (see \cref{lem:semi-group-reg} (iv)), which is the scale we work with.
\end{remark}

\begin{definition}\label{def:Se}
    We denote
    \begin{align}
        &\sol^\epsilon_{1,p}:=L^\infty([0,T];B^{1-\epsilon}_{p,\infty})\cap C^{\frac12}([0,T];B^{-\epsilon}_{p,\infty}),\label{def:Se1}
        \\&\sol^\epsilon_{2,p}:=L^\infty([0,T];B^{2-\epsilon}_{p,\infty})\cap C^{\frac12}([0,T];B^{1-\epsilon}_{p,\infty}),\label{def:Se2}
        \\&\mathbb D^\epsilon:=L^\infty([0,T];B^{-\epsilon}_{p,\infty})+L^{q_\epsilon}([0,T];L_p),\label{def:De}
    \end{align}
    with the norms
    \begin{align*}
        &\|f\|_{\sol^\epsilon_{1,p}}:=\max\{\|f\|_{L^\infty([0,T];B^{1-\epsilon}_{p,\infty})},\|f\|_{C^{\frac12}([0,T];B^{-\epsilon}_{p,\infty})}\},
        \\&\|u\|_{\sol^\epsilon_{2,p}}:=\max\{\|u\|_{L^\infty([0,T];B^{2-\epsilon}_{p,\infty})},\|u\|_{C^{\frac12}([0,T];B^{1-\epsilon}_{p,\infty})}\},
    \end{align*}
    and with the usual norm $\|h\|_{\mathbb D^\epsilon}:=\inf\{\|h_1\|_{L^\infty B^{-\epsilon}_{p,\infty}}+\|h_2\|_{\LL^{q_\epsilon}_p}:h=h_1+h_2\}$ on the sum space. For a subinterval $J\subset[0,T]$ the spaces $\sol^\epsilon_{k,p}(J)$, $\mathbb D^\epsilon(J)$ are defined in the same way.
\end{definition}
Since $L_{k,p}\hookrightarrow B^k_{p,\infty}$, we have $\sol_{k,p}\hookrightarrow\sol^0_{k,p}\hookrightarrow\sol^\epsilon_{k,p}$ for $k=1,2$, and $\LL^\infty_p\hookrightarrow\mathbb D^0\hookrightarrow\mathbb D^\epsilon$. Moreover $\|\nabla u\|_{\sol^\epsilon_{1,p}}\lesssim\|u\|_{\sol^\epsilon_{2,p}}$. Since $1-\epsilon>\frac dp$, $B^{1-\epsilon}_{p,\infty}\hookrightarrow L_p\cap C_b(\Rd)$, so that for $f\in\sol^\epsilon_{1,p}$,
\begin{align}\label{eq:S1e-emb}
    \sup_{t\in[0,T]}\big(\|f_t\|_{L_p}+\|f_t\|_{L_\infty}\big)\lesssim\|f\|_{\sol^\epsilon_{1,p}}.
\end{align}

We consider the following equation
\begin{align}
    \label{main-PDE}
   d_t u+(\frac{1}{2}a^{ij}\partial^2_{ij}u+b\cdot \nabla u)dt+(g\cdot \nabla u) d \gamma_t=fdt +\rho d\bar\gamma_t\text{ on } [0,T]\times\mR^d,
\end{align}
where $f\in\mathbb D^\epsilon$ and $\rho\in\sol^\epsilon_{1,p}$.
\begin{assumption}
  \label{ass:main-PDE}
  We assume:
  \begin{itemize}
    \item[$(H^b)$]   $b\in \LL_q^\infty([0,T])$ for some $q\in[p,\infty]$;
    \item[$(H^g)$] $g:[0,T]\times\mathbb{R}^d\rightarrow\mathbb{R}^{d}$ and  $g\in \sol_{1,p}$;
    \item[$(H_1^a)$] There exists a constant $K\ge1$ such that $(a^{ij})_{1\leq i,j\leq d}$ is symmetric, uniformly bounded and nondegenerate, namely
\[
K^{-1} |\xi|^2 \leq \langle a (t,x) \xi,\xi\rangle \leq K|\xi|^2 \quad
\forall\,\xi\in\R^d,\,(t,x)\in [0,T] \times \R^d;
\]
	\item[$(H_2^a)$]  $a$ is H\"older continuous in space, uniformly in time: for some $\alpha\in(0,1]$,
\[
|a(t,x)-a(t,y)|\leq K|x-y|^\alpha \quad \forall (t,x,y)\in [0,T]\times \R^{2d}.
\]
  \end{itemize}
\end{assumption}
For later use we note that, if $1/q'=1/p-1/q$, then $q'\in[p,\infty]$ and, if $\nabla v\in B^{1-\epsilon}_{p,\infty}\hookrightarrow L_p\cap L_\infty$, then
 \begin{align*}
     \|\nabla v\|_{L_{q'}}\le\|\nabla v\|_{L_p}^{p/q'}\|\nabla v\|_{L_\infty}^{1-p/q'}\lesssim\|v\|_{B^{2-\epsilon}_{p,\infty}},
 \end{align*}
 so that for every interval $J\subset[0,T]$,
\begin{align}\label{est:bnablau}
    \|b\cdot\nabla v\|_{\LL^\infty_p(J)}\lesssim\|b\|_{\LL^\infty_q}\|v\|_{L^\infty(J;B^{2-\epsilon}_{p,\infty})}.
\end{align}
Moreover, since $\alpha>\epsilon$, multiplication by $a^{ij}(t,\cdot)$ is bounded on $B^{-\epsilon}_{p,\infty}$ by \cref{lem.prodS} (i) below, so that $a^{ij}\partial^2_{ij}u\in L^\infty([0,T];B^{-\epsilon}_{p,\infty})$ for $u\in\sol^\epsilon_{2,p}$.

\begin{definition}\label{def:solYPDE}
Let $\epsilon\in[0,\epsilon_0)$. We say $u$ is an (analytically strong) \emph{$\epsilon$-solution} to \eqref{main-PDE} if $u\in \sol^\epsilon_{2,p}$ and for every $s\le t$, as an identity in $B^{-1-\epsilon}_{p,\infty}$,
\begin{multline}
    \label{id.intsolution-PDE}
   u_t-u_s+ \int_s^t\big(\frac{1}{2}a^{ij}\partial^2_{ij}u_r+b\cdot\nabla u_r\big)dr+\int_s^t(g\cdot\nabla u_r)d\gamma_r
   =\int_s^t\rho_rd\bar\gamma_r+ \int_s^t f_rdr.
\end{multline}
\end{definition}
\begin{remark}\label{rem:eps-sol}
If $\epsilon'<\epsilon$, every $\epsilon'$-solution is an $\epsilon$-solution. Under \cref{ass:main-PDE}, $\epsilon$-solutions with the same terminal value are unique (\cref{thm:solution-YPDE}), so when the data lie in the $\epsilon'$-spaces the $\epsilon$-solution is the $\epsilon'$-solution. Below, ``solution'' means $\epsilon$-solution, with $\epsilon$ the parameter fixed in the statement at hand.
\end{remark}
The $dr$-integrals in \eqref{id.intsolution-PDE} are understood according to the convention for Besov-valued integrals in \cref{sub:not}. The integrals $\int(g\cdot \nabla u_r)d\gamma_r$ and $\int\rho_r d\bar\gamma_r$ are $B^{-\epsilon}_{p,\infty}$-valued Young integrals, i.e.\ the limits in $B^{-\epsilon}_{p,\infty}$ of the Riemann sums
\begin{align*}
    \sum_{[s,t]\in\pi} g(s) \cdot\nabla u(s)\gamma_{s,t}
    \tand
    \sum_{[s,t]\in\pi} \rho(s)\bar\gamma_{s,t}
\end{align*}
respectively as the mesh size of the partition $\pi$ vanishes. They exist by \cref{lem:youngPdg,lem.prodS} below, since $g\cdot\nabla u\in\sol^\epsilon_{1,p}$. When $\epsilon=0$ and $f\in\LL^\infty_p$, all terms in \eqref{id.intsolution-PDE} except $\int a^{ij}\partial^2_{ij}u_rdr$ are $L_p$-valued.

\subsection{Young integration in Banach spaces}
 Let $\ival$ be a connected interval and  $E$ be a Banach space.
\begin{definition}\label{def.youngint}
Let $A: \ival_\leq^2 \rightarrow E$. Let $(s,t)\in\ival^2_\le$ and $\Pi$ be a finite partition of $[s,t]$. If the Riemann sum $\sum_{[u,v]\in \Pi} A_{u,v}$ converges as the mesh size of $\Pi$ vanishes, we denote the limit by $(\I A)_{s,t}$. If $(\I A)_{s,t}$ exists for each $(s,t)\in\ival^2_\le$, we say that $A$ is (Riemann) integrable on $\ival$.

If $A$ has the form $A_{s,t}=G_s(y_t-y_s)$ for some maps $G,y$, then we write $\int_s^t G_rdy_r=(\I A)_{s,t}$, which is the usual Young integral \cite{Young}. If $\int_s^t G_rdy_r$  exists for each  $(s,t)\in\ival^2_\le$,  we say that $G$ is $y$-integrable on $\ival$.
\end{definition}
The existence of $\I A$ as defined in \cref{def.youngint} is usually obtained through the following version of the sewing lemma.
\begin{lemma}\label{lem:sewing}
  Let $(E,|\ \cdot\ |) $ be a Banach space, take $0\leq\iota<T$, let $\gamma_0>\eta_0\ge0$, $\gamma_1>\eta_1\ge0$, $\gamma_1>1$ and $A:(\iota,T]_\leq^2\to E$ be a function. Define $\delta A_{s,u,t}=A_{s,t}- A_{s,u}-A_{u,t}$. Assume that there is a finite constant $C$ such that
  \begin{align*}
    | A_{s,t}|\leq C(t-s)^{\gamma_0}(s-\iota)^{-\eta_0}
    \tand|\delta A_{s,u,t}|\leq C(t-s)^{\gamma_1}(s-\iota)^{-\eta_1}
  \end{align*}
   for all $(s,u,t)\in (\iota,T]^3_\leq$.
  Then   $\I A$ exists on $(\iota,T]^2_\le $ and extends continuously to $[\iota,T]^2_\le$. Furthermore, there are finite constants $K_1,K_2$ which depend only on $\gamma_i,\eta_i$'s such that
  \begin{align}
      &|(\I A)_{s,t}-A_{s,t}|\leq K_1 C(t-s)^{\gamma_1}(s-\iota)^{-\eta_1}, \quad \forall (s,t)\in (\iota,T]^2_\leq;\label{est.young}
   \\& |(\I A )_{s,t}|\leq K_2C\sum_{j=0,1}(t-s)^{\gamma_j-\eta_j}, \quad \forall (s,t)\in [\iota,T]^2_\leq.\label{est.tamedyoung}
  \end{align}
  If $\eta_0=\eta_1=0$ and $A$ is defined on $[\iota,T]^2_\le$ with the above bounds for all $(s,u,t)\in[\iota,T]^3_\le$, then $\I A$ exists on $[\iota,T]^2_\le$, i.e.\ the Riemann sums also converge for $s=\iota$.
\end{lemma}
\begin{proof}
  Fix $s_0\in(\iota,T]$. Since $(a-\iota)^{-\eta_j}\le(s_0-\iota)^{-\eta_j}$ for $a\ge s_0$, the assumptions give, for all $(a,u,b)\in[s_0,T]^3_\le$,
  \begin{align*}
    |A_{a,b}|\le C(s_0-\iota)^{-\eta_0}(b-a)^{\gamma_0}\tand|\delta A_{a,u,b}|\le C(s_0-\iota)^{-\eta_1}(b-a)^{\gamma_1}.
  \end{align*}
  In particular $A_{a,a}=0$, and the restriction of $A$ to $[s_0,T]^2_\le$ satisfies the hypotheses of the sewing lemma \cite[Lemma 4.2]{FH} with $\alpha=\min\{\gamma_0,1\}$ and $\beta=\gamma_1>1$. By that lemma and its proof, the Riemann sums of $A$ converge on $[s_0,T]^2_\le$ as the mesh size vanishes, and $|(\I A)_{s,t}-A_{s,t}|\le2^{\gamma_1}\zeta(\gamma_1)C(s_0-\iota)^{-\eta_1}(t-s)^{\gamma_1}$ for $(s,t)\in[s_0,T]^2_\le$, where $\zeta$ is the Riemann zeta function. Since $s_0$ is arbitrary, $\I A$ exists on $(\iota,T]^2_\le$, and taking $s_0=s$ gives \eqref{est.young} with $K_1=2^{\gamma_1}\zeta(\gamma_1)$.
  Using the assumptions, we have
  \begin{align*}
    |(\I A)_{s,t}|
    \les C(t-s)^{\gamma_0}(s-\iota)^{-\eta_0}+C(t-s)^{\gamma_1}(s-\iota)^{-\eta_1}.
  \end{align*}
  Applying the taming lemma \cite[Lemma 3.4]{LL}, we obtain \eqref{est.tamedyoung} for $(s,t)\in(\iota,T]^2_\le$ (in \cite{LL} the exponents are assumed to lie in $[0,1]$, but the same argument carries over verbatim to $\gamma_j>\eta_j\ge0$). This also shows the existence of $(\I A)_{\iota,t}$ as $\lim_{s\downarrow \iota}(\I A)_{s,t}$ for each $t\in(\iota,T]$. The estimate for $(\I A)_{s,t}$ (as in \eqref{est.tamedyoung}) is now extended trivially to all $(s,t)\in[\iota,T]^2_\le$. If $\eta_0=\eta_1=0$ and the bounds hold on $[\iota,T]^3_\le$, the first step applies verbatim with $s_0=\iota$, so the Riemann sums converge on $[\iota,T]^2_\le$.
\end{proof}

\begin{remark}[Riemann sums at the singular endpoint]\label{rem:sewing-endpoint}
    In \cref{lem:sewing}, $(\I A)_{\iota,t}$ is defined as $\lim_{s\downarrow\iota}(\I A)_{s,t}$. If $A$ is also defined at $s=\iota$ and $\eta_0=0$, $|A_{\iota,v}|\le C(v-\iota)^{\gamma_0}$, then the Riemann sums over partitions of $[\iota,t]$ converge to $(\I A)_{\iota,t}$; this follows by splitting the sum into the first interval, the intervals with $v-u>u-\iota$ (whose right endpoints are geometrically spaced) and the remaining ones, for which \eqref{est.young} applies. If $\eta_0>0$ this may fail. For instance, in \cref{lem:youngPdg} (ii) the Riemann sums $\sum_\Pi P^a_{\iota,u}f_u\gamma_{u,v}$ converge in $\bar B^{1-\epsilon}$ (where $\eta_0=0$) but not, in general, in $\bar B^{2-\epsilon}$ (where $\eta_0=\frac12$); this is why the $\bar B^{2-\epsilon}$-valued integral at $s=\iota$ is defined as a limit as $s\downarrow\iota$.
\end{remark}

\subsection{An intermediate case: space-independent diffusion coefficients}\label{sub:intermediate}
In this subsection the diffusion coefficient does not depend on $x$, but it may depend on an additional parameter $z\in\Rd$; this is needed for the freezing argument in \cref{sub:general}. Thus, let $a:[0,T]\times\Rd\to\R^{d\times d}$, $(r,z)\mapsto a(r,z)$, be measurable and such that $a(\cdot,z)$ satisfies $(H^a_1)$ for every $z\in\Rd$. For $0\le s\le t\le T$ we put $P^{a(z)}_{s,t}:=P^Q_{s,t}$ with $Q_{s,t}=\int_s^ta(r,z)dr$, i.e.\ $P^{a(z)}_{s,t}$ is the convolution with the centred Gaussian density with covariance $\int_s^ta(r,z)dr$. Then $P^{a(z)}_{\iota,s}P^{a(z)}_{s,t}=P^{a(z)}_{\iota,t}$ for $\iota\le s\le t$, $P^{a(z)}_{s,t}$ is a Fourier multiplier and commutes with derivatives and with the Littlewood--Paley blocks $\Delta_j$, and, for every tempered distribution $\varphi$, $$\partial_sP^{a(z)}_{s,t}\varphi=-\frac12a^{ij}(s,z)\partial^2_{ij}P^{a(z)}_{s,t}\varphi\quad \text{   and   } \quad  \partial_tP^{a(z)}_{s,t}\varphi=\frac12a^{ij}(t,z)\partial^2_{ij}P^{a(z)}_{s,t}\varphi$$ for a.e.\ $s<t$, as identities on the Fourier side. When $a$ does not depend on $z$ we simply write $P^a_{s,t}$.

To run the freezing argument, we need an extension of the spaces above that includes the additional parameter. For $\nu\in\R$ and $p\in(1,\infty)$, let $\bar B^\nu$ denote the space of families $f=(f^z)_{z\in\Rd}$ of tempered distributions such that $z\mapsto\Delta_jf^z$ is a measurable map into $L_p(\Rd)$ for every $j$ and
\begin{equation*}
    \|f\|_{\bar B^\nu}:=\sup_{j\ge0}2^{j\nu}\Big(\int_\Rd \| \Delta_jf^z\|_{L_p}^p dz \Big)^{\frac1p}<\infty ;
\end{equation*}
that is, $\bar B^\nu=B^\nu_{p,\infty}(\R^d_x;L_p(\R^d_z))$. Note that the supremum over $j$ is taken outside the $z$-integral. We also write $\Lb_p:=L_p(\Rd\times\Rd)$, identified with families $(f^z)$ with $f^z\in L_p$, so that $\Lb_p\hookrightarrow\bar B^0$. As in \cref{sub:not}, a function $v:[0,T]\to\bar B^\nu$ is called measurable if it is strongly measurable as a map into the separable space $\bar B^{\nu-1}_{p,p}:=B^{\nu-1}_{p,p}(\R^d_x;L_p(\R^d_z))\supset\bar B^\nu$, and time integrals of such functions are Bochner integrals in $\bar B^{\nu-1}_{p,p}$. The spaces $\solb^\epsilon_{k,p}$ and $\bar{\mathbb D}^\epsilon$ are defined as in \cref{def:Se}, with $B^\nu_{p,\infty}$ and $L_p$ replaced by $\bar B^\nu$ and $\Lb_p$:
\begin{align}
  \label{def:sol2p} &\solb^\epsilon_{2,p}:=L^\infty([0,T];\bar B^{2-\epsilon})\cap C^{\frac12}([0,T];\bar B^{1-\epsilon}),
  \\& \solb^\epsilon_{1,p}:=L^\infty([0,T];\bar B^{1-\epsilon})\cap C^{\frac12}([0,T];\bar B^{-\epsilon}),\label{def:sol1p}
  \\ &\bar{\mathbb D}^\epsilon:=L^\infty([0,T];\bar B^{-\epsilon})+L^{q_\epsilon}([0,T];\Lb_p). 
\end{align}
Operators acting in the variable $x$ are extended $z$-wise: $(P^a_{s,t}f)^z:=P^{a(z)}_{s,t}f^z$, $(a\partial^2_{ij}f)^z:=a(\cdot,z)\partial^2_{ij}f^z$ and $(gf)^z:=gf^z$ for a function $g$ of $(t,x)$. As in \eqref{eq:S1e-emb}, $\sup_t(\int\|f^z_t\|^p_{L_\infty}dz)^{1/p}\lesssim\|f\|_{\solb^\epsilon_{1,p}}$; this follows from Bernstein's inequality $\|\Delta_jf^z\|_{L_\infty}\lesssim2^{jd/p}\|\Delta_jf^z\|_{L_p}$ and $1-\epsilon>\frac dp$.

We list the following estimates for the semigroup $(P^{a}_{s,t})$ that will be used heavily later. The proofs are given in \cref{sec:app1}.
\begin{lemma}[Heat kernel bounds]
    \label{lem:semi-group-reg}
    Let $p\in(1,\infty)$ and $\nu\in\R$. There is a constant $N$ depending only on $d,p,K,T,\nu$ and on the parameters $\mu,q'$ below such that for every $z\in\Rd$, all $0\le s<t\le T$ and, in (iii), all $\iota\in[0,s)$, writing $P_{s,t}=P^{a(z)}_{s,t}$ and $B^\nu=B^\nu_{p,\infty}$:
  \begin{enumerate}[(i)]
      \item for $\mu\in[0,3]$, \quad $\|P_{s,t}f\|_{B^{\nu+\mu}}\le N(t-s)^{-\frac\mu2}\|f\|_{B^{\nu}}$;\hfill\refstepcounter{equation}\textup{(\theequation)}\label{est:a-1}
      \item for $\mu\in[0,2]$, \quad $\|(P_{s,t}-I)f\|_{B^{\nu}}\le N(t-s)^{\frac\mu2}\|f\|_{B^{\nu+\mu}}$;\hfill\refstepcounter{equation}\textup{(\theequation)}\label{est:pt-3}
      \item for $\mu\in[0,3]$, \quad $\|P_{\iota,t}f-P_{\iota,s}f\|_{B^{\nu}}\le N(t-s)^{\frac12}(s-\iota)^{-\frac\mu2}\|f\|_{B^{\nu+1-\mu}}$;\hfill\refstepcounter{equation}\textup{(\theequation)}\label{est:pt-4}
      \item (maximal regularity) for $\mu\in[0,2]$ and $h\in L^\infty([s,t];B^\nu)$,
      \begin{align}\label{est:inte-P-h}
          \Big\|\int_s^tP_{s,r}h_rdr\Big\|_{B^{\nu+\mu}}\le N(t-s)^{1-\frac\mu2}\sup_{r\in[s,t]}\|h_r\|_{B^\nu};
      \end{align}
      \item for $q'\in[1,\infty)$, $\mu\in[0,2-\frac2{q'}]$ and $h\in L^{q'}([s,t];L_p)$,
      \begin{align}\label{est:heat-LqLp}
          \Big\|\int_s^tP_{s,r}h_rdr\Big\|_{B^{\mu}}\le N(t-s)^{1-\frac\mu2-\frac1{q'}}\|h\|_{L^{q'}([s,t];L_p)}.
      \end{align}
  \end{enumerate}
  The same estimates hold with $B^\nu_{p,\infty}$, $L_p$ replaced by $\bar B^\nu$, $\Lb_p$, with the same constants.
\end{lemma}
Note that (iii) follows from (i) and (ii), since $P_{\iota,t}-P_{\iota,s}=(P_{s,t}-I)P_{\iota,s}$. In addition,  combining (iv) and (v), we have for $h\in\mathbb D^\epsilon$ (or $\bar{\mathbb D}^\epsilon$) and $\mu\in[\epsilon,2]$,
\begin{align}\label{est:inte-P-h-D}
    \Big\|\int_s^tP_{s,r}h_rdr\Big\|_{B^{\mu-\epsilon}_{p,\infty}}\le N(t-s)^{1-\frac{\mu}2}\|h\|_{\mathbb D^\epsilon([s,t])}.
\end{align}
Indeed, if $h=h_1+h_2$ with $h_1\in L^\infty B^{-\epsilon}_{p,\infty}$ and $h_2\in\LL^{q_\epsilon}_p$, we apply (iv) with $\nu=-\epsilon$ to $h_1$, and, if $\epsilon>0$, (v) with $\mu-\epsilon\in[0,2-\epsilon]$ in place of $\mu$ and $q'=q_\epsilon=\frac2\epsilon$ to $h_2$; in both cases the exponent of $t-s$ is $1-\frac\mu2$. (If $\epsilon=0$, then $\mathbb D^0=L^\infty([0,T];B^0_{p,\infty})$ since $L_p\hookrightarrow B^0_{p,\infty}$, and only (iv) is needed.)
\begin{remark}\label{rem.extheat}
    All estimates of \cref{lem:semi-group-reg} are proved block by block: for each Littlewood--Paley block, the bound holds uniformly in $z$, and one then takes the $L_p$-norm in $z$ and the supremum over $j$. This is why the supremum over $j$ is placed outside the $z$-integral in the definition of $\bar B^\nu$.
\end{remark}

\begin{lemma}\label{lem.prodS}
  Let $p\in(d,\infty)$ and write $B^\nu=B^\nu_{p,\infty}$.
  \begin{enumerate}[(i)]
      \item (Pointwise multipliers) Let $\nu\in(-1,1)$ and $\varrho\in(|\nu|,1]$. Then $\|mF\|_{B^\nu}\le N\|m\|_{C^\varrho}\|F\|_{B^\nu}$. More generally, if $m=(m^z)_z$ is a family of functions such that $(z,x)\mapsto m^z(x)$ is measurable, then $\|(m^zF^z)_z\|_{\bar B^\nu}\le N\sup_z\|m^z\|_{C^\varrho}\|F\|_{\bar B^\nu}$.
      \item (Products) Let $\nu\in(\frac dp,1]$. Then for $g\in L_{1,p}$ and $F\in\bar B^\nu$, $\|(gF^z)_z\|_{\bar B^\nu}\le N\|g\|_{L_{1,p}}\|F\|_{\bar B^\nu}$.
      \item For $h\in\solb^\epsilon_{2,p}$, $g\in\sol_{1,p}$ and $f\in\solb^\epsilon_{1,p}$,
    \begin{gather}
     \|\nabla h\|_{\solb^\epsilon_{1,p}}\lesssim\|h\|_{\solb^\epsilon_{2,p}},\label{est:nabla-h}
      \\
      \|gf\|_{\solb^\epsilon_{1,p}}\lesssim \|g\|_{\sol_{1,p}}\|f\|_{\solb^\epsilon_{1,p}}.\label{est:fg}
    \end{gather}
      \item (Commutator) Let $\nu\in(-\alpha,0]$, $a\in C^\alpha(\Rd)$, and let $\varphi_k(x)=k^d\varphi(kx)$, $k\ge1$, for a Schwartz function $\varphi$ with $\int\varphi=1$. Then for $G\in B^\nu$,
      \begin{align}\label{est:commutator}
          \|a(\varphi_k\ast G)-\varphi_k\ast(aG)\|_{L_p}\le N k^{-(\alpha+\nu)}\log(1+k)\|a\|_{C^\alpha}\|G\|_{B^\nu}.
      \end{align}
  \end{enumerate}
  The same estimates hold without the parameter $z$, i.e.\ with $\bar B^\nu,\solb^\epsilon_{k,p}$ replaced by $B^\nu,\sol^\epsilon_{k,p}$.
\end{lemma}
\begin{proof}
    (i), (ii) and (iv) are standard consequences of Bony's paraproduct decomposition; we give the proofs in \cref{sec:app1}. For (iii), the first estimate follows from the boundedness of $\partial_k:\bar B^\nu\to\bar B^{\nu-1}$. For the second, fix $s\le t$. Note that $\|g_s\|_{L_{1,p}}\le\|g\|_{\LL^\infty_{1,p}}$ for every $s$, not only for a.e.\ $s$, since $g\in C([0,T];L_p)$ and the $L_{1,p}$-norm is lower semicontinuous with respect to convergence in $L_p$. By (ii) with $\nu=1-\epsilon$, $\|g_tf_t\|_{\bar B^{1-\epsilon}}\lesssim\|g_t\|_{L_{1,p}}\|f_t\|_{\bar B^{1-\epsilon}}$. Moreover $g_tf_t-g_sf_s=(g_t-g_s)f_t+g_s(f_t-f_s)$. Using $\Lb_p\hookrightarrow\bar B^0\hookrightarrow\bar B^{-\epsilon}$ and the $L_\infty$-bound below \eqref{def:sol1p},
    \begin{align*}
        \|(g_t-g_s)f_t\|_{\bar B^{-\epsilon}}\lesssim\|g_t-g_s\|_{L_p}\Big(\int\|f^z_t\|^p_{L_\infty}dz\Big)^{\frac1p}\lesssim\|g\|_{\sol_{1,p}}\|f\|_{\solb^\epsilon_{1,p}}|t-s|^{\frac12}.
    \end{align*}
    Since $g_s\in L_{1,p}\hookrightarrow C^{1-d/p}$ and $1-\frac dp>\epsilon$, (i) with $\nu=-\epsilon$ and $\varrho=1-\frac dp$ gives $\|g_s(f_t-f_s)\|_{\bar B^{-\epsilon}}\lesssim\|g\|_{\sol_{1,p}}\|f_t-f_s\|_{\bar B^{-\epsilon}}$. This proves \eqref{est:fg}.
\end{proof}

\begin{lemma}\label{lem:youngPdg}
    Let $a$ be as at the beginning of \cref{sub:intermediate}, $f\in\solb^\epsilon_{1,p}$ and $\gamma\in C^{\beta}([0,T])$. Then the following statements hold.
\begin{enumerate}[(i)]
    \item The Young integral $(s,t)\mapsto \int_s^t  f_r d \gamma_r\in \bar B^{-\epsilon}$ is well-defined on $[0,T]^2_\le$ and
    \begin{align}\label{est:young-Lp}\begin{split}
        \Big\|\int_s^tf_rd\gamma_r-f_s\gamma_{s,t}\Big\|_{\bar B^{-\epsilon}}\le\Gamma\|\gamma\|_{C^\beta}\|f\|_{\solb^\epsilon_{1,p}}|t-s|^{\beta+\frac12},
        \\
        \Big\|\int_s^tf_rd\gamma_r\Big\|_{\bar B^{-\epsilon}}\le\Gamma\|\gamma\|_{C^\beta}\|f\|_{\solb^\epsilon_{1,p}}|t-s|^{\beta}.
    \end{split}\end{align}
    Moreover, for every bounded linear operator $H$ from $\bar B^{-\epsilon}$ to a Banach space $E$, $Hf$ is $\gamma$-integrable on $[0,T]$ and
    \begin{align}\label{id.Hintf}
         H\int_s^t f_rd\gamma_r=
        \int_s^t H f_rd\gamma_r.
    \end{align}
    \item For each $\iota\in[0,T]$, the Young integral $(s,t)\mapsto \int_s^t P^a_{\iota,r} f_r d \gamma_r$ is well-defined in $\bar B^{-\epsilon}$ on $[\iota,T]^2_\le$, as the limit in $\bar B^{-\epsilon}$ of the Riemann sums $\sum_{[v,w]\in\Pi}P^a_{\iota,v}f_v\gamma_{v,w}$. For $s>\iota$ it takes values in $\bar B^{2-\epsilon}$, and for every $(s,t)\in[0,T]^2_\le$,
    \begin{align}
   \Big\|\int_s^tP^a_{s,r}f_rd\gamma_r\Big\|_{\bar B^{1-\epsilon}}&\leq  \Gamma\|\gamma\|_{C^\beta}\|f\|_{\solb^\epsilon_{1,p}}
   |t-s|^{\beta},
   \label{est:bar-1p}
   \\
   \Big\|\int_s^tP^a_{s,r}f_rd\gamma_r\Big\|_{\bar B^{2-\epsilon}}&\leq\Gamma\|\gamma\|_{C^\beta}\|f\|_{\solb^\epsilon_{1,p}}
   |t-s|^{\beta-\frac{1}{2}}.\label{est:bar-2p}
    \end{align}
\end{enumerate}
Here $\Gamma$ is a finite constant depending only on $d,p,\beta,K,T,\epsilon$. The same statements hold for $f\in\sol^\epsilon_{1,p}$, with $\bar B$ replaced by $B_{p,\infty}$ and, in (ii), $P^a$ replaced by $P^{a(z)}$ for a fixed $z\in\Rd$.
\end{lemma}
\begin{proof}
 (i). Let $A_{s,t}=f_s\gamma_{s,t}$. Then $\delta A_{s,u,t}=-(f_u-f_s)\gamma_{u,t}$, so that $\|\delta A_{s,u,t}\|_{\bar B^{-\epsilon}}\le\|\gamma\|_{C^\beta}\|f\|_{\solb^\epsilon_{1,p}}(t-s)^{\beta+\frac12}$, and $\beta+\frac12>1$. \cref{lem:sewing} (with $\iota=0$, $\eta_0=\eta_1=0$, and the bounds holding on $[0,T]^3_\le$) gives the existence of $\int f d\gamma$ on $[0,T]^2_\le$ and \eqref{est:young-Lp}, using also $\|A_{s,t}\|_{\bar B^{-\epsilon}}\lesssim\|\gamma\|_{C^\beta}\|f\|_{\solb^\epsilon_{1,p}}(t-s)^\beta$. Since $\|\delta(HA)_{s,u,t}\|_E\le\|H\|\|\delta A_{s,u,t}\|_{\bar B^{-\epsilon}}$, the same argument applies to $HA$, and \eqref{id.Hintf} follows because $H$ is continuous and commutes with Riemann sums.

 (ii). Fix $\iota\in[0,T]$, write $P_{s,t}=P^a_{s,t}$ and let $A_{s,t}:=P_{\iota,s}f_s\gamma_{s,t}$. For $(s,u,t)\in(\iota,T]^3_\le$,
    \begin{align*}
        \delta A_{s,u,t}=-(P_{\iota,u}-P_{\iota,s})f_u\gamma_{u,t}-P_{\iota,s}(f_u-f_s)\gamma_{u,t}.
    \end{align*}
 We first work in $\bar B^{-\epsilon}$. Since $P_{\iota,u}-P_{\iota,s}=P_{\iota,s}(P_{s,u}-I)$, \cref{lem:semi-group-reg} (i) with $\mu=0$ and (ii) with $\mu=1$ give $\|\delta A_{s,u,t}\|_{\bar B^{-\epsilon}}\lesssim\|\gamma\|_{C^\beta}\|f\|_{\solb^\epsilon_{1,p}}(t-s)^{\beta+\frac12}$, without any singularity at $\iota$. The same holds for $A$ itself: $\|A_{s,t}\|_{\bar B^{-\epsilon}}\lesssim\|\gamma\|_{C^\beta}\|f\|_{\solb^\epsilon_{1,p}}(t-s)^{\beta}$. Hence, by \cref{lem:sewing} with $\eta_0=\eta_1=0$, applied on the closed interval $[\iota,T]$, the Riemann sums converge in $\bar B^{-\epsilon}$ on $[\iota,T]^2_\le$, and \eqref{est.tamedyoung} gives $\|\int_\iota^sP_{\iota,r}f_rd\gamma_r\|_{\bar B^{-\epsilon}}\lesssim(s-\iota)^\beta$.

 Next, let $\nu\in\{1-\epsilon,2-\epsilon\}$. By \cref{lem:semi-group-reg} (i) with $\mu=\nu-1+\epsilon$, by (iii) with $\mu=\nu+\epsilon$, and by (i) with $\mu=\nu+\epsilon$ (applied to $f_u-f_s\in\bar B^{-\epsilon}$),
    \begin{align*}
        \|A_{s,t}\|_{\bar B^{\nu}}&\lesssim\|\gamma\|_{C^\beta}\|f\|_{\solb^\epsilon_{1,p}}(t-s)^{\beta}(s-\iota)^{-\frac{\nu-1+\epsilon}2},
        \\
        \|\delta A_{s,u,t}\|_{\bar B^{\nu}}&\lesssim\|\gamma\|_{C^\beta}\|f\|_{\solb^\epsilon_{1,p}}(t-s)^{\beta+\frac12}(s-\iota)^{-\frac{\nu+\epsilon}2}.
    \end{align*}
 We apply \cref{lem:sewing} in $E=\bar B^{\nu}$ with $\gamma_0=\beta$, $\eta_0=\frac{\nu-1+\epsilon}2\le\frac12$, $\gamma_1=\beta+\frac12$, $\eta_1=\frac{\nu+\epsilon}2\le1$; its conditions hold because $\beta>\frac12$. Since $\gamma_0-\eta_0=\gamma_1-\eta_1=\beta-\frac{\nu-1+\epsilon}2$, \eqref{est.tamedyoung} with $\iota=s$ gives \eqref{est:bar-1p} and \eqref{est:bar-2p} for $\nu=1-\epsilon$ and $\nu=2-\epsilon$ respectively. For $s>\iota$, the $\bar B^{\nu}$-valued integral $\int_s^tP_{\iota,r}f_rd\gamma_r$ coincides with the $\bar B^{-\epsilon}$-valued one, both being limits of the same Riemann sums and $\bar B^\nu\hookrightarrow\bar B^{-\epsilon}$. For $s=\iota$, \cref{lem:sewing} defines the $\bar B^\nu$-valued integral as the limit of $\int_{s}^tP_{\iota,r}f_rd\gamma_r$ in $\bar B^\nu$ as $s\downarrow\iota$; since, in $\bar B^{-\epsilon}$, $\int_\iota^tP_{\iota,r}f_rd\gamma_r=\int_\iota^sP_{\iota,r}f_rd\gamma_r+\int_s^tP_{\iota,r}f_rd\gamma_r$ and the first term tends to $0$ by the bound above, this limit is again the $\bar B^{-\epsilon}$-valued integral.
  \end{proof}

\begin{lemma}\label{lem.duhamel}
    Let $a$ be as at the beginning of \cref{sub:intermediate}, let $f \in\solb^\epsilon_{1,p}$, $h\in \bar{\mathbb D}^\epsilon$, and let $u\in\solb^\epsilon_{2,p}$ satisfy, for every $0\le s\le t\le T$, in $\bar B^{-1-\epsilon}$,
    \begin{align}
        \label{eq:Yheat}
        u_t-u_s+\int_s^t\frac12a^{ij}(r)\partial^2_{ij}u_rdr=\int_s^tf_rd\gamma_r+\int_s^th_rdr.
    \end{align}
    Then for every $s\le t$, in $\bar B^{-\epsilon}$,
    \begin{equation}\label{id.mildheat}
        P^a_{s,t}u_t-u_s=\int_s^tP^a_{s,r}f_rd\gamma_r+\int_s^tP^a_{s,r}h_rdr.
    \end{equation}
    The same holds, with the same proof, if $f d\gamma$ is replaced by a finite sum $\sum_if^id\gamma^i$ with $f^i\in\solb^\epsilon_{1,p}$ and $\gamma^i\in C^\beta([0,T])$.
\end{lemma}
\begin{proof}
    Fix $\bar s<\bar t$, write $P_r:=P^a_{\bar s,r}$ for $r\ge\bar s$, and let $\Pi$ be a partition of $[\bar s,\bar t]$. For $[s,t]\in\Pi$ we have $$P_tu_t-P_su_s=P_s(u_t-u_s)+(P_t-P_s)u_t \text{  and  } (P_t-P_s)u_t=\frac12\int_s^ta^{ij}(r)\partial^2_{ij}P_ru_tdr.$$ Together with \eqref{eq:Yheat} and the fact that $a$ does not depend on $x$, we obtain
    \begin{align*}
        P_tu_t-P_su_s=A^1_{s,t}+P_s\int_s^tf_rd\gamma_r+P_s\int_s^th_rdr,
    \end{align*}
    where
    \begin{align*}
        A^1_{s,t}:=\frac12\int_s^ta^{ij}(r)P_r\partial^2_{ij}(u_t-u_r)dr+\frac12\int_s^ta^{ij}(r)(P_r-P_s)\partial^2_{ij}u_rdr.
    \end{align*}
    By \cref{lem:semi-group-reg} (i) with $\mu=0$, $\|P_r\partial^2_{ij}(u_t-u_r)\|_{\bar B^{-1-\epsilon}}\lesssim\|u_t-u_r\|_{\bar B^{1-\epsilon}}\lesssim(t-r)^{\frac12}\|u\|_{\solb^\epsilon_{2,p}}$. Since $P_r-P_s=(P^a_{s,r}-I)P_s$, (ii) with $\mu=1$ gives $\|(P_r-P_s)\partial^2_{ij}u_r\|_{\bar B^{-1-\epsilon}}\lesssim(r-s)^{\frac12}\|u_r\|_{\bar B^{2-\epsilon}}$. Hence
    \begin{align*}
        \sum_{[s,t]\in\Pi}\|A^1_{s,t}\|_{\bar B^{-1-\epsilon}}\lesssim|\Pi|^{\frac12}\|u\|_{\solb^\epsilon_{2,p}}\to0
    \end{align*}
    as $|\Pi|\to0$. Next, by \eqref{est:young-Lp}, $P_s\int_s^tf_rd\gamma_r=P_sf_s\gamma_{s,t}+R_{s,t}$ with $\|R_{s,t}\|_{\bar B^{-\epsilon}}\lesssim(t-s)^{\beta+\frac12}$, so that $\sum_\Pi R_{s,t}\to0$, while $\sum_\Pi P_sf_s\gamma_{s,t}\to\int_{\bar s}^{\bar t}P^a_{\bar s,r}f_rd\gamma_r$ in $\bar B^{-\epsilon}$ by \cref{lem:youngPdg} (ii). Finally, since $\bar{\mathbb D}^\epsilon\hookrightarrow L^1([0,T];\bar B^{-\epsilon})$, \cref{lem:semi-group-reg} (ii) with $\mu=1$ gives
    \begin{align*}
        \Big\|\sum_{[s,t]\in\Pi}P_s\int_s^th_rdr-\int_{\bar s}^{\bar t}P_rh_rdr\Big\|_{\bar B^{-1-\epsilon}}&\le\sum_{[s,t]\in\Pi}\int_s^t\|(P_r-P_s)h_r\|_{\bar B^{-1-\epsilon}}dr\\&\lesssim|\Pi|^{\frac12}\|h\|_{L^1([0,T];\bar B^{-\epsilon})}\to0.
    \end{align*}
    Since $\sum_\Pi(P_tu_t-P_su_s)=P^a_{\bar s,\bar t}u_{\bar t}-u_{\bar s}$, letting $|\Pi|\to0$ yields \eqref{id.mildheat} in $\bar B^{-1-\epsilon}$, and hence in $\bar B^{-\epsilon}$ because both sides belong to $\bar B^{-\epsilon}$.
\end{proof}

\begin{proposition}
  \label{thm:YPDE-add}
Let $a$ be as at the beginning of \cref{sub:intermediate}, let $g \in (\sol_{1,p})^{d}$, $\rho\in \solb^\epsilon_{1,p}$, $h\in \bar{\mathbb D}^\epsilon$ and $\phi\in \bar B^{2-\epsilon}$.
Let $u\in \solb^\epsilon_{2,p}$ satisfy $u_T=\phi$ and, for every $0\le s\le t\le T$, in $\bar B^{-1-\epsilon}$,
    \begin{align}
        u_t-u_s+\int_s^t\frac12 a^{ij}(r)\partial^2_{ij} u_rdr+\int_s^tg\cdot \nabla u_r  d\gamma_r=\int_s^t\rho_r d\bar\gamma_r+\int_s^th_rdr.
        \label{eq:YPDE}
    \end{align}
   Then there exists a finite constant $\mathscr{N}$  depending on $K, T, p,d, \beta,\epsilon $ and on upper bounds for $\|g\|_{\sol_{1,p}}, \|\gamma\|_{C^\beta}$ such that
    \begin{gather}  \label{est:apr-add}
    \|u\|_{\solb^\epsilon_{2,p}}
        \leq \sN\big(\|\bar\gamma\|_{C^\beta}\|\rho\|_{\solb^\epsilon_{1,p} }+\|h\|_{\bar{\mathbb D}^\epsilon}+\|\phi\|_{\bar B^{2-\epsilon}}\big).
    \end{gather}
    Moreover, there are $T_0\in(0,1]$ and $N_0$, depending only on $K,p,d,\beta,\epsilon$ and on upper bounds for $\|g\|_{\sol_{1,p}}, \|\gamma\|_{C^\beta}$, such that if $T\le T_0$, then
    \begin{gather}  \label{est:apr-add-small}
    \|u\|_{\solb^\epsilon_{2,p}}
        \leq N_0\big(T^{\beta-\frac12}\|\bar\gamma\|_{C^\beta}\|\rho\|_{\solb^\epsilon_{1,p} }+\|h\|_{\bar{\mathbb D}^\epsilon}+\|\phi\|_{\bar B^{2-\epsilon}}\big).
    \end{gather}
    The same holds if $\rho d\bar\gamma$ is replaced by a finite sum $\sum_i\rho^id\bar\gamma^i$, with $\sum_i\|\bar\gamma^i\|_{C^\beta}\|\rho^i\|_{\solb^\epsilon_{1,p}}$ on the right-hand sides.
\end{proposition}
\begin{proof}
 Write $P_{s,t}=P^a_{s,t}$ and let $F:=g\cdot\nabla u$. By \cref{lem.prodS}, $\|F\|_{\solb^\epsilon_{1,p}}\lesssim\|g\|_{\sol_{1,p}}\|u\|_{\solb^\epsilon_{2,p}}$. By \cref{lem.duhamel} (with two Young terms), for every $s\le t$,
    \begin{align}\label{eq:Pt-time}
        P_{s,t}u_t-u_s=-\int_s^tP_{s,r}F_rd\gamma_r+\int_s^tP_{s,r}\rho_rd\bar\gamma_r+\int_s^tP_{s,r}h_rdr.
    \end{align}
 Put $$M:=\|\gamma\|_{C^\beta}\|F\|_{\solb^\epsilon_{1,p}}+\|\bar\gamma\|_{C^\beta}\|\rho\|_{\solb^\epsilon_{1,p}} \qquad \text{and }\qquad H:=\|h\|_{\bar{\mathbb D}^\epsilon},$$ and assume first $T\le1$. We estimate the two parts of the $\solb^\epsilon_{2,p}$-norm, using \eqref{eq:Pt-time}, \cref{lem:youngPdg} (ii) and \cref{lem:semi-group-reg}.

 First take $t=T$ in \eqref{eq:Pt-time} then use \eqref{est:a-1} with $\mu=0$, \eqref{est:bar-2p}, and \eqref{est:inte-P-h-D} with $\mu=2$, we get
 \begin{align}\label{est:a}
     \sup_{s\le T}\|u_s\|_{\bar B^{2-\epsilon}}\lesssim\|\phi\|_{\bar B^{2-\epsilon}}+T^{\beta-\frac12}M+H.
 \end{align}
Secondly for $s<t$, write $$u_t-u_s=(I-P_{s,t})u_t+(P_{s,t}u_t-u_s).$$ By \eqref{est:pt-3} with $\nu=1-\epsilon$, $\mu=1$, $\|(I-P_{s,t})u_t\|_{\bar B^{1-\epsilon}}\lesssim(t-s)^{\frac12}\|u_t\|_{\bar B^{2-\epsilon}}$. By \eqref{eq:Pt-time} on $[s,t]$, \eqref{est:bar-1p}, and \eqref{est:inte-P-h-D} with $\mu=1$,
 \begin{align}\label{est:b}
     \|P_{s,t}u_t-u_s\|_{\bar B^{1-\epsilon}}\lesssim(t-s)^\beta M+(t-s)^{\frac12}H\le(t-s)^{\frac12}\big(T^{\beta-\frac12}M+H\big).
 \end{align}
 Since also $\sup_t\|u_t\|_{\bar B^{1-\epsilon}}\le\sup_t\|u_t\|_{\bar B^{2-\epsilon}}$, combining \eqref{est:a} and \eqref{est:b} we get
 \begin{align}\label{est:usol2}\begin{split}
     \|u\|_{\solb^\epsilon_{2,p}}&\le N_1\big(\|\phi\|_{\bar B^{2-\epsilon}}+T^{\beta-\frac12}M+H\big)
     \\&\le N_1\big(\|\phi\|_{\bar B^{2-\epsilon}}+T^{\beta-\frac12}(N\|\gamma\|_{C^\beta}\|g\|_{\sol_{1,p}}\|u\|_{\solb^\epsilon_{2,p}}+\|\bar\gamma\|_{C^\beta}\|\rho\|_{\solb^\epsilon_{1,p}})+H\big),
 \end{split}\end{align}
 with $N_1,N$ independent of $T\le1$. Choosing $T_0\le1$ such that $N_1NT_0^{\beta-\frac12}\|\gamma\|_{C^\beta}\|g\|_{\sol_{1,p}}\le\frac12$, we obtain \eqref{est:apr-add-small}.

 Finally we show for arbitrary $T$. The estimate \eqref{est:apr-add-small} holds, with the same $T_0$ and $N_0$, on every interval $[S,S']\subset[0,T]$ of length at most $T_0$, with $\phi$ replaced by $u_{S'}$; this follows by a time shift, since the restrictions of $\gamma,\bar\gamma,g,\rho,h$ have norms bounded by those on $[0,T]$. Cover $[0,T]$ by the intervals $J_k:=[(T-kT_0)\vee0,T-(k-1)T_0]$, $k=1,\ldots,\lceil T/T_0\rceil$. Applying the short-time estimate successively on $J_1,J_2,\ldots$, with terminal values $u_{T-(k-1)T_0}$ whose $\bar B^{2-\epsilon}$-norms are controlled by the estimate on $J_{k-1}$, we bound $\|u\|_{\solb^\epsilon_{2,p}(J_k)}$ for every $k$. The supremum part of the norm on $[0,T]$ is the maximum over $k$, and the $C^{\frac12}$-seminorm on $[0,T]$ is bounded by the number of intervals times the maximum of the seminorms on the $J_k$. This shows \eqref{est:apr-add}.
\end{proof}

\subsection{The general case}\label{sub:general}
\begin{theorem}[A priori estimates]\label{Thm:est-roughPDE}
Suppose \cref{ass:main-PDE} holds. Let $\phi\in B^{2-\epsilon}_{p,\infty}$, $ f\in \mathbb D^\epsilon$ and $\rho \in \sol^\epsilon_{1,p}$. Let $u\in \sol^\epsilon_{2,p}$ be an $\epsilon$-solution to \eqref{main-PDE} such that $u(T)=\phi$.
 Put $\beta_1:=\min\{\beta,1-\frac\epsilon2\}$, $\lambda:=\frac12(1+\frac dp+\epsilon)$ and $\theta:=\frac\lambda2+(1-\lambda)\beta_1\in(\frac12,\beta)$. Then
\begin{align}
    \label{apriori-est}
        \|u\|_{C^{\theta}([0,T],C_b(\Rd))}+\|u\|_{C^{\beta_1}([0,T],B^{-\epsilon}_{p,\infty})} + \Vert u\Vert_{\sol^\epsilon_{2,p}}\leq \cN\big(\Vert f\Vert_{\mathbb D^\epsilon}+\Vert \bar\gamma\Vert_{C^\beta}\|\rho\|_{\sol^\epsilon_{1,p}}+\|\phi\|_{B^{2-\epsilon}_{p,\infty}}\big),
\end{align}
where $\cN$ is a finite constant depending on $K, \alpha, p, d,T,\beta,\epsilon$ and on upper bounds for $\Vert g\Vert_{ \sol_{1,p}}$, $\Vert \gamma\Vert_{C^\beta}$, $\|b\|_{\LL_{q}^\infty([0,T])}$. The same holds if $\rho d\bar\gamma$ is replaced by a finite sum $\sum_i\rho^id\bar\gamma^i$, with $\sum_i\|\bar\gamma^i\|_{C^\beta}\|\rho^i\|_{\sol^\epsilon_{1,p}}$ on the right-hand side.
\end{theorem}
\begin{proof}
We use $c_\delta$ to denote a generic finite constant depending on $\delta$ (and on the parameters of the statement), and $N$ for a constant not depending on $\delta$ and $\tau$ below. Write $B^\nu=B^\nu_{p,\infty}$. We divide the proof into three steps.

\textit{Step 1: short time intervals.} We show that there are $\tau_0\in(0,1]$ and $N_2$ such that for every interval $J=[S,S']\subset[0,T]$ with $S'-S\le\tau_0$,
\begin{align}\label{est:short}
    \|u\|_{\sol^\epsilon_{2,p}(J)}\le N_2\big(\|f\|_{\mathbb D^\epsilon}+\|\bar\gamma\|_{C^\beta}\|\rho\|_{\sol^\epsilon_{1,p}}+\|u_{S'}\|_{B^{2-\epsilon}}\big).
\end{align}
By a time shift we may assume $J=[0,\tau]$, $\tau=S'\le T$. Let $\xi$ be a nonnegative smooth function with support in the unit ball, $\xi\not\equiv0$, let $\psi\in C^\infty_c(\Rd)$ with $\psi=1$ on the unit ball and support in the ball of radius $2$, put $\xi^\delta(\cdot):=\delta^{-d}\xi(\cdot/\delta)$ for $\delta\in(0,1]$ to be chosen, $\xi^{z,\delta}(x):=\xi^\delta(x-z)$ and $\psi^{z,\delta}(x):=\psi((x-z)/\delta)$. Define the linear operators $\Xi,\Xi_k,\Xi_{kl}$ by $(\Xi v)^z:=\xi^{z,\delta}v$, $(\Xi_kv)^z:=(\partial_k\xi^{z,\delta})v$, $(\Xi_{kl}v)^z:=(\partial^2_{kl}\xi^{z,\delta})v$. By \cref{lem.localization}, applied at each fixed time and to time increments, these operators are bounded from $B^\nu$ to $\bar B^\nu$, from $L_p$ to $\Lb_p$, from $\sol^\epsilon_{k,p}$ to $\solb^\epsilon_{k,p}$ and from $\mathbb D^\epsilon$ to $\bar{\mathbb D}^\epsilon$, with norms at most $c_\delta$, and $\|v\|_{\sol^\epsilon_{k,p}}\le c_\delta\|\Xi v\|_{\solb^\epsilon_{k,p}}$.

Put $\bar u:=\Xi u$, $\bar\rho:=\Xi\rho$, $\bar f:=\Xi f$, $\bar\phi:=\Xi u_\tau$. Applying the bounded operator $\Xi:B^{-1-\epsilon}\to\bar B^{-1-\epsilon}$ to \eqref{id.intsolution-PDE}, using \eqref{id.Hintf} to interchange $\Xi$ with the Young integrals, and using that $\Xi$ is also bounded from $B^{-2-\epsilon}_{p,p}$ to $\bar B^{-2-\epsilon}_{p,p}$ (\cref{lem.localization} with $r=p$) to interchange it with the $dr$-integrals, we see that $\bar u\in\solb^\epsilon_{2,p}$ satisfies \eqref{eq:YPDE} on $[0,\tau]$ with the frozen coefficient $a(r,z)$, i.e.\ $(a(r)\partial^2_{ij}\bar u)^z=a^{ij}(r,z)\partial^2_{ij}\bar u^z$, with $\rho d\bar\gamma$ replaced by $F_2d\gamma+\bar\rho d\bar\gamma$ and $h$ replaced by $F_1+\bar f-\Xi(b\cdot\nabla u)$, where
\begin{align}\label{def:F1F2}
  F_1^z:&=\frac{1}{2}\big(a^{ij}(\cdot,z)\partial^2_{ij}\bar u^z-\xi^{z,\delta}a^{ij}\partial^2_{ij}u\big) \tand
 F_2^z:=g\cdot  \nabla \bar u^z-\xi^{z,\delta}g\cdot \nabla u=u\,g\cdot\nabla\xi^{z,\delta}.
\end{align}
In other words,
\begin{align}\label{eq:frozen}
    d_t\bar u^z+\frac12a^{ij}(t,z)\partial^2_{ij}\bar u^zdt+g\cdot \nabla  \bar u^zd\gamma_t=\big(F_1^z+ \bar f^z-\xi^{z,\delta}b\cdot\nabla u\big)dt+F_2^zd\gamma_t+\bar \rho^zd\bar \gamma_t.
\end{align}
Since $a$ is symmetric and $\psi^{z,\delta}=1$ on the support of $\xi^{z,\delta}$, the product rule gives, for every $r$,
\begin{align*}
  2F_1^z=m^z\,\partial^2_{ij}\bar u^z+2a^{ij}(r,\cdot)(\Xi_j\partial_iu)^z+a^{ij}(r,\cdot)(\Xi_{ij}u)^z,\qquad m^z:=\big(a^{ij}(r,z)-a^{ij}(r,\cdot)\big)\psi^{z,\delta}.
\end{align*}
By $(H^a_2)$, $\sup_z\|m^z\|_{C^{\alpha'}}\le NK\delta^{\alpha-\alpha'}$ with $\alpha':=\frac{\alpha+\epsilon}2\in(\epsilon,\alpha)$. Indeed, $|m^z|\le K(2\delta)^\alpha$, and for $|x-y|\le\delta$ we have $$|m^z(x)-m^z(y)|\le K|x-y|^\alpha+K(2\delta)^\alpha\|\nabla\psi\|_\infty\delta^{-1}|x-y|\lesssim K\delta^{\alpha-\alpha'}|x-y|^{\alpha'},$$ while for $|x-y|>\delta$ we have $|m^z(x)-m^z(y)|\le2K(2\delta)^\alpha\lesssim K\delta^{\alpha-\alpha'}|x-y|^{\alpha'}$. Hence, by \cref{lem.prodS} (i) with $\nu=-\epsilon$ and $\varrho=\alpha'$,
\begin{align*}
    \|m\,\partial^2_{ij}\bar u_r\|_{\bar B^{-\epsilon}}\le N\delta^{\frac{\alpha-\epsilon}2}\|\bar u_r\|_{\bar B^{2-\epsilon}}.
\end{align*}
The other two terms of $2F_1$ are functions, bounded in $\Lb_p$ by $c_\delta\|u_r\|_{L_{1,p}}$. To make them small, let $\kappa:=\frac12(\frac dp+1-\epsilon)\in(\frac dp,1-\epsilon)$. Then $B^{1+\kappa}\hookrightarrow L_{1,p}$ and $\nabla:B^{1+\kappa}\to B^\kappa\hookrightarrow L_p\cap L_\infty$, and interpolating between $B^{1-\epsilon}$ and $B^{2-\epsilon}$ (note $1+\kappa=(1-\kappa-\epsilon)(1-\epsilon)+(\kappa+\epsilon)(2-\epsilon)$ with $\kappa+\epsilon<1$),
\begin{align*}
    \|u_r-u_\tau\|_{B^{1+\kappa}}\le\|u_r-u_\tau\|_{B^{1-\epsilon}}^{1-\kappa-\epsilon}\|u_r-u_\tau\|_{B^{2-\epsilon}}^{\kappa+\epsilon}\lesssim\tau^{\kappa'}\|u\|_{\sol^\epsilon_{2,p}([0,\tau])},\qquad \kappa':=\frac{1-\kappa-\epsilon}2>0 .
\end{align*}
Hence $\sup_{r\le\tau}\|u_r\|_{B^{1+\kappa}}\lesssim\|u_\tau\|_{B^{2-\epsilon}}+\tau^{\kappa'}\|u\|_{\sol^\epsilon_{2,p}([0,\tau])}$. Using also $\|u\|_{\sol^\epsilon_{2,p}}\le c_\delta\|\bar u\|_{\solb^\epsilon_{2,p}}$, we obtain
\begin{align*}
    \|F_1\|_{L^\infty([0,\tau];\bar B^{-\epsilon})}\le\big(N\delta^{\frac{\alpha-\epsilon}2}+c_\delta\tau^{\kappa'}\big)\|\bar u\|_{\solb^\epsilon_{2,p}([0,\tau])}+c_\delta\|u_\tau\|_{B^{2-\epsilon}},
\end{align*}
and, by \eqref{est:bnablau} with $\|\nabla u_r\|_{L_p\cap L_\infty}\lesssim\|u_r\|_{B^{1+\kappa}}$, also $$\|\Xi(b\cdot\nabla u)\|_{\LLb^\infty_p([0,\tau])}\le c_\delta\|b\|_{\LL^\infty_q}(\tau^{\kappa'}\|\bar u\|_{\solb^\epsilon_{2,p}([0,\tau])}+\|u_\tau\|_{B^{2-\epsilon}}).$$ Moreover $F_2=g\cdot(\Xi_ku)_{k}$, so that \cref{lem.prodS} yields
\begin{align}\label{rem.g-critical}
  \|F_2\|_{\solb^\epsilon_{1,p}([0,\tau])}  \le c_\delta\|g\|_{\sol_{1,p}}\|u\|_{\sol^\epsilon_{1,p}([0,\tau])}\le c_\delta\|\bar u\|_{\solb^\epsilon_{2,p}([0,\tau])}.
\end{align}
We now apply \eqref{est:apr-add-small} on $[0,\tau]$, assuming $\tau\le T_0$; the constants $T_0$ and $N_0$ do not depend on $z$, since $a(\cdot,z)$ satisfies $(H^a_1)$ with the same $K$ for all $z$. We obtain
\begin{align*}
   \Vert \bar u\Vert_{\solb^\epsilon_{2,p}([0,\tau])}
   &\le N_0\big(\tau^{\beta-\frac12}\|\gamma\|_{C^\beta}\|F_2\|_{\solb^\epsilon_{1,p}}+\|\bar\gamma\|_{C^\beta}\|\bar\rho\|_{\solb^\epsilon_{1,p}}+\|F_1\|_{L^\infty\bar B^{-\epsilon}}
   \\&\qquad+\|\bar f\|_{\bar{\mathbb D}^\epsilon}+\|\Xi(b\cdot\nabla u)\|_{\LLb^\infty_p}+\|\bar\phi\|_{\bar B^{2-\epsilon}}\big)
   \\&\le N_0\big(N\delta^{\frac{\alpha-\epsilon}2}+c_\delta\tau^{\kappa'}+c_\delta\tau^{\beta-\frac12}\big)\|\bar u\|_{\solb^\epsilon_{2,p}([0,\tau])}
   \\&\qquad+c_\delta\big(\|\bar\gamma\|_{C^\beta}\|\rho\|_{\sol^\epsilon_{1,p}}+\|f\|_{\mathbb D^\epsilon}+\|u_\tau\|_{B^{2-\epsilon}}\big).
\end{align*}
We first choose $\delta$ such that $N_0N\delta^{\frac{\alpha-\epsilon}2}\le\frac14$ and then $\tau_0\le T_0$ such that $N_0c_\delta(\tau_0^{\kappa'}+\tau_0^{\beta-\frac12})\le\frac14$. For $\tau\le\tau_0$ this gives $$\|\bar u\|_{\solb^\epsilon_{2,p}([0,\tau])}\le 2c_\delta(\|\bar\gamma\|_{C^\beta}\|\rho\|_{\sol^\epsilon_{1,p}}+\|f\|_{\mathbb D^\epsilon}+\|u_\tau\|_{B^{2-\epsilon}}),$$ and \eqref{est:short} follows from $\|u\|_{\sol^\epsilon_{2,p}}\le c_\delta\|\bar u\|_{\solb^\epsilon_{2,p}}$.

\textit{Step 2: the whole interval.} Applying \eqref{est:short} successively on the intervals $[(T-k\tau_0)\vee0,T-(k-1)\tau_0]$, as in the proof of \cref{thm:YPDE-add}, yields
\begin{align}
    \label{est:Lp-L2p}
     \Vert u\Vert_{\sol^\epsilon_{2,p}([0,T])}\lesssim \Vert f\Vert_{\mathbb D^\epsilon}+\Vert \bar\gamma\Vert_{C^\beta}\Vert \rho\Vert_{\sol^\epsilon_{1,p}}+\|\phi\|_{B^{2-\epsilon}}.
\end{align}

\textit{Step 3: H\"older norms in time.} From \eqref{id.intsolution-PDE}, $$\|\int_s^t(\frac12a^{ij}\partial^2_{ij}u+b\cdot\nabla u)dr\|_{B^{-\epsilon}}\lesssim(t-s)\|u\|_{\sol^\epsilon_{2,p}}$$ by \cref{lem.prodS} (i) and \eqref{est:bnablau}; $\|\int_s^tf_rdr\|_{B^{-\epsilon}}\lesssim(t-s)^{1-\frac\epsilon2}\|f\|_{\mathbb D^\epsilon}$ by H\"older's inequality; and the Young integrals are bounded in $B^{-\epsilon}$ by $(t-s)^\beta(\|\gamma\|_{C^\beta}\|g\cdot\nabla u\|_{\sol^\epsilon_{1,p}}+\|\bar\gamma\|_{C^\beta}\|\rho\|_{\sol^\epsilon_{1,p}})$ by \cref{lem:youngPdg} (i) and \cref{lem.prodS}. Together with \eqref{est:Lp-L2p}, this gives the estimate for $\|u\|_{C^{\beta_1}([0,T],B^{-\epsilon})}$ in \eqref{apriori-est}.
To estimate $\|u\|_{C^{\theta}([0,T],C_b(\Rd))}$, we interpolate between $B^{-\epsilon}$ and $B^{1-\epsilon}$: for $s\le t$,
    \begin{align*}
        \|u_t-u_s\|_{B^{\lambda-\epsilon}}\le \|u_t-u_s\|_{B^{-\epsilon}}^{1-\lambda} \|u_t-u_s\|_{B^{1-\epsilon}}^{\lambda}\lesssim (t-s)^{(1-\lambda)\beta_1+\frac\lambda2}\big(\|u\|_{C^{\beta_1}([0,T]; B^{-\epsilon})}+\|u\|_{\sol^\epsilon_{2,p}}\big).
    \end{align*}
    Since $\lambda-\epsilon=\frac12(1+\frac dp-\epsilon)>\frac dp$ because $\epsilon<1-\frac dp$, we have $B^{\lambda-\epsilon}\hookrightarrow C_b(\Rd)$; also $\sup_t\|u_t\|_{C_b}\lesssim\sup_t\|u_t\|_{B^{2-\epsilon}}$. This gives \eqref{apriori-est}, and $\theta\in(\frac12,\beta)$ because $\lambda<1$ and $\beta_1>\frac12$. The case of finitely many Young terms $\sum_i\rho^id\bar\gamma^i$ is identical.
\end{proof}

\begin{corollary}[Stability] \label{cor:stab}
Suppose \cref{ass:main-PDE} holds and let $\rho \in \sol^\epsilon_{1,p}$. For $i=1,2$, let $ f^i\in \mathbb D^\epsilon$, let $\gamma$, $\bar\gamma^i$ be  $\beta$-H\"older  paths, and let $u^i\in \sol^\epsilon_{2,p}$ be $\epsilon$-solutions to \eqref{main-PDE} with $f=f^i$, $\bar\gamma=\bar\gamma^i$. Then, with $\theta$ and $\cN$ from \cref{Thm:est-roughPDE},
\begin{align}
    \label{est:stab}
    \begin{split}
        &\|u^1-u^2\|_{C^{\theta}([0,T],C_b(\Rd))}+\Vert u^1-u^2\Vert_{\sol^\epsilon_{2,p}}
        \\&\leq \cN\big(\Vert f^1-f^2\Vert_{\mathbb D^\epsilon}+\Vert \bar\gamma^1-\bar\gamma^2\Vert_{C^\beta}\|\rho\|_{\sol^\epsilon_{1,p}}+\|u^1(T)-u^2(T)\|_{B^{2-\epsilon}_{p,\infty}} \big).
    \end{split}
\end{align}
\end{corollary}
\begin{proof}
The difference $u:=u^1-u^2$ is a solution to
\begin{align*}
     d_t u+(\frac{1}{2}a^{ij}\partial^2_{ij}u+b\cdot \nabla u)dt+(g\cdot \nabla u) d\gamma_t=(f^1-f^2)dt +\rho d (\bar\gamma_t^1-\bar\gamma_t^2),
\end{align*}
so \eqref{est:stab} follows from \eqref{apriori-est}.
\end{proof}
\begin{theorem}
    [Existence and Uniqueness]\label{thm:solution-YPDE}
    Suppose \cref{ass:main-PDE} holds, let $\phi\in B^{2-\epsilon}_{p,\infty}$, $ f\in \mathbb D^\epsilon$ and $\rho \in \sol^\epsilon_{1,p}$. Then there exists a unique $\epsilon$-solution $u\in \sol^\epsilon_{2,p}$ to \eqref{main-PDE} such that $u(T)=\phi$. If, in addition, $\gamma$ and $\bar\gamma$ are continuously differentiable, then $\partial_tu\in\mathbb D^\epsilon$.
\end{theorem}
\begin{proof} Uniqueness follows directly from \cref{cor:stab}. Existence is shown in two steps.

\textit{Step 1: $\gamma,\bar\gamma\in C^1([0,T])$.} In this case, the Young integrals in \eqref{id.intsolution-PDE} are Riemann--Stieltjes integrals, e.g.\ $\int_s^t g\cdot\nabla u_r d\gamma_r=\int_s^t\gamma'_rg\cdot\nabla u_rdr$. We use the method of continuity in the spaces
\begin{align*}
    X:=\{u\in\sol^\epsilon_{2,p}:\partial_tu\in\mathbb D^\epsilon,\ u(T)=0\},\quad \|u\|_X:=\|u\|_{\sol^\epsilon_{2,p}}+\|\partial_tu\|_{\mathbb D^\epsilon},\qquad Y:=\mathbb D^\epsilon.
\end{align*}
For $\tau\in[0,1]$ put $a^\tau:=(1-\tau)I+\tau a$, which satisfies \cref{ass:main-PDE} $(H_1^a)$, $(H_2^a)$ with the same $K$ and $\alpha$, and
\begin{align*}
    L_\tau u:=\partial_tu+\frac12a^{\tau,ij}\partial^2_{ij}u+\tau(b+\gamma'g)\cdot\nabla u.
\end{align*}
By \cref{lem.prodS} (i) and \eqref{est:bnablau}, $L_\tau:X\to Y$ is bounded, uniformly in $\tau$. 
Furthermore, by \cref{Thm:est-roughPDE}, $\|u\|_{\sol^\epsilon_{2,p}}\le N\|L_\tau u\|_Y$, where $N$ does not depend on $\tau$.
It is evident that the operator $L_0=\partial_t+\frac12\Delta$ is a bijection from $X$ onto $Y$. 
By the method of continuity, $L_1:X\to Y$ is a bijection.

Now let $\psi_s:=P^I_{s,T}\phi$. Then $\psi\in\sol^\epsilon_{2,p}$ by \cref{lem:semi-group-reg}, and $\partial_s\psi=-\frac12\Delta\psi\in L^\infty([0,T];B^{-\epsilon}_{p,\infty})$. Let $w\in X$ be such that $L_1w=f+\rho\bar\gamma'-L_1\psi\in Y$ (note that $\rho\bar\gamma'\in\LL^\infty_p$ by \eqref{eq:S1e-emb}). Then $u:=w+\psi\in\sol^\epsilon_{2,p}$, $\partial_tu\in Y$, $u(T)=\phi$, and integrating $\partial_tu$ in time shows that $u$ is a solution to \eqref{main-PDE}.

\textit{Step 2: general $\gamma,\bar\gamma$.} Fix $\beta'\in(\frac12,\beta)$, so that all results of this section apply with $\beta'$ in place of $\beta$. Let $(\gamma^n)_n,(\bar \gamma^n)_n\subset C^\infty([0,T])$ be sequences converging to $\gamma,\bar\gamma$ in $C^{\beta-}([0,T])$, so that they are bounded in $C^\beta$ and converge in $C^{\beta'}$. Let $u^n\in\sol^\epsilon_{2,p}$ be the solution from Step 1 to
\begin{align}\label{eq:pde-un}
  d_t u^n+(\frac{1}{2}a^{ij}\partial^2_{ij}u^n+b\cdot \nabla u^n)dt+(g\cdot \nabla u^n) d \gamma^n_t=f dt+\rho d\bar\gamma_t^n,
  \qquad
   u^n(T,\cdot)=\phi.
\end{align}
By \cref{Thm:est-roughPDE}, $\sup_n\|u^n\|_{\sol^\epsilon_{2,p}}=:N_*<\infty$. The difference $u^n-u^m$ solves \eqref{main-PDE} with $\gamma^n$ in place of $\gamma$, $f=0$, zero terminal value and the two Young terms $\rho d(\bar\gamma^n-\bar\gamma^m)-g\cdot\nabla u^md(\gamma^n-\gamma^m)$. Since $\|g\cdot\nabla u^m\|_{\sol^\epsilon_{1,p}}\lesssim\|g\|_{\sol_{1,p}}N_*$ by \cref{lem.prodS}, \cref{Thm:est-roughPDE} (with $\beta'$) gives
\begin{align*}
    \|u^n-u^m\|_{\sol^\epsilon_{2,p}}\lesssim\|\bar\gamma^n-\bar\gamma^m\|_{C^{\beta'}}\|\rho\|_{\sol^\epsilon_{1,p}}+\|\gamma^n-\gamma^m\|_{C^{\beta'}}\|g\|_{\sol_{1,p}}N_*\to0,
\end{align*}
with constants uniform in $n,m$. Hence $u^n\to u$ in $\sol^\epsilon_{2,p}$ for some $u\in\sol^\epsilon_{2,p}$ with $u(T)=\phi$. We pass to the limit in \eqref{id.intsolution-PDE} for $u^n$ in $B^{-1-\epsilon}_{p,\infty}$: the $dr$-integrals converge by \cref{lem.prodS} (i) and \eqref{est:bnablau}, and, writing $$\int g\cdot\nabla u^nd\gamma^n-\int g\cdot\nabla ud\gamma=\int g\cdot\nabla(u^n-u)d\gamma^n+\int g\cdot\nabla u\,d(\gamma^n-\gamma),$$ the Young integrals converge by \eqref{est:young-Lp} (with $\beta'$) and \cref{lem.prodS}; the same applies to $\int\rho d\bar\gamma^n$. Thus $u$ is a solution to \eqref{main-PDE}. Finally, if $\gamma,\bar\gamma\in C^1$, the solution from Step 1 satisfies $\partial_tu\in\mathbb D^\epsilon$ and coincides with $u$ by uniqueness.
\end{proof}
\section{Krylov estimates for rough SDEs and tightness}
\label{sec:tightness}
Throughout this section we assume that $\gamma\in C^\beta([0,T])$ with $\beta\in(\frac12,1)$ and that \cref{ass:main} holds, and we put $a:=\sigma\sigma^\top$; thus $a$, $b$, $g$ satisfy \cref{ass:main-PDE}, with the H\"older exponent $\alpha$ of $\sigma$ and, in $(H^a_2)$, with $K$ replaced by $2K^{3/2}$. 
Let $\epsilon_0$ be as in \eqref{def:eps0}, and put
\begin{align}\label{def:beta-eps}
    \beta_0:=\frac12\Big(\beta+\frac12\Big)\in\Big(\frac12,\beta\Big),\qquad q_*:=\frac4{\epsilon_0}.
\end{align}
All results of \cref{sec:PDE} apply with $\beta_0$ in place of $\beta$. The point of $\beta_0$ is that smooth approximations of $\gamma$ converge in $C^{\beta_0}$, but in general not in $C^\beta$.

In this section we prove Krylov-type estimates for the approximations of the solution to the Young SDE \eqref{eq:rough-SDE} obtained by mollifying $\gamma$. These estimates connect the Young PDE \eqref{main-PDE} with the laws of the approximating solutions, and they yield candidate martingale (or weak) solutions through a tightness argument.

Let $(\gamma^n)_n\subset C^\infty([0,T])$ be a sequence converging to $\gamma$ in $C^{\beta-}([0,T])$.
Let $\varphi_n(x)=n^d\varphi(nx)$ for a probability density $\varphi\in C^\infty_c(\Rd)$, and mollify $\sigma$ and $b$ in space:
\begin{align*}
    \sigma^n(t,\cdot):=\sigma(t,\cdot)\ast\varphi_n,\qquad b^n(t,\cdot):=b(t,\cdot)\ast\varphi_n,\qquad a^n:=\sigma^n(\sigma^n)^\top.
\end{align*}
The coefficient $g$ is not mollified. 
By changing the constant $K$ in \cref{ass:main-PDE} we can assume that $a^n,b^n,g$ satisfy \cref{ass:main-PDE} uniformly in $n$.
Moreover, $\|b^n\|_{\LL^\infty_q}\le\|b\|_{\LL^\infty_q}$; for each fixed $n$, $a^n(t,\cdot)$ has bounded derivatives of all orders, and $b^n$ and $\nabla b^n$ belong to $L^\infty([0,T];L_q\cap L_\infty)$, with bounds depending on $n$; and, for every $R>0$ and $q'<\infty$,
\begin{align}\label{eq:bn-conv}
    \lim_{n\to\infty}\|(b^n-b)\1_{B_R}\|_{\LL^{q'}_p}=0,
\end{align}
by the convergence of mollifications in $L_p(B_R)$ for each $t$ and dominated convergence in $t$. For each $n$ and $x\in\Rd$ let $X^n=X^n(x)$, defined on a filtered probability space $(\Omega,\cF,(\cF_t),\mP)$ with an $(\cF_t)$-Brownian motion $W$, be a weak solution to the mollified equation
\begin{align}
    \label{eq:SDE-m} dX_t^n=b^n(t,X_t^n)dt+g(t,X_t^n)d\gamma_t^n+\sigma^n(t,X_t^n)dW_t,\quad X_0^n=x .
\end{align}
Since the drift $b^n+g\,(\gamma^n)'$ is bounded and measurable, and $\sigma^n$ is bounded, continuous and nondegenerate, such a solution exists \cite[Chapter 7]{SV}. The results below hold for every such solution.

We use It\^o's formula in the following form, for functions which are twice continuously differentiable in $x$ but only absolutely continuous in $t$.
\begin{lemma}[It\^o's formula]\label{lem:Ito}
    Let $W$ be an $(\cF_t)$-Brownian motion in $\Rd$ on a filtered probability space $(\Omega,\cF,(\cF_t),\mP)$, and let $X$ be a continuous adapted process with
    \begin{align*}
        X_t=X_0+\int_0^t\beta_rdr+\int_0^t\varsigma_rdW_r,\qquad t\in[0,T],
    \end{align*}
    where $\beta$ and $\varsigma$ are bounded progressively measurable processes with values in $\Rd$ and $\R^{d\times d}$. Let $u:[0,T]\times\Rd\to\R$ be such that $u$, $\nabla u$ and $\nabla^2u$ are bounded and continuous on $[0,T]\times\Rd$, and assume that there is a bounded measurable $v:[0,T]\times\Rd\to\R$ such that $x\mapsto v(r,x)$ is uniformly continuous, uniformly in $r\in[0,T]$, and
    \begin{align}\label{eq:u-ac}
        u(t,x)-u(s,x)=\int_s^tv(r,x)dr\qquad\text{for all }0\le s\le t\le T,\ x\in\Rd.
    \end{align}
    Then, almost surely, for all $0\le s\le t\le T$,
    \begin{align}\label{eq:Ito-classical}
        u(t,X_t)=u(s,X_s)+\int_s^t\Big(v+\beta_r\cdot\nabla u+\frac12\mathrm{tr}\big(\varsigma_r\varsigma_r^\top\nabla^2u\big)\Big)(r,X_r)dr+\int_s^t\nabla u(r,X_r)\cdot\varsigma_rdW_r.
    \end{align}
\end{lemma}
\begin{proof}
    Both sides of \eqref{eq:Ito-classical} are continuous in $(s,t)$, so it suffices to prove \eqref{eq:Ito-classical} almost surely for fixed $s<t$. Let $\pi=\{s=t_0<t_1<\dots<t_N=t\}$ be a partition, and for $r\in[t_k,t_{k+1})$ put $\tau(r):=t_k$ and $\tau^+(r):=t_{k+1}$. For each $k$,
    \begin{align*}
        u(t_{k+1},X_{t_{k+1}})-u(t_k,X_{t_k})=\big(u(t_{k+1},X_{t_{k+1}})-u(t_{k+1},X_{t_k})\big)+\int_{t_k}^{t_{k+1}}v(r,X_{t_k})dr,
    \end{align*}
    by \eqref{eq:u-ac} at the point $x=X_{t_k}$. Applying the classical It\^o formula to the $C^2_b$-function $u(t_{k+1},\cdot)$ on $[t_k,t_{k+1}]$ and summing over $k$, we obtain
    \begin{align*}
        u(t,X_t)-u(s,X_s)&=\int_s^t\Big(v(r,X_{\tau(r)})+\beta_r\cdot\nabla u(\tau^+(r),X_r)+\frac12\mathrm{tr}\big(\varsigma_r\varsigma_r^\top\nabla^2u(\tau^+(r),X_r)\big)\Big)dr
        \\&\qquad+\int_s^t\nabla u(\tau^+(r),X_r)\cdot\varsigma_rdW_r.
    \end{align*}
    Let the mesh of $\pi$ tend to $0$. Since the path of $X$ is continuous, hence bounded, and $\nabla u$, $\nabla^2u$ are uniformly continuous on bounded subsets of $[0,T]\times\Rd$, we have $\nabla u(\tau^+(r),X_r)\to\nabla u(r,X_r)$ and $\nabla^2u(\tau^+(r),X_r)\to\nabla^2u(r,X_r)$ for every $r$, almost surely. Moreover, $|v(r,X_{\tau(r)})-v(r,X_r)|\le\omega_v(\sup_{|r-r'|\le|\pi|}|X_r-X_{r'}|)\to0$ uniformly in $r$, where $\omega_v$ is the modulus of continuity of $v$ in $x$. By dominated convergence, the $dr$-integrals converge almost surely, and, by It\^o's isometry, the stochastic integrals converge in $L_2(\mP)$. This proves \eqref{eq:Ito-classical}.
\end{proof}

The next lemma shows that, for the mollified coefficients, the solutions of the Kolmogorov equations satisfy the assumptions of \cref{lem:Ito}.
\begin{lemma}\label{lem:reg-un}
    Let $\epsilon\in[0,\epsilon_0)$, let $a,b,g$ satisfy \cref{ass:main-PDE}, and assume in addition that $a(t,\cdot)$ has bounded first and second derivatives and that $b,\nabla b\in L^\infty([0,T];L_q\cap L_\infty)$, uniformly in $t$. Let $\gamma,\bar\gamma\in C^1([0,T])$, $f\in C^\infty_c([0,T]\times\Rd)$ and $\rho\in\sol^\epsilon_{1,p}$, and let $u\in\sol^\epsilon_{2,p}$ be the solution to \eqref{main-PDE} with $u(T)=0$, given by \cref{thm:solution-YPDE}. Then $\nabla u\in\sol^\epsilon_{2,p}$, the functions $u$, $\nabla u$ and $\nabla^2u$ are bounded and continuous on $[0,T]\times\Rd$, and $u$ satisfies \eqref{eq:u-ac} with
    \begin{align*}
        v:=f+\rho\,\bar\gamma'-\frac12a^{ij}\partial^2_{ij}u-(b+g\gamma')\cdot\nabla u,
    \end{align*}
    which is bounded and measurable, with $x\mapsto v(r,x)$ uniformly continuous, uniformly in $r$.
\end{lemma}
\begin{proof}
    \textit{Step 1: $\nabla u\in\sol^\epsilon_{2,p}$.} For $h\in\Rd\setminus\{0\}$ let $\tau_hw:=w(\cdot+h)$ and $D_hw:=(\tau_hw-w)/|h|$. Since $\gamma,\bar\gamma\in C^1$, the Young integrals in \eqref{id.intsolution-PDE} are Riemann--Stieltjes integrals. The function $\tau_hu$ solves \eqref{main-PDE} with the coefficients and data translated by $h$. Subtracting the equation for $u$, we see that $D_hu\in\sol^\epsilon_{2,p}$ solves \eqref{main-PDE} with $\rho=0$, $D_hu(T)=0$ and $f$ replaced by
        \begin{align*}
        F_h&:=D_hf+\bar\gamma'D_h\rho-\frac12(D_ha^{ij})\partial^2_{ij}\tau_hu-(D_hb+\gamma'D_hg)\cdot\nabla\tau_hu.
    \end{align*}
    We bound $F_h$ in $\mathbb D^\epsilon$, uniformly in $h$. On each Littlewood--Paley block, $D_h$ is the Fourier multiplier $(e^{ih\cdot\xi}-1)/|h|$, which is bounded by $|\xi|$; hence $\|D_hw\|_{B^{\nu}_{p,\infty}}\le N\|w\|_{B^{\nu+1}_{p,\infty}}$, and $\|D_h\rho\|_{L^\infty B^{-\epsilon}_{p,\infty}}\le N\|\rho\|_{\sol^\epsilon_{1,p}}$. Also $\|D_hf\|_{\LL^\infty_p}\le\|\nabla f\|_{\LL^\infty_p}$. Since $\|D_ha(t)\|_{C^1}\le\|\nabla a(t)\|_{C^1}$, \cref{lem.prodS} (i) with $\nu=-\epsilon$ and $\varrho=1$ gives $\|(D_ha)\partial^2\tau_hu\|_{L^\infty B^{-\epsilon}_{p,\infty}}\lesssim\|u\|_{\sol^\epsilon_{2,p}}$. As in \eqref{est:bnablau}, $\|(D_hb)\cdot\nabla\tau_hu\|_{\LL^\infty_p}\lesssim\|\nabla b\|_{\LL^\infty_q}\|u\|_{\sol^\epsilon_{2,p}}$. Finally $\|D_hg(t)\|_{L_p}\le\|\nabla g(t)\|_{L_p}$ and $\nabla u$ is bounded by \eqref{eq:S1e-emb}, so $\|(D_hg)\cdot\nabla\tau_hu\|_{\LL^\infty_p}\lesssim\|g\|_{\sol_{1,p}}\|u\|_{\sol^\epsilon_{2,p}}$. By \cref{Thm:est-roughPDE}, $\sup_h\|D_hu\|_{\sol^\epsilon_{2,p}}<\infty$. Taking $h=\lambda e_k$ and letting $\lambda\to0$, we have $\Delta_jD_hu_t\to\partial_k\Delta_ju_t$ in $L_p$ for every $j$ and $t$, so that
    \begin{align*}
        \|\partial_ku\|_{\sol^\epsilon_{2,p}}\le\liminf_{\lambda\to0}\|D_{\lambda e_k}u\|_{\sol^\epsilon_{2,p}}<\infty.
    \end{align*}

    \textit{Step 2: continuity.} Every $w\in\sol^\epsilon_{1,p}$ is bounded and continuous on $[0,T]\times\Rd$, and $w(t,\cdot)$ is H\"older continuous of order $1-\epsilon-\frac dp$, uniformly in $t$. Indeed, $B^{1-\epsilon}_{p,\infty}\hookrightarrow C^{1-\epsilon-d/p}$, and, with $\lambda=\frac12(1+\frac dp+\epsilon)$ as in \cref{Thm:est-roughPDE}, $\|w_t-w_s\|_{B^{\lambda-\epsilon}_{p,\infty}}\le\|w_t-w_s\|^{1-\lambda}_{B^{-\epsilon}_{p,\infty}}\|w_t-w_s\|^\lambda_{B^{1-\epsilon}_{p,\infty}}\lesssim|t-s|^{\frac{1-\lambda}2}\|w\|_{\sol^\epsilon_{1,p}}$ with $B^{\lambda-\epsilon}_{p,\infty}\hookrightarrow C_b$. By Step 1, $u$, $\partial_ku$ and $\partial^2_{kl}u$ belong to $\sol^\epsilon_{1,p}$, which proves the claims about $u$, $\nabla u$, $\nabla^2u$.

    \textit{Step 3: the time derivative.} From \eqref{id.intsolution-PDE} with Riemann--Stieltjes integrals, $u_t-u_s=\int_s^tv_rdr$ in $B^{-1-\epsilon}_{p,\infty}$. Each term of $v$ is bounded and belongs to $\LL^\infty_p$: $f$ is smooth with compact support, $\rho$ and $\nabla u$ are bounded by \eqref{eq:S1e-emb}, $\nabla^2u$ is bounded by Step 2, and $b$, $g$ are bounded. Hence the integral is also a Bochner integral in $L_p$, which coincides almost everywhere with the pointwise integral $\int_s^tv(r,x)dr$. Both $u_t-u_s$ and $x\mapsto\int_s^tv(r,x)dr$ are continuous, so \eqref{eq:u-ac} holds everywhere. Finally, $x\mapsto v(r,x)$ is uniformly continuous, uniformly in $r$: $\rho(r,\cdot)$, $\nabla^2u(r,\cdot)$ and $\nabla u(r,\cdot)$ are uniformly H\"older by Step 2, $g(r,\cdot)$ is uniformly H\"older since $L_{1,p}\hookrightarrow C^{1-d/p}$, $a(r,\cdot)$ and $b(r,\cdot)$ are Lipschitz, uniformly in $r$, and all factors are bounded.
\end{proof}

\begin{lemma}[Krylov estimates]\label{lem:Krylov}
 There is $\theta_1\in(\frac12,\beta_0)$, depending only on $p,d,\alpha,\beta$, and for every $q'\in[q_*,\infty]$ a constant $C$, depending on $q'$, $\alpha, K, p,T,d,\beta$, $\Vert b\Vert_{\LL_{q}^{\infty}}, \Vert g\Vert_{\sol_{1,p}}$ and $\sup_n\Vert \gamma^n\Vert_{C^\beta}$, such that for every $0\le s\le t\le T$, every $q'\in[q_*,\infty]$ and every $\Upsilon\in C^\infty([0,T])$,
 \begin{align}
\label{est:Kry-b} \sup_n\sup_{x\in\mR^d} \mE\Big(\int_{s}^{t}|f(r,X_r^n(x))|dr\Big|\cF_s\Big)&\le C |t-s|^{\theta_1}    \Vert f\Vert_{\LL_{p}^{q'}},\quad \forall f\in\LL_p^{q'},
\\\label{est:Kry-g}
\sup_{n}\sup_{x\in\mR^d} \mE\Big(\big|\int_{s}^{t}\rho(r,X_r^n(x))d\Upsilon_r\big|\Big|\cF_s\Big)&\le C |t-s|^{\theta_1} \Vert \Upsilon\Vert_{C^{\beta_0}} \Vert \rho\Vert_{\sol^0_{1,p}},\quad \forall \rho\in\sol^0_{1,p}.
 \end{align}
In particular, for every $m\geq1$ there is a constant $C_m$ such that for all such $f$ and $\rho$,
  \begin{align}\label{est:Krylov}\begin{split}
\sup_n\sup_{x\in\mR^d}& \Big(   \Big\| \int_{s}^{t}|f(r,X_r^n(x))|dr\Big\|_{L_m(\Omega)}+\Big\| \int_{s}^{t}\rho(r,X_r^n(x))d\Upsilon_r\Big\|_{L_m(\Omega)}\Big)
\\&\qquad\leq C_m|t-s|^{\theta_1}\big(\Vert f\Vert_{\LL_{p}^{q'}}+\|\Upsilon\|_{C^{\beta_0}}\Vert \rho\Vert_{\sol^0_{1,p}}\big).
  \end{split}\end{align}
\end{lemma}
\begin{proof}
For $\epsilon\in[0,\frac{\epsilon_0}2]$ let $\theta(\epsilon)$ be the exponent of \cref{Thm:est-roughPDE} applied with $\beta_0$ in place of $\beta$. Since $\lambda=\frac12(1+\frac dp+\epsilon)\le\frac12(1+\frac dp+\frac{\epsilon_0}2)<1$ and $\beta_1=\min\{\beta_0,1-\frac\epsilon2\}\ge\min\{\beta_0,1-\frac{\epsilon_0}4\}>\frac12$, we have $\inf_{\epsilon\le\epsilon_0/2}\theta(\epsilon)>\frac12$. Let $\lambda_1:=\frac12(1+\frac dp)$ and
\begin{align}\label{def:theta}
    \theta_1:=\min\Big\{\inf_{\epsilon\in[0,\epsilon_0/2]}\theta(\epsilon),\ 1-\frac{\lambda_1}2\Big\}\in\Big(\frac12,\beta_0\Big).
\end{align}
Fix $n$, $x$, and $\bar t\in(0,T]$; it suffices to consider $t=\bar t$. All constants below are uniform in $n$, $x$ and $\bar t$, since the constants of \cref{sec:PDE} depend on the path in the Young term only through $\sup_n\|\gamma^n\|_{C^{\beta_0}}$, on the drift only through $\|b^n\|_{\LL^\infty_q}\le\|b\|_{\LL^\infty_q}$, and on the diffusion coefficient only through $\alpha$ and the constant $4K^{3/2}$, which are the same for all $a^n$. The process $X^n$ is of the form considered in \cref{lem:Ito}, with $\beta_r=b^n(r,X^n_r)+g(r,X^n_r)(\gamma^n)'_r$ and $\varsigma_r=\sigma^n(r,X^n_r)$, which are bounded, and $\frac12\mathrm{tr}(\varsigma_r\varsigma_r^\top\nabla^2w)=\frac12a^{n,ij}\partial^2_{ij}w(r,X^n_r)$.

We first show \eqref{est:Kry-b} for $f\in C^\infty_c([0,T]\times\Rd)$, in the form $|\mE(\int_s^{\bar t}f(r,X^n_r)dr|\cF_s)|\le C|\bar t-s|^{\theta_1}\|f\|_{\LL^{q'}_p}$. Let $\epsilon:=\frac2{q'}\in[0,\frac{\epsilon_0}2]$, so that $f\in\LL^{q'}_p=\LL^{q_\epsilon}_p\subset\mathbb D^\epsilon$. By \cref{thm:solution-YPDE} (on $[0,\bar t]$, with $a^n$, $b^n$, $\gamma^n$ in place of $a$, $b$, $\gamma$), there is a solution $u^n\in\sol^\epsilon_{2,p}([0,\bar t])$ to
\begin{align}\label{eq:PDE-no-drift}
    d_tu^n+\Big(\frac{1}{2}a^{n,ij}\partial^2_{ij} u^n+b^n\cdot \nabla u^n\Big)dt+g\cdot \nabla u^n d\gamma_t^n=fdt , \quad u^n(\bar t,\cdot)=0,
\end{align}
By \cref{lem:reg-un} (with $\rho=0$), $u^n$ satisfies the assumptions of \cref{lem:Ito} with $v=f-\frac12a^{n,ij}\partial^2_{ij}u^n-(b^n+g(\gamma^n)')\cdot\nabla u^n$. Hence, by \eqref{eq:Ito-classical},
\begin{align*}
    0=u^n(\bar t,X_{\bar t}^n)= u^n(s,X_s^n)+\int_s^{\bar t}f(r,X_r^n)dr+\int_s^{\bar t}(\nabla u^n\cdot\sigma^n)(r,X_r^n)dW_r,
\end{align*}
where the stochastic integral is a martingale since $\nabla u^n$ is bounded. Hence
\begin{align*}
   \Big|\mE\Big(\int_{s}^{\bar t}f(r,X_r^n)dr\Big|\cF_s\Big)\Big|&=|u^n(s,X_s^n)|=|u^n({\bar t},X_s^n)- u^n(s,X_s^n)|
   \\&\le\|u^n\|_{C^{\theta(\epsilon)}([0,\bar t];C_b)}|\bar t-s|^{\theta(\epsilon)}\lesssim\|f\|_{\LL_{p}^{q'}}  |{\bar t}-s|^{\theta_1},
\end{align*}
by \cref{Thm:est-roughPDE} and $\|f\|_{\mathbb D^\epsilon}\le\|f\|_{\LL^{q'}_p}$.

\textit{Extension to measurable $f$.} Fix $s\le\bar t$ and $A\in\cF_s$, and let $\mu(B):=\mE[\1_A\int_s^{\bar t}\1_B(r,X^n_r)dr]$ for Borel sets $B\subset[0,T]\times\Rd$. We have shown that
\begin{align}\label{est:mu-f}
    \int f\,d\mu\le C|\bar t-s|^{\theta_1}\|f\|_{\LL^{q'}_p}\mP(A)
\end{align}
for nonnegative $f\in C^\infty_c$. If $f\ge0$ is continuous with compact support, its space-time mollifications $f_j$ are smooth, converge to $f$ uniformly and satisfy $\|f_j\|_{\LL^{q'}_p}\le\|f\|_{\LL^{q'}_p}$ by Minkowski's inequality, so \eqref{est:mu-f} holds for $f$. For an open set $U\subset\R\times B_R$, $\1_U$ is the increasing limit of such functions with norms at most $\|\1_U\|_{\LL^{q'}_p}$, so $\mu(U)\le C\|\1_U\|_{\LL^{q_*}_p}\mP(A)$, using \eqref{est:mu-f} with $q'=q_*<\infty$. A Lebesgue null set contained in $[0,T]\times B_{R}$ is contained in open sets $U\subset\R\times B_{R+1}$ of arbitrarily small Lebesgue measure, for which $\|\1_U\|_{\LL^{q_*}_p}$ is arbitrarily small; hence $\mu$ is absolutely continuous with respect to Lebesgue measure. Now let $f\ge0$ be bounded, measurable and compactly supported. Its space-time mollifications $f_j$ converge to $f$ Lebesgue-almost everywhere, hence $\mu$-almost everywhere, and $\|f_j\|_{\LL^{q'}_p}\le\|f\|_{\LL^{q'}_p}$, so Fatou's lemma gives \eqref{est:mu-f} for $f$. By monotone convergence, \eqref{est:mu-f} holds for every measurable $f\ge0$. Since $A\in\cF_s$ is arbitrary, this is \eqref{est:Kry-b}.

We show \eqref{est:Kry-g}. Let $v^{n}\in \sol^0_{2,p}([0,\bar t])$ be the solution, provided by \cref{thm:solution-YPDE} with $\epsilon=0$ and with $a^n$, $b^n$, $\gamma^n$ in place of $a$, $b$, $\gamma$, to
\begin{align}\label{eq:PDE-gamman}
          d_tv^{n}+\Big(\frac{1}{2}a^{n,ij}\partial^2_{ij} v^{n}+b^n\cdot \nabla v^{n}\Big)dt+g\cdot \nabla v^{n} d\gamma_t^n=\rho\, d \Upsilon_t,
 \quad v^{n}({\bar t},\cdot)=0.
\end{align}
By \cref{lem:reg-un} (with $f=0$ and $\bar\gamma=\Upsilon$), $v^n$ satisfies the assumptions of \cref{lem:Ito}. Put $$\caa_t:=\int_0^t \rho(r,X_r^n)d\Upsilon_r \quad \text{  and   } N_t:=\int_0^t(\nabla v^{n}\cdot\sigma^n)(r,X_r^n)dW_r.$$ By \eqref{eq:Ito-classical},
\begin{align*}
    \caa_{\bar t}-\caa_s
   = -v^{n}(s,X_s^n)-(N_{\bar t}-N_s)
   = v^{n}(\bar t,X_s^n)-v^{n}(s,X_s^n)-(N_{\bar t}-N_s).
\end{align*}
By \cref{Thm:est-roughPDE}, $\|v^n\|_{\sol^0_{2,p}}+\|v^n\|_{C^{\theta(0)}([0,\bar t];C_b)}\lesssim\|\Upsilon\|_{C^{\beta_0}}\|\rho\|_{\sol^0_{1,p}}$, so the first difference is bounded by a constant times $\|\Upsilon\|_{C^{\beta_0}}\|\rho\|_{\sol^0_{1,p}}|\bar t-s|^{\theta(0)}$. For the martingale part, we use that $\nabla v^n(\bar t)=0$. Interpolating between $B^0_{p,\infty}$ and $B^1_{p,\infty}$ and using $\nabla v^n\in\sol^0_{1,p}$, we get for $r\le\bar t$
\begin{align*}
    \|\nabla v^n_r\|_{B^{\lambda_1}_{p,\infty}}=\|\nabla v^n_r-\nabla v^n_{\bar t}\|_{B^{\lambda_1}_{p,\infty}}\lesssim(\bar t-r)^{\frac{1-\lambda_1}2}\|v^n\|_{\sol^0_{2,p}},
\end{align*}
and $B^{\lambda_1}_{p,\infty}\hookrightarrow L_\infty$ since $\lambda_1>\frac dp$. Hence, by the conditional It\^o isometry,
\begin{align*}
    \mE\big(|N_{\bar t}-N_s|\big|\cF_s\big)\le\Big(\mE\Big(\int_s^{\bar t}|(\sigma^n)^\top\nabla v^n|^2(r,X^n_r)dr\Big|\cF_s\Big)\Big)^{\frac12}\lesssim(\bar t-s)^{1-\frac{\lambda_1}2}\|\Upsilon\|_{C^{\beta_0}}\|\rho\|_{\sol^0_{1,p}}.
\end{align*}
Combining the two bounds gives \eqref{est:Kry-g}.

Finally, \eqref{est:Kry-b} and \eqref{est:Kry-g} state that the continuous adapted processes $t\mapsto\int_0^t|f(r,X^n_r)|dr$ and $t\mapsto\int_0^t\rho(r,X^n_r)d\Upsilon_r$ satisfy a conditional $L_1$-bound on their increments; the quantitative John--Nirenberg inequality \cite[Theorem 3.4]{le-John} upgrades it to \eqref{est:Krylov}.
\end{proof}

\begin{proposition}
    [Tightness] \label{prop:tightness}
    Let $\mathbb{C}:=C([0,T];\mR^d)$ be endowed with the uniform topology and its Borel $\sigma$-field $\mathcal{B}(\mathbb{C})$, and let $\bP^{n,x}$ denote the law of $X^n(x)$ on $\mC$. Then $(\bP^{n,x})_{n\geq 1} $ is tight. More precisely, for every $m\ge1$ there is a constant $C$ such that
    \begin{align}\label{est:Xn-mom}
        \sup_n\big\|X^n_t-X^n_s\big\|_{L_m(\Omega)}\le C(t-s)^{\frac12},\qquad 0\le s\le t\le T.
    \end{align}
\end{proposition}
\begin{proof}
    We decompose $b=b_1+b_2$ with $b_1:=b\1_{|b|>1}$ and $b_2:=b\1_{|b|\le1}$. Since $q\ge p$, $\|b_1(t)\|_{L_p}^p\le\|b(t)\|_{L_q}^q$, so $b_1\in\LL^\infty_p$, and $|b_2|\le1$; hence $|b^n|\le|b_1|\ast\varphi_n+1$ with $\||b_1|\ast\varphi_n\|_{\LL^\infty_p}\le\|b_1\|_{\LL^\infty_p}$. By \eqref{est:Krylov} (with $f=|b_1|\ast\varphi_n$, $\rho=g\in\sol_{1,p}\subset\sol^0_{1,p}$ and $\Upsilon=\gamma^n$), for every $m\ge2$ and $s\le t\le T$,
    \begin{align*}
         \Big\|\int_s^t b^n(r,X^n_r)dr\Big\|_{L_m(\Omega)} &+
        \Big\|\int_s^t g(r,X^n_r)d\gamma^n_r\Big\|_{L_m(\Omega)}
        \\&\lesssim\big( \|b_1\|_{\LL^\infty_p}+1+\sup_n\|\gamma^n\|_{C^{\beta_0}}\|g\|_{\sol_{1,p}} \big) (t-s)^{\theta_1}.
    \end{align*}
    The stochastic integral $\int \sigma^n (r,X^n_r)dW_r$ is estimated by the BDG inequality, since $\sigma^n$ is bounded uniformly in $n$. This gives \eqref{est:Xn-mom}, and tightness follows from Kolmogorov's criterion.
\end{proof}

\section{Martingale problem and weak solutions for rough SDEs}
\label{sec:mar-weak}


Recall that $\mathbb{C}=C([0,T];\mR^d)$. Let $\mathcal{P}(\mathbb{C})$ denote the set of probability measures on $(\mathbb{C},\mathcal{B}(\mathbb{C}))$, let $(w_t)$ be the coordinate process, $w_t(\omega)=\omega_t$, and let $(\mathcal{B}_t)_{t\in[0,T]}$ be its natural filtration. For $\mathbf{P}\in\mathcal{P}(\mathbb{C})$, $\mathbf{E}^{\mathbf{P}}$ denotes the expectation with respect to $\mathbf{P}$. We write
\begin{align}\label{def:L}
   Lf(t,x):=\partial_tf (t,x)+\frac{1}{2}a^{ij}(t,x)\partial^2_{ij}f (t,x)+b(t,x)\cdot\nabla f(t,x).
\end{align}

\subsection{Definitions and the Krylov--Young property}
\begin{definition}[Young martingale solution]\label{def:mar-sol}
  Let $(s,y)\in[0,T)\times\Rd$. We say that $\bP\in\mathcal{P}(\mC)$ is a {\it{martingale solution}} to \eqref{eq:rough-SDE} starting from $y$ at time $s$ if $\bP(w_r=y \text{ for all } r\in[0,s])=1$ and the following hold:
  \begin{enumerate}
      \item For every $\rho\in\sol^0_{1,p}$ and every sequence $(\gamma^n)\subset C^\infty([0,T]) $ converging to $\gamma$ in $C^{\beta-}$, the integrals $\int_{s}^{\cdot}\rho(r,w_r)d\gamma_r^n$ converge uniformly on $[s,T]$ in $\bP$-probability. We denote the limit by $(D_t(\rho))_{t\in[s,T]}$; it does not depend on the sequence (interlace two sequences).
     \item $\int_{s}^{T}|b(r,w_r)|dr<\infty$, $\bP$-a.s.
  \item For every $f\in  C_c^\infty([0,T]\times\mR^d)$, the process
  \begin{align}\label{eq:martingale-problem}
    M_t^f:=f(t,w_t)-f(s,y)-\int_{s}^{t} Lf(r,w_r)dr
   -D_t(g\cdot\nabla f),\qquad t\in[s,T],
  \end{align}
  is a continuous $(\mathcal{B}_t)$-martingale under $\bP$.
  \end{enumerate}
   The set of all such $\bP$ is denoted by $\mathcal{M}_{L,g,\gamma}(s,y)$, and we write $\mathcal{M}_{L,g,\gamma}(x):=\mathcal{M}_{L,g,\gamma}(0,x)$.
\end{definition}
Note that $g\cdot\nabla f\in\sol_{1,p}\subset\sol^0_{1,p}$ for $f\in C^\infty_c([0,T]\times\Rd)$, so $D(g\cdot\nabla f)$ in \eqref{eq:martingale-problem} is well defined.

\begin{definition}[Weak solution]\label{def:weak-sol}
Let $(\tilde\Omega,\tilde{\mathcal{F}},(\tilde{\mathcal{F}}_t)_{t\in[0,T]},\tilde{\mathcal{P}})$ be a filtered probability space and $(\tilde X,\tilde W)$ a pair of continuous $(\tilde{\mathcal{F}}_t)$-adapted processes. We call $(\tilde\Omega,\tilde{\mathcal{F}},(\tilde{\mathcal{F}}_t),\tilde{\mathcal{P}};\tilde X,\tilde W)$ a {\it weak solution} of  \eqref{eq:rough-SDE} starting from $x$ if
\begin{enumerate}
  \item $\tilde W$ is a $d$-dimensional $(\tilde{\mathcal{F}}_t)$-Brownian motion;
  \item for every sequence $(\gamma^n)\subset C^\infty([0,T]) $ converging to $\gamma$ in $C^{\beta-}$, the integrals $\int_{0}^{\cdot}g(r,\tilde X_r)d\gamma_r^n$ converge uniformly on $[0,T]$ in $\tilde{\mathcal{P}}$-probability; we denote the limit by $G$;
\item $\int_0^T  |b(t,\tilde X_t)|dt<\infty$, $\tilde{\mathcal{P}}$-a.s.;
  \item for every $t\in[0,T]$, $\tilde{\mathcal{P}}$-a.s.,
  \begin{align*}
       \tilde X_t=x+\int_0^tb(r,\tilde X_r)dr+G_t+\int_0^t\sigma(r,\tilde X_r)d\tilde W_r.
  \end{align*}
\end{enumerate}
\end{definition}

\begin{definition}
    [Krylov--Young property]\label{lem:Y-kry} Let $(s,y)\in[0,T)\times\Rd$. We say that $\bP\in\mathcal{P}(\mC)$ satisfies the {\it {\ky}  property } (from time $s$) if there exist finite constants $\theta_{KY}>\frac12$ and $C_{KY}$ such that for all $s\le s'\le t\le T$, $f\in \LL^\infty_p$, $\rho\in\sol^0_{1,p}$ and $\Upsilon\in C^\infty([0,T])$,
 \begin{align}
\label{est:Kry-b-fin-M}
\bE^{\bP}\Big(\int_{s'}^{t}|f(r,w_r)|dr+\big|\int_{s'}^{t}\rho(r,w_r)d\Upsilon_r\big|\,\Big|\,\cB_{s'}\Big)\le C_{KY} |t-s'|^{\theta_{KY}}    \big(\Vert f\Vert_{\LL_{p}^{\infty}}+\Vert \Upsilon\Vert_{C^{\beta_0}} \Vert \rho\Vert_{\sol^0_{1,p}}\big).
 \end{align}
Likewise, a weak solution $X$ satisfies the {\it {\ky} property} if \eqref{est:Kry-b-fin-M} holds with $\bE^{\bP}(\cdot|\cB_{s'})$ and $w$ replaced by $\mE^{\tilde\cP}(\cdot|\tilde\cF_{s'})$ and $\tilde X$ (and $s=0$):
 \begin{align}
\label{est:Kry-b-fin}
\mE^{\tilde\cP}\Big(\int_{s'}^{t}|f(r,\tilde X_r)|dr+\big|\int_{s'}^{t}\rho(r,\tilde X_r)d\Upsilon_r\big|
\,\Big|\,\tilde\cF_{s'}\Big)\le C_{KY} |t-s'|^{\theta_{KY}}    \big( \Vert f\Vert_{\LL_{p}^{\infty}}+\Vert\Upsilon\Vert_{C^{\beta_0}} \Vert \rho\Vert_{\sol^0_{1,p}}\big).
 \end{align}
\end{definition}
\begin{remark}\label{rem:KY-class}
    \cref{lem:Krylov} shows that the approximating solutions $X^n$ satisfy \eqref{est:Kry-b-fin} uniformly in $n$, and \cref{thm:existence} below shows that this property passes to the limit. By the John--Nirenberg inequality, the \ky property implies that $\int_0^t b(r,w_r)dr$ and $D_t(g)$ have $\theta_{KY}$-H\"older moments (see \cref{lem:KY-cons}), which is the type of a priori regularity expected from a solution of \eqref{eq:rough-SDE}. In this sense the \ky property describes the natural class of solutions, similarly to \cite{Zhangzhao, GP, HZ}; note, however, that here the singularity is caused by the time variable through $\gamma$.
\end{remark}


\begin{lemma}\label{lem:KY-cons}
    Let $\bP\in\mathcal{P}(\mC)$ satisfy the {\ky} property from time $s$, with constants $\theta_{KY},C_{KY}$. Then:
    \begin{enumerate}[(i)]
        \item For every $m\ge1$ there is $C$ depending on $m,C_{KY},\theta_{KY}$ such that for all $s\le s'\le t\le T$, $f$, $\rho$, $\Upsilon$ as in \cref{lem:Y-kry},
  \begin{align}\label{est:Krylov-fin-M}
  \Big\| \int_{s'}^{t}|f(r,w_r)|dr\Big\|_{L_m(\bP)}+\Big\| \int_{s'}^{t}\rho(r,w_r)d\Upsilon_r\Big\|_{L_m(\bP)}\leq C|t-s'|^{\theta_{KY}}\big(\Vert f\Vert_{\LL_{p}^\infty}+\|\Upsilon\|_{C^{\beta_0}}\Vert \rho\Vert_{\sol^0_{1,p}}\big).
  \end{align}
  \item Condition (1) of \cref{def:mar-sol} holds. Moreover, the convergence holds in $L_m(\bP;C([s,T]))$ for every $m\ge1$, $\rho\mapsto D(\rho)$ is linear, and for all $s\le s'\le t\le T$, $\bP$-a.s.,
  \begin{align}\label{est:D-KY}\begin{split}
      &\bE^\bP\big(|D_t(\rho)-D_{s'}(\rho)|\,\big|\,\cB_{s'}\big)\le C_{KY}|t-s'|^{\theta_{KY}}\|\gamma\|_{C^{\beta_0}}\|\rho\|_{\sol^0_{1,p}},
      \\
      &\big\|D_t(\rho)-D_{s'}(\rho)\big\|_{L_m(\bP)}\le C|t-s'|^{\theta_{KY}}\|\gamma\|_{C^{\beta_0}}\|\rho\|_{\sol^0_{1,p}}.
  \end{split}\end{align}
  \item Condition (2) of \cref{def:mar-sol} holds, and $\bE^\bP\int_s^T|b(r,w_r)|dr<\infty$.
  \item The occupation measure $\mu(A):=\bE^\bP\int_s^T\1_A(r,w_r)dr$ on $[s,T]\times\Rd$ has a density $\mathfrak m$ with $\int_s^T\|\mathfrak m(r,\cdot)\|_{L_{p'}}dr\le C_{KY}T^{\theta_{KY}}$, $p'=\frac p{p-1}$. Consequently, if $(h_k)$ is bounded in $\LL^\infty_p$ and $\|h_k(r)\|_{L_p}\to0$ for a.e.\ $r$, then $\bE^\bP\int_s^T|h_k(r,w_r)|dr\to0$.
  \item If $(\rho_k)$ is bounded in $\sol^0_{1,p}$ and $\|\rho_k(r)\|_{L_p}\to0$ for a.e.\ $r$, then $\sup_{t}\bE^\bP|D_t(\rho_k)|\to0$.
    \end{enumerate}
    The same statements hold for a weak solution $X$ with the \ky property; in this case, in addition,
    \begin{align}
        \label{est:Lm-X-mom}
        \Big\| \sup_{r\in[s,t]}|X_r-X_s|\Big\|_{L_m(\tilde\cP)}\lesssim |t-s|^\frac{1}{2},\qquad 0\le s\le t\le T,\ m\ge1 .
    \end{align}
\end{lemma}
\begin{proof}
 (i) The conditional bound \eqref{est:Kry-b-fin-M} for the continuous adapted processes $t\mapsto\int_s^t|f(r,w_r)|dr$ and $t\mapsto\int_s^t\rho(r,w_r)d\Upsilon_r$ yields \eqref{est:Krylov-fin-M} by the John--Nirenberg inequality \cite[Theorem 3.4]{le-John}.

 (ii) Let $(\gamma^n)$ converge to $\gamma$ in $C^{\beta-}$, so that $\gamma^n\to\gamma$ in $C^{\beta_0}$. For $Z_t:=\int_s^t\rho(r,w_r)d(\gamma^n-\gamma^k)_r$, (i) gives $$\|Z_t-Z_{s'}\|_{L_m(\bP)}\le C|t-s'|^{\theta_{KY}}\|\gamma^n-\gamma^k\|_{C^{\beta_0}}\|\rho\|_{\sol^0_{1,p}},\quad \forall m.$$
 By Kolmogorov's criterion with $m\theta_{KY}>1$, $$\|\sup_{t\in[s,T]}|Z_t|\|_{L_m(\bP)}\lesssim\|\gamma^n-\gamma^k\|_{C^{\beta_0}}\|\rho\|_{\sol^0_{1,p}}\to0  \text{ as } n,k\to\infty.$$  Hence $\int_s^\cdot\rho(r,w_r)d\gamma^n_r$ converges in $L_m(\bP;C([s,T]))$, linearity of $D$ is clear, and \eqref{est:D-KY} follows from \eqref{est:Kry-b-fin-M}, (i) and Fatou's lemma (for the conditional bound, along an a.s.\ convergent subsequence), using $\|\gamma^n\|_{C^{\beta_0}}\to\|\gamma\|_{C^{\beta_0}}$.

 (iii) Decompose $b=b_1+b_2$ as in the proof of \cref{prop:tightness} and apply (i) with $f=b_1$; $b_2$ is bounded.

 (iv) By \eqref{est:Kry-b-fin-M}, $\mu(|h|)\le C_{KY}T^{\theta_{KY}}\|h\|_{\LL^\infty_p}$ for all measurable $h$. Taking $h=\1_N$ for a Lebesgue null set $N$ shows that $\mu$ is absolutely continuous; let $\mathfrak m$ be its density. Testing with $h(r,\cdot):=\mathfrak m_M(r,\cdot)^{p'-1}/\|\mathfrak m_M(r,\cdot)\|_{L_{p'}}^{p'-1}$, where $\mathfrak m_M:=(\mathfrak m\wedge M)\1_{B_M}$, gives $\int\|\mathfrak m_M(r)\|_{L_{p'}}dr\le C_{KY}T^{\theta_{KY}}$, and letting $M\to\infty$ yields the bound on $\mathfrak m$. Then by dominated convergence $$\bE^\bP\int|h_k(r,w_r)|dr\le\int\|h_k(r)\|_{L_p}\|\mathfrak m(r)\|_{L_{p'}}dr\to0.$$ 

 (v) Let $(\gamma^m)$ converge to $\gamma$ in $C^{\beta-}$. For each $m$, by \eqref{est:Kry-b-fin-M} with $\Upsilon=\gamma-\gamma^m$ (in the limit sense of (ii)) and by (iv),
 \begin{align*}
          \bE^\bP|D_t(\rho_k)|&\le\bE^\bP\Big|D_t(\rho_k)-\int_s^t\rho_k(r,w_r)d\gamma^m_r\Big|+\sup|(\gamma^m)'|\,\bE^\bP\int_s^T|\rho_k(r,w_r)|dr
     \\&\le C\|\gamma-\gamma^m\|_{C^{\beta_0}}\sup_k\|\rho_k\|_{\sol^0_{1,p}}+o_{k\to\infty}(1).
 \end{align*}
 Letting first $k\to\infty$ and then $m\to\infty$ gives the claim.

 For a weak solution $X$ with the {\ky} property, (i)--(iii) are proved in the same way. Since $G$ is the limit of $\int_0^\cdot g(r,X_r)d\gamma^n_r$, (ii) with $\rho=g$ gives $$\|G_t-G_s\|_{L_m}\lesssim|t-s|^{\theta_{KY}},\quad  \text{and also }  \|\int_s^t|b(r,X_r)|dr\|_{L_m}\lesssim|t-s|^{\theta_{KY}}$$ by (i) and (iii). Since $\theta_{KY}>\frac12$, the Kolmogrov continuity criteria and the BDG inequality for $\int\sigma(r,X_r)dW_r$ give \eqref{est:Lm-X-mom}.
\end{proof}

\begin{lemma}
    [Young--It\^o formula]\label{lem:Young-Ito} Let $X$ be a weak solution of \eqref{eq:rough-SDE} with the {\ky} property, and let $f\in C^\infty([0,T]\times\mathbb{R}^d)$ such that $f$ and all its derivatives have polynomial growth in $x$, uniformly in $t$. Then $ \tilde\cP$-a.s., for all $t\in[0,T]$,
    \begin{align}
        \label{eq:Ito}
        \begin{split}f(t,X_t)=&f(0,X_0)+\int_0^t(\partial_tf+\nabla f\cdot b)(r,X_r)dr+\int_0^t\nabla f(r,X_r)dG_r\\&+\int_0^t(\nabla f\cdot \sigma)(r,X_r)dW_r+\frac{1}{2}\int_0^t\big(a^{ij}\partial_{ij}^2f\big)(r,X_r)dr,\end{split}
    \end{align}
    where $\int_0^t\nabla f(r,X_r)dG_r$ is a Young integral. Moreover, $\tilde\cP$-a.s.\ for all $t$,
    \begin{align}\label{eq:ID}
        \int_0^t\nabla f(r,X_r)dG_r=\lim_{n\to\infty}\int_0^t(g\cdot\nabla f)(r,X_r)d\gamma^n_r,
    \end{align}
    the limit being taken in probability, for any sequence $(\gamma^n)\subset C^\infty([0,T])$ converging to $\gamma$ in $C^{\beta-}$.
\end{lemma}
\begin{proof}
  Fix $\delta>0$ so small that $\frac12-\delta+\theta_{KY}-2\delta>1$. By \cref{lem:KY-cons} and Kolmogorov's criterion, $$X\in C^{\frac12-\delta}([0,T])\qquad \text{ and } \qquad G\in C^{\theta_{KY}-\delta}([0,T])\qquad a.s.$$  with moments of all orders; moreover by \eqref{est:Krylov-fin-M} with $\Upsilon=\gamma^n$, $$D^n:=\int_0^\cdot g(r,X_r)d\gamma^n_r \text{ is bounded in } L_m(\tilde\cP;C^{\theta_{KY}-\delta}) \text{ uniformly in } n.$$  Hence $\Phi_r:=\nabla f(r,X_r)$ is $(\frac12-\delta)$-H\"older, which implies the Young integral $\int\Phi dG$ is well defined.  Moreover since $D^n\to G$ uniformly in probability, interpolation gives $D^n\to G$ in $C^{\theta_{KY}-2\delta}$ in probability. By the continuity of the Young integral, $$\int_0^t\Phi dD^n\to\int_0^t\Phi dG  \text{ uniformly in probability.}$$ Since $D^n$ is absolutely continuous with derivative $g(r,X_r)(\gamma^n)'_r$, we have $$\int_0^t\Phi dD^n=\int_0^t(g\cdot\nabla f)(r,X_r)d\gamma^n_r$$ which proves \eqref{eq:ID}.

  To prove \eqref{eq:Ito}, let $\Pi$ be a partition of $[0,t]$ and, for $[v,u]\in\Pi$, write $$\delta X_{v,u}=X_u-X_v=V_{v,u}+I_{v,u}$$ with $$V_{v,u}:=\int_v^ub(r,X_r)dr+G_u-G_v \text{ and } I_{v,u}:=\int_v^u\sigma(r,X_r)dW_r.$$ By \cref{lem:KY-cons}, $$\|V_{v,u}\|_{L_4}\lesssim(u-v)^{\theta_{KY}}, \text{  and  }  \|I_{v,u}\|_{L_4}+\|\sup_{r\in[v,u]}|X_r-X_v|\|_{L_4}\lesssim(u-v)^{\frac12}.$$ Since $f$ and its derivatives have polynomial growth and $X$ has moments of all orders by \cref{lem:KY-cons}, Taylor's formula gives 
    \begin{align*}
                f(u,X_u)-f(v,X_v)
        &= \partial_tf(v,X_v)(u-v)+\nabla f(v,X_v)\cdot\delta X_{v,u}\\&\quad+\frac12\nabla^2 f(v,X_v)[\delta X_{v,u},\delta X_{v,u}]+R_{v,u},
    \end{align*}
    with $|R_{v,u}|\lesssim(1+|X_v|+|\delta X_{v,u}|)^m\big((u-v)^2+(u-v)|\delta X_{v,u}|+|\delta X_{v,u}|^3\big)$, where $m$ is the largest polynomial growth order among the finitely many derivatives $\partial_t^2f,\partial_t\nabla f,\nabla^3f$ and $\partial_tf,\nabla^2f$ involved. Since $\sup_t|X_t|$ has moments of all orders by \eqref{est:Lm-X-mom} and $\|\delta X_{v,u}\|_{L_k}\lesssim(u-v)^{\frac12}$ for all $k$ (by the bounds on $I_{v,u}$, $V_{v,u}$ above), H\"older's inequality gives $\mE|R_{v,u}|\lesssim(u-v)^{\frac32}$, hence $\sum_{\Pi}\mE|R_{v,u}|\lesssim|\Pi|^{\frac12}\to0$. We now let $|\Pi|\to0$ in each sum. By definition of Riemann integral, a.s., 
    \begin{align}
        \label{est:Ito1}
        \lim_{|\Pi|\to0}\sum_{[v,u]\in\Pi} \partial_tf(v,X_v)(u-v)=\int_0^t\partial_tf(r,X_r)dr.
    \end{align}
    Observe that 
    \begin{align*}
        \sum_\Pi\nabla f(v,X_v)\cdot\int_v^ub(r,X_r)dr-&\int_0^t(\nabla f\cdot b)(r,X_r)dr\\&=\sum_\Pi\int_v^u(\nabla f(v,X_v)-\nabla f(r,X_r))\cdot b(r,X_r)dr,
    \end{align*}
  so we have
    \begin{align}
      \Big\|\sum_\Pi\int_v^u(\nabla f(v,X_v)-\nabla f(r,X_r))\cdot b(r,X_r)dr\Big\|_{L_1}\lesssim \sum_\Pi(u-v)^{\frac12+\theta_{KY}}\to0,  \label{est:Ito2}
    \end{align}
  where the bound follows from H\"older's inequality, since $|\nabla f(v,X_v)-\nabla f(r,X_r)|\lesssim(1+\sup_s|X_s|)^m\big((r-v)+|X_r-X_v|\big)$, $\|\sup_{r\in[v,u]}|X_r-X_v|\|_{L_k}\lesssim(u-v)^{\frac12}$ and $\|\int_v^u|b(r,X_r)|dr\|_{L_k}\lesssim(u-v)^{\theta_{KY}}$ for all $k$ (\cref{lem:KY-cons}).
   Further by the definition of the Young integral we also have
   \begin{align}
       \label{est:Ito3}\sum_\Pi\nabla f(v,X_v)\cdot(G_u-G_v)\to\int_0^t\nabla f(r,X_r)dG_r \quad\qquad a.s.
   \end{align}
   By It\^o's isometry, the sum converges in $L_2$ to the stochastic integral provided $\mE\int_0^t|\nabla f(r,X_r)-\nabla f(\pi(r),X_{\pi(r)})|^2dr\to0$, where $\pi(r):=v$ for $r\in[v,u)$. This holds by dominated convergence on $\Omega\times[0,t]$: the integrand tends to $0$ pointwise by continuity of $r\mapsto\nabla f(r,X_r)$, and it is dominated by $C(1+\sup_s|X_s|)^{2m}\in L_1$ since $X$ has moments of all orders. Hence
   \begin{align}
       \label{est:Ito4}\sum_\Pi\nabla f(v,X_v)\cdot I_{v,u}\to\int_0^t(\nabla f\cdot\sigma)(r,X_r)dW_r \quad \quad \text{ in } L_2.
   \end{align}
In the last term $\frac12\nabla^2f(v,X_v)[\delta X_{v,u},\delta X_{v,u}]$, by decomposing
     \begin{align*}
        [\delta X_{v,u},\delta X_{v,u}]=(V_{v,u})^{\otimes2}&+\big(\int_v^u\sigma(r,X_r)dW_r\big)^{\otimes2}+V_{v,u}\otimes\int_v^u\sigma(r,X_r)dW_r\\&+\int_v^u\sigma(r,X_r)dW_r\otimes V_{v,u},
    \end{align*}
we can see the terms containing $V_{v,u}$ have $L_1$-norm $\lesssim(u-v)^{\theta_{KY}+\frac12}$, so their sum tends to $0$ in $L_1$; for the remaining term, $$Z_{v,u}:=\frac12\partial^2_{ij}f(v,X_v)(I^i_{v,u}I^j_{v,u}-\int_v^ua^{ij}(r,X_r)dr)$$ satisfies $\mE(Z_{v,u}|\tilde\cF_v)=0$ and $\mE|Z_{v,u}|^2\lesssim(u-v)^2$, so $$\mE|\sum_\Pi Z_{v,u}|^2=\sum_\Pi\mE|Z_{v,u}|^2\to0$$ while \begin{align}
    \label{est:Ito6}
    \sum_\Pi\frac12\partial^2_{ij}f(v,X_v)\int_v^ua^{ij}(r,X_r)dr\to\frac12\int_0^t(a^{ij}\partial^2_{ij}f)(r,X_r)dr \qquad a.s.
\end{align}
 
     Since $\sum_\Pi(f(u,X_u)-f(v,X_v))=f(t,X_t)-f(0,X_0)$, putting \eqref{est:Ito1}-\eqref{est:Ito6} together proves \eqref{eq:Ito} for fixed $t$, and hence for all $t$ by continuity (along a sequence of partitions with $|\Pi|\to0$ on which the $L_1$/$L_2$ convergences hold a.s.).
\end{proof}

\begin{corollary}\label{prop:w-m}
    Let $(\tilde\Omega,\tilde{\mathcal{F}},(\tilde{\mathcal{F}}_t),\tilde{\mathcal{P}};X,W)$ be a weak solution to \eqref{eq:rough-SDE} starting from $x$ with the \ky property. Then $\bP:=\tilde\cP\circ X^{-1}\in\cM_{L,g,\gamma}(x)$, and $\bP$ satisfies the \ky property with the same constants.
\end{corollary}
\begin{proof}
    Since $\sigma(X_r:r\le s')\subset\tilde\cF_{s'}$, \eqref{est:Kry-b-fin} implies \eqref{est:Kry-b-fin-M} for $\bP$ by the tower property. Conditions (1) and (2) of \cref{def:mar-sol} then follow from \cref{lem:KY-cons}, and under $\bP$ the process $D(\rho)$ is the image of the corresponding limit for $X$. For $f\in C^\infty_c([0,T]\times\Rd)$, \cref{lem:Young-Ito} shows that $M^f$ evaluated along $X$ equals $\int_0^\cdot(\nabla f\cdot\sigma)(r,X_r)dW_r$, which is a martingale with respect to $(\tilde\cF_t)$, and hence with respect to the natural filtration of $X$. Thus $M^f$ is a $(\cB_t)$-martingale under $\bP$.
\end{proof}

We also need the martingale property for test functions of Sobolev type. Here and below, when \eqref{est:Kry-b-fin-M} or \cref{lem:KY-cons} is applied to a function defined on $[0,t_0]$, we extend it to $[0,T]$ by $v(t):=v(t_0)$ for $t>t_0$; this does not increase its norms.
\begin{lemma}\label{lem:MP-Sobolev}
    Let $\bP\in\cM_{L,g,\gamma}(s,y)$ satisfy the \ky property and let $t_0\in(s,T]$. Let $u\in\sol^0_{2,p}([0,t_0])$ with $u(t_0)\in C^\infty_c(\Rd)$, and assume that $$\varrho:=\partial_tu+\frac12a^{ij}\partial^2_{ij}u  \in \LL^\infty_p([0,t_0]),$$ i.e. \ $u_t-u_r+\int_r^t\frac12a^{ij}\partial^2_{ij}u_vdv=\int_r^t\varrho_vdv$ for $r\le t\le t_0$. Then
    \begin{align*}
        M^u_t:=u(t,w_t)-u(s,y)-\int_s^t\big(\varrho+b\cdot\nabla u\big)(r,w_r)dr-D_t(g\cdot\nabla u),\qquad t\in[s,t_0],
    \end{align*}
    is a $(\cB_t)$-martingale under $\bP$. 
\end{lemma}
\begin{proof} Our proof distributes into following two steps. 

  {\it Step 1: spatially smooth $u$.} Assume first that, in addition, $u$ is smooth in $x$ with all derivatives in $C([0,t_0];B^m_{p,\infty})$ for every $m$, and that $\partial_tu\in L^\infty([0,t_0];B^m_{p,\infty})$ for every $m$.  
  Extend $u$ to $[-1,t_0+1]$ by reflection in time, mollify in time with a kernel of width $\eta$, multiply by a smooth cutoff in time equal to $1$ on $[0,t_0]$ and by $\chi(x/R)$ with $\chi\in C^\infty_c(\Rd)$, $\chi=1$ near $0$. 
  This gives $u^{\eta,R}\in C^\infty_c([0,T]\times\Rd)$, so $M^{u^{\eta,R}}$ is a martingale by \cref{def:mar-sol}. As $\eta\to0$ and $R\to\infty$: $u^{\eta,R}\to u$ uniformly together with first and second spatial derivatives; $\partial_tu^{\eta,R}\to\partial_tu$ in $L^1([0,t_0];C_b(\Rd))$ (because $B^m_{p,\infty}\hookrightarrow C_b$ for $m>d/p$); and $\nabla u^{\eta,R}(r)\to\nabla u(r)$ in $L_p$ for every $r$, with $\sup\|g\cdot\nabla u^{\eta,R}\|_{\sol^0_{1,p}}<\infty$. Using \cref{lem:KY-cons} (iii) for the drift term and (v) for $D(g\cdot\nabla u^{\eta,R})$, we obtain $M^{u^{\eta,R}}_t\to M^u_t$ in $L_1(\bP)$, so $M^u$ is a martingale.

    {\it Step 2: general $u$.} Let $u^{(m)}:=u\ast\varphi_m$ (mollification in $x$ only). Then $u^{(m)}$ satisfies the assumptions of Step 1 (with $\partial_tu^{(m)}=(\varrho-\frac12a^{ij}\partial^2_{ij}u)\ast\varphi_m$), and
    \begin{align*}
       \varrho^{(m)}:=\partial_tu^{(m)}+\frac12a^{ij}\partial^2_{ij}u^{(m)}=\varrho\ast\varphi_m+\frac12[a^{ij},\varphi_m]\partial^2_{ij}u .
    \end{align*}
    As $m\to\infty$: $u^{(m)}\to u$ and $\nabla u^{(m)}\to\nabla u$ uniformly, since $u$ and $\nabla u$ are bounded and uniformly H\"older continuous in $x$, uniformly in $t$; $\|\varrho\ast\varphi_m(r)-\varrho(r)\|_{L_p}\to0$ for every $r$, with $\sup_m\|\varrho\ast\varphi_m\|_{\LL^\infty_p}\le\|\varrho\|_{\LL^\infty_p}$; and $[a^{ij},\varphi_m]\partial^2_{ij}u\to0$ in $\LL^\infty_p$ by \eqref{est:commutator} with $\nu=0$. Moreover $(g\cdot\nabla u^{(m)})$ is bounded in $\sol^0_{1,p}$ by \cref{lem.prodS}, and $g\cdot\nabla(u^{(m)}-u)(r)\to0$ in $L_p$ for every $r$. By \cref{lem:KY-cons} (iii)--(v), $M^{u^{(m)}}_t\to M^u_t$ in $L_1(\bP)$ for every $t$, and $M^u$ is a martingale.
\end{proof}

\subsection{Existence}
\begin{theorem}[Existence]\label{thm:existence} Let $\theta_1$ and $C$ be the constants of \cref{lem:Krylov}.
\begin{enumerate}
    \item (Martingale solution) Every accumulation point $\bP^{x}$ of  $(\bP^{n,x})_n$ from \cref{prop:tightness} belongs to $\cM_{L,g,\gamma}(x)$ and satisfies the \ky property with $\theta_{KY}=\theta_1$ and $C_{KY}=C$. Moreover, $\bP^x$ satisfies \eqref{est:Kry-b-fin-M} also for $f\in\LL^{q'}_p$, $q'\in[q_*,\infty)$, with $\|f\|_{\LL^{q'}_p}$ on the right-hand side.
  \item (Weak solution) There exists a weak solution $(\tilde\Omega,\tilde{\mathcal{F}},(\tilde{\mathcal{F}}_t),\tilde{\mathcal{P}};X,W)$ to \eqref{eq:rough-SDE} starting from $x$ with the \ky property, whose law is an accumulation point of $(\bP^{n,x})_n$.
\end{enumerate}
\end{theorem}
\begin{proof}
    (1) Let $\bP^{n_k,x}\to\bP^x$ weakly. Fix $0\le s\le t\le T$ and a nonnegative $H\in C_b(\mC)$ which depends only on $(w_r)_{r\le s}$.

    \textit{{\ky} property.} For $\rho\in\sol^0_{1,p}$ and $\Upsilon\in C^\infty([0,T])$, the functional $w\mapsto|\int_s^t\rho(r,w_r)\Upsilon'_rdr|$ is bounded and continuous on $\mC$, because $\rho(r,\cdot)$ is continuous and $\sup_r\|\rho(r,\cdot)\|_{L_\infty}<\infty$ by \eqref{eq:S1e-emb}. The same holds for $w\mapsto\int_s^t|f(r,w_r)|dr$ when $f\in C_c([0,T]\times\Rd)$. Hence, by \eqref{est:Kry-b} and \eqref{est:Kry-g}, for $q'\in[q_*,\infty]$,
    \begin{align}\label{est:KY-limit}
        \bE^{\bP^x}\Big[\Big(\int_s^t|f(r,w_r)|dr+\big|\int_s^t\rho(r,w_r)d\Upsilon_r\big|\Big)H\Big]\le C|t-s|^{\theta_1}\big(\|f\|_{\LL^{q'}_p}+\|\Upsilon\|_{C^{\beta_0}}\|\rho\|_{\sol^0_{1,p}}\big)\bE^{\bP^x}H
    \end{align}
    for $f\in C_c([0,T]\times\Rd)$. To extend this to measurable $f$, let $f\ge0$ be bounded with compact support, and let $f_j:=f\ast\psi_j$ be space-time mollifications. Then $f_j\in C_c$, $\|f_j\|_{\LL^{q'}_p}\le\|f\|_{\LL^{q'}_p}$ by Minkowski's inequality, and $f_j\to f$ a.e. The measure $\mu_H(A):=\bE^{\bP^x}[H\int_s^t\1_A(r,w_r)dr]$ is absolutely continuous with respect to $drdx$, by \eqref{est:KY-limit} with $q'=q_*$ and approximation of null sets by open sets. Hence Fatou's lemma gives \eqref{est:KY-limit} for $f$, and monotone convergence extends it to all measurable $f$. Since such $H$ generate $\cB_s$, a monotone class argument gives \eqref{est:Kry-b-fin-M} for $\bP^x$, also in the $\LL^{q'}_p$-version. By \cref{lem:KY-cons}, conditions (1) and (2) of \cref{def:mar-sol} hold.

    \textit{Martingale property.} Let $f\in C_c^\infty([0,T]\times\mR^{d})$. By It\^o's formula for $X^n$ (\cref{lem:Ito}), the process
        \begin{align*}
        M^{n,f}_t:=f(t,w_t)-f(0,x)&-\int_0^t\big(Lf+(b^n-b)\cdot\nabla f+\frac12(a^{n,ij}-a^{ij})\partial^2_{ij}f\big)(r,w_r)dr
        \\&-\int_0^t(g\cdot\nabla f)(r,w_r)d\gamma^n_r
    \end{align*}
    is a $(\cB_t)$-martingale under $\bP^{n,x}$, i.e.\ $\bE^{\bP^{n,x}}[(M^{n,f}_t-M^{n,f}_s)H]=0$. We pass to the limit along $n_k$ in each term. The functionals $$w\mapsto f(t,w_t) \text{  and } w\mapsto\int_s^t(\partial_tf+\frac12a^{ij}\partial^2_{ij}f)(r,w_r)dr$$ are bounded and continuous. Next, $b\cdot\nabla f\in\LL^{q_*}_p$; approximating it in $\LL^{q_*}_p$ by $h_j\in C_c$, and using \eqref{est:Kry-b} with $q'=q_*$ for $\bP^{n,x}$ and its $\LL^{q_*}_p$-version for $\bP^x$, we get $$\bE^{\bP^{n_k,x}}[H\int_s^t(b\cdot\nabla f)(r,w_r)dr]\to\bE^{\bP^{x}}[H\int_s^t(b\cdot\nabla f)(r,w_r)dr];$$ moreover, $$\bE^{\bP^{n,x}}\int_s^t|(b^n-b)\cdot\nabla f|(r,w_r)dr\lesssim\|(b^n-b)\cdot\nabla f\|_{\LL^{q_*}_p}\to0$$ by \eqref{est:Kry-b} and \eqref{eq:bn-conv}, while $(a^n-a)^{ij}\partial^2_{ij}f\to0$ uniformly, since $\sigma^n\to\sigma$ uniformly. Finally, put $\rho:=g\cdot\nabla f\in\sol^0_{1,p}$. For fixed $m$, the functional $w\mapsto H(w)\int_s^t\rho(r,w_r)d\gamma^m_r$ is bounded and continuous, while by \eqref{est:Kry-g} and \cref{lem:KY-cons} (ii),
    \begin{align*}
        &\sup_{n\ge m}\bE^{\bP^{n,x}}\Big|\int_s^t\rho(r,w_r)d(\gamma^n_r-\gamma^m_r)\Big|+\bE^{\bP^x}\Big|D_t(\rho)-D_s(\rho)-\int_s^t\rho(r,w_r)d\gamma^m_r\Big|
        \\&\qquad\lesssim\sup_{n\ge m}\|\gamma^n-\gamma^m\|_{C^{\beta_0}}\|\rho\|_{\sol^0_{1,p}},
    \end{align*}
    which tends to $0$ as $m\to\infty$. Therefore $$\bE^{\bP^x}[(M^f_t-M^f_s)H]=\lim_k\bE^{\bP^{n_k,x}}[(M^{n_k,f}_t-M^{n_k,f}_s)H]=0.$$ Since $M^f$ is integrable by \cref{lem:KY-cons}, a monotone class argument shows that $M^f$ is a $(\cB_t)$-martingale under $\bP^x$.

    (2) Each $X^n$ is defined together with a Brownian motion $W^n$ (on a probability space that may depend on $n$). By \cref{prop:tightness}, the laws of $(X^n,W^n)$ on $C([0,T];\R^{2d})$ are tight. By Skorokhod's representation theorem, along a subsequence (not relabelled) there are processes $(\tilde X^n,\tilde W^n)$ and $(X,W)$ on a common probability space $(\tilde\Omega,\tilde\cF,\tilde\cP)$ such that $(\tilde X^n,\tilde W^n)$ has the law of $(X^n,W^n)$ and $(\tilde X^n,\tilde W^n)\to(X,W)$ uniformly on $[0,T]$, $\tilde\cP$-a.s. Let $(\tilde\cF_t)$ be the usual augmentation of the filtration generated by $(X,W)$. Then: 
    \begin{enumerate}[(a)]
        \item $\tilde W^n$ is a Brownian motion with respect to the filtration generated by $(\tilde X^n,\tilde W^n)$, and $\tilde X^n$ solves \eqref{eq:SDE-m} with $\tilde W^n$, since these properties depend only on the joint law. Passing to the limit in the finite-dimensional distributions shows that $W$ is an $(\tilde\cF_t)$-Brownian motion.
        \item $X$ has the \ky property with respect to $(\tilde\cF_t)$, with the constants of \cref{lem:Krylov}, also in the $\LL^{q_*}_p$-version: this follows exactly as in (1), with test functionals $H$ of $(X_r,W_r)_{r\le s}$.
        \item Fix $t$. By \eqref{est:Krylov} with $\Upsilon=\gamma^n-\gamma^m$ and $\rho=g$, $\sup_{n\ge m}\mE|\int_0^tg(r,\tilde X^n_r)d(\gamma^n-\gamma^m)_r|\lesssim\sup_{n\ge m}\|\gamma^n-\gamma^m\|_{C^{\beta_0}}\to0$ as $m\to\infty$. For fixed $m$, $\int_0^tg(r,\tilde X^n_r)d\gamma^m_r\to\int_0^tg(r,X_r)d\gamma^m_r$ a.s., by dominated convergence. Together with \cref{lem:KY-cons} (ii) for $X$, this shows that condition (2) of \cref{def:weak-sol} holds and that $\int_0^tg(r,\tilde X^n_r)d\gamma^n_r\to G_t$ in $L_1(\tilde\cP)$.
        \item Fix $t$. Write $b=b_1+b_2$ as in \cref{prop:tightness}, and let $\eta>0$. Choose $R$ with $\sup_n\tilde\cP(\sup_r|\tilde X^n_r|>R)+\tilde\cP(\sup_r|X_r|>R)<\eta$, and $h\in C_c([0,T]\times\Rd)$ with $\|b_1+b_2\1_{B_{R+1}}-h\|_{\LL^{q_*}_p}<\eta$. By \eqref{est:Kry-b} with $q'=q_*$ for $\tilde X^n$ and the $\LL^{q_*}_p$-version of \eqref{est:Kry-b-fin} for $X$, the $L_1$-distances between $\int_0^tb(r,\tilde X^n_r)dr$ and $\int_0^th(r,\tilde X^n_r)dr$, and between $\int_0^tb(r,X_r)dr$ and $\int_0^th(r,X_r)dr$, are at most $C\eta$ outside events of probability $\eta$, while $\int_0^th(r,\tilde X^n_r)dr\to\int_0^th(r,X_r)dr$ a.s. Moreover, on $\{\sup_r|\tilde X^n_r|\le R\}$ we have $|\int_0^t(b^n-b)(r,\tilde X^n_r)dr|\le\int_0^t|(b^n-b)\1_{B_R}|(r,\tilde X^n_r)dr$, whose expectation tends to $0$ by \eqref{est:Kry-b} and \eqref{eq:bn-conv}. Hence $\int_0^tb^n(r,\tilde X^n_r)dr\to\int_0^tb(r,X_r)dr$ in probability.
        \item Since $\sigma^n\to\sigma$ uniformly, $\sigma$ is bounded and uniformly continuous, and $(\tilde X^n,\tilde W^n)\to(X,W)$ uniformly a.s., the standard argument (approximating $\sigma(r,X_r)$ by step processes) gives $\int_0^t\sigma^n(r,\tilde X^n_r)d\tilde W^n_r\to\int_0^t\sigma(r,X_r)dW_r$ in probability.
    \end{enumerate}
    Passing to the limit in $$\tilde X^n_t=x+\int_0^tb^n(r,\tilde X^n_r)dr+\int_0^tg(r,\tilde X^n_r)d\gamma^n_r+\int_0^t\sigma^n(r,\tilde X^n_r)d\tilde W^n_r$$ shows that $(\tilde\Omega,\tilde\cF,(\tilde\cF_t),\tilde\cP;X,W)$ is a weak solution with the {\ky}  property. Its law is the limit of the laws of $X^n$ along the subsequence.
\end{proof}

\subsection{Uniqueness}
The key to uniqueness is the following conditional representation through the Young PDE.
\begin{lemma}\label{lem:rep}
   Let $(s,y)\in[0,T)\times\Rd$, and let $\bP\in\mathcal{M}_{L,g,\gamma}(s,y)$ satisfy the {\ky} property. Let $t_0\in(s,T]$, $\phi\in C^\infty_c(\Rd)$, and let $u\in\sol^0_{2,p}([0,t_0])$ be the solution, provided by \cref{thm:solution-YPDE}, to
   \begin{align}\label{eq:PDE-0-rep}
   d_tu+\frac{1}{2}a^{ij}\partial^2_{ij} udt+b\cdot \nabla udt+g\cdot \nabla u d\gamma_t=0 \text{ on } [0,t_0]\times \mR^d, \quad u(t_0)=\phi.
\end{align}
Then for every $r\in[s,t_0]$,
    \begin{align}\label{eq:rep-mar}
\bE^{\bP}\big(\phi(w_{t_0})\big|\cB_r\big)=u(r,w_r),\qquad \bP\text{-a.s.}
    \end{align}
In particular, $\bE^{\bP}\phi(w_{t_0})=u(s,y)$.
\end{lemma}
\begin{proof}
Let $(\gamma^n)\subset C^\infty([0,T])$ converge to $\gamma$ in $C^{\beta-}$, and let $u^n\in\sol^0_{2,p}([0,t_0])$ be the solution to \eqref{eq:PDE-0-rep} with $\gamma^n$ in place of $\gamma$ (\cref{thm:solution-YPDE}). Then $$\partial_tu^n+\frac12a^{ij}\partial^2_{ij}u^n=\varrho^n:=-(b+(\gamma^n)'g)\cdot\nabla u^n\in\LL^\infty_p([0,t_0])$$ by \cref{thm:solution-YPDE}, so that $\varrho^n+b\cdot\nabla u^n=-(\gamma^n)'g\cdot\nabla u^n$, and \cref{lem:MP-Sobolev} shows that, with $\rho_n:=g\cdot\nabla u^n$,
\begin{align*}
    M^{u^n}_{t_0}-M^{u^n}_r=\phi(w_{t_0})-u^n(r,w_r)+\int_r^{t_0}\rho_n(v,w_v)d\gamma^n_v-\big(D_{t_0}(\rho_n)-D_r(\rho_n)\big)
\end{align*}
has zero conditional expectation given $\cB_r$. By \cref{def:mar-sol} (1), $$\int_r^{t_0}\rho_nd\gamma^n-(D_{t_0}(\rho_n)-D_r(\rho_n))=\lim_m\int_r^{t_0}\rho_n(v,w_v)d(\gamma^n-\gamma^m)_v\quad \text{ in probability,}$$  so by \eqref{est:Kry-b-fin-M} and the conditional Fatou lemma,
\begin{align*}
    \Big|\bE^\bP\big(\phi(w_{t_0})\big|\cB_r\big)-u^n(r,w_r)\Big|\le C_{KY}T^{\theta_{KY}}\|\gamma^n-\gamma\|_{C^{\beta_0}}\|\rho_n\|_{\sol^0_{1,p}}.
\end{align*}
By \cref{Thm:est-roughPDE} (with $\beta_0$) and \cref{lem.prodS}, $$\sup_n\|\rho_n\|_{\sol^0_{1,p}}\lesssim\|g\|_{\sol_{1,p}}\sup_n\|u^n\|_{\sol^0_{2,p}}<\infty,$$ so the right-hand side tends to $0$. Moreover, $u^n-u$ solves \eqref{main-PDE} on $[0,t_0]$ with $\gamma^n$ in place of $\gamma$, $f=0$, zero terminal value, and the Young term $\rho d\bar\gamma$ given by $-g\cdot\nabla u\,d(\gamma^n-\gamma)$. Hence, by \cref{Thm:est-roughPDE} (with $\beta_0$), $$\sup_{r,x}|u^n(r,x)-u(r,x)|\lesssim\|\gamma^n-\gamma\|_{C^{\beta_0}}\|g\|_{\sol_{1,p}}\|u\|_{\sol^0_{2,p}}\to0.$$ Letting $n\to\infty$ gives \eqref{eq:rep-mar}.
\end{proof}

\begin{theorem}[Uniqueness]
   \label{thm:unique}
   For every $(s,y)\in[0,T)\times\Rd$ there is exactly one $\bP\in\mathcal{M}_{L,g,\gamma}(s,y)$ with the \ky property; we denote it by $\bP^{s,y}$, and write $\bP^x:=\bP^{0,x}$. Moreover:
   \begin{enumerate}[(i)]
       \item $\bP^{n,x}\to\bP^x$ weakly as $n\to\infty$;
       \item there exists a weak solution to \eqref{eq:rough-SDE} starting from $x$ with the \ky property, and the law of every such weak solution equals $\bP^x$. In particular, weak solutions with the \ky property are unique in law.
   \end{enumerate}
\end{theorem}
\begin{proof}
    \textit{Existence.} For $s=0$, existence is \cref{thm:existence} (1). For general $s$, apply \cref{thm:existence} (1) to the shifted coefficients $b(s+\cdot,\cdot)$, $\sigma(s+\cdot,\cdot)$, $g(s+\cdot,\cdot)$ and the shifted path $\gamma_{s+\cdot}$ on $[0,T-s]$, which satisfy the same assumptions, and shift the resulting measure back.

    \textit{Uniqueness.} Let $\bP^1,\bP^2\in\cM_{L,g,\gamma}(s,y)$ have the \ky property. We show by induction on $m$ that
    \begin{align}\label{eq:fdd}
        \bE^{\bP^1}\prod_{i=1}^m\psi_i(w_{t_i})=\bE^{\bP^2}\prod_{i=1}^m\psi_i(w_{t_i})
    \end{align}
    for all $s< t_1<\cdots<t_m\le T$ and all $\psi_1,\ldots,\psi_m\in C_c(\Rd)$. Since both measures are concentrated on paths that equal $y$ on $[0,s]$, this implies $\bP^1=\bP^2$. By uniform approximation, it suffices to consider $\psi_m\in C^\infty_c(\Rd)$. Let $u$ be the solution to \eqref{eq:PDE-0-rep} with $t_0=t_m$ and $\phi=\psi_m$; note that $u$ does not depend on $\bP^i$, and $u(r,\cdot)\in C_b(\Rd)$ for each $r$. If $m=1$, \cref{lem:rep} gives $\bE^{\bP^i}\psi_1(w_{t_1})=u(s,y)$ for $i=1,2$. If $m\ge2$, \cref{lem:rep} with $r=t_{m-1}$ gives
    \begin{align*}
        \bE^{\bP^i}\prod_{i'=1}^m\psi_{i'}(w_{t_{i'}})=\bE^{\bP^i}\Big[\prod_{i'=1}^{m-2}\psi_{i'}(w_{t_{i'}})\,\tilde\psi(w_{t_{m-1}})\Big],\qquad\tilde\psi:=\psi_{m-1}\,u(t_{m-1},\cdot)\in C_c(\Rd),
    \end{align*}
    and the right-hand side does not depend on $i$ by the induction hypothesis (with $m-1$ functions). This proves \eqref{eq:fdd}.

    (i) By \cref{prop:tightness}, \cref{thm:existence} (1) and uniqueness, every subsequence of $(\bP^{n,x})_n$ has a further subsequence converging to $\bP^x$.

    (ii) Existence is \cref{thm:existence} (2). If $X$ is a weak solution with the \ky property, then its law belongs to $\cM_{L,g,\gamma}(x)$ and has the \ky property, by \cref{prop:w-m}; hence it equals $\bP^x$.
\end{proof}

\begin{remark}
Compared with \cite[Theorem 4.19]{FHL21}, when all coefficients in \eqref{eq:rough-SDE} are deterministic, \cref{thm:unique} provides a well-posedness theory for the martingale problem under much weaker regularity conditions on $b$ and $g$, comparable to the classical work of Stroock and Varadhan \cite{SV}, at the price of the nondegeneracy and H\"older continuity of $\sigma$ and the restriction $\beta>\frac12$. Note that \cite{FHL21} also covers degenerate $\sigma$ and the rough regime $\beta>\frac13$.
\end{remark}

\subsection{Markov property}
Since the coefficients $b,\sigma,g$ and the driver $\gamma$ depend on time, the solution to \eqref{eq:rough-SDE} is a time-inhomogeneous Markov process. We state its Markov property in time-homogeneous form by passing to the space-time process. On the state space $\cS:=[0,T]\times\mR^d$ we define, for $h\ge0$ and $F\in C_b(\cS)$,
\begin{align}\label{def:Qh}
    Q_hF(s,y):=\bE^{\bP^{s,y}}F\big(s+h,w_{s+h}\big),\qquad (s,y)\in\cS,\ s+h\le T,
\end{align}
where $\bP^{s,y}$ is the measure from \cref{thm:unique}, for $s<T$, and $Q_0F:=F$.
\begin{theorem}\label{thm:Mark}
    For every $(s,y)\in[0,T)\times\mR^d$, every $r\in[s,T]$, every $h\ge0$ with $r+h\le T$ and every $F\in C_b(\cS)$, the function $Q_hF(r,\cdot)$ is Borel measurable and
\begin{align}\label{eq:Markov}
    \bE^{\bP^{s,y}}\big(F(r+h,w_{r+h})\,\big|\,\cB_r\big)=Q_hF(r,w_r),\qquad \bP^{s,y}\text{-a.s.}
\end{align}
In other words, under $\bP^{s,y}$ the space-time process $(t,w_t)_{t\in[s,T]}$ is a time-homogeneous Markov process on $\cS$ with transition operators $(Q_h)_{h\ge0}$.
Equivalently, if $(\tilde\Omega,\tilde{\mathcal{F}},(\tilde{\mathcal{F}}_t),\tilde{\mathcal{P}};X,W)$ is a weak solution to \eqref{eq:rough-SDE} starting from $x$ with the {\ky} property, then for $r,h,F$ as above,
\begin{align}
    \label{eq:Mark-weak}
    \mE^{\tilde\cP}\big(F(r+h,X_{r+h})\,\big|\,X_v,v\le r\big)=Q_hF(r,X_r),\qquad \tilde\cP\text{-a.s.}
\end{align}
\end{theorem}
\begin{proof}
For $h=0$ there is nothing to prove, so let $h>0$ and put $t_0:=r+h$ and $\phi:=F(t_0,\cdot)\in C_b(\Rd)$. First let $\phi\in C^\infty_c(\Rd)$, and let $u$ be the solution to \eqref{eq:PDE-0-rep} on $[0,t_0]$ with $u(t_0)=\phi$. By \cref{lem:rep} applied to $\bP^{s,y}$, the left-hand side of \eqref{eq:Markov} equals $u(r,w_r)$, and by \cref{lem:rep} applied to $\bP^{r,z}$ (with $r$ in place of $s$), $$Q_hF(r,z)=\bE^{\bP^{r,z}}\phi(w_{t_0})=u(r,z) \quad\forall z.$$  Hence \eqref{eq:Markov} holds, and $Q_hF(r,\cdot)=u(r,\cdot)$ is continuous. For general $\phi\in C_b(\Rd)$, choose $\phi_k\in C^\infty_c(\Rd)$ with $\sup_k\|\phi_k\|_{L_\infty}<\infty$ and $\phi_k\to\phi$ pointwise. By dominated convergence, $\bE^{\bP^{r,z}}\phi_k(w_{t_0})\to\bE^{\bP^{r,z}}\phi(w_{t_0})$ for every $z$, so $Q_hF(r,\cdot)$ is a pointwise limit of continuous functions, hence Borel; and the conditional dominated convergence theorem gives \eqref{eq:Markov}. Finally, \eqref{eq:Mark-weak} follows from \eqref{eq:Markov} with $(s,y)=(0,x)$, since the law of $X$ is $\bP^x$ by \cref{thm:unique}.
\end{proof}
\appendix
\section{Auxiliary proofs}\label{sec:app1}
Throughout this appendix $B^\nu=B^\nu_{p,\infty}$, $(\Delta_j)_{j\ge0}$ are the Littlewood--Paley blocks from \cref{sub:not}, $S_k:=\sum_{l<k}\Delta_l$, and $N$ denotes a constant depending only on the parameters of the respective statement. We use Bernstein's inequality: if $\hat g$ is supported in a ball of radius $C2^j$, then $\|\nabla^kg\|_{L_r}\le N2^{jk+jd(\frac1q-\frac1r)}\|g\|_{L_q}$ for $1\le q\le r\le\infty$ \cite[Lemma 2.1]{BCD}.

\begin{proof}[Proof of \cref{lem:semi-group-reg}]
Fix $z$ and write $P_{s,t}=P^{a(z)}_{s,t}$. Since $K^{-1}(t-s)I\le\int_s^ta(r,z)dr\le K(t-s)I$, the proof of \cite[Lemma 2.4]{BCD} gives $c,N$ depending only on $d,K$ such that, for $j\ge1$ and $g\in L_p$ with $\hat g$ supported in $\{2^{j-1}\le|\xi|\le2^{j+1}\}$,
\begin{align}\label{est:d-a}
    \|P_{s,t}g\|_{L_p}\le Ne^{-c4^j(t-s)}\|g\|_{L_p};
\end{align}
moreover $P_{s,t}$ is a contraction on $L_p$ and commutes with $\Delta_j$. Each estimate is proved for a single block $\Delta_jf$ uniformly in $z$; one then takes the supremum over $j$ (after the $L_p$-norm in $z$ in the barred case, by Minkowski's inequality). The block $j=0$ is bounded directly using $t-s\le T$.

(i) follows from \eqref{est:d-a} and $\sup_{x>0}x^{\mu/2}e^{-cx}<\infty$. (ii) follows from $$(P_{s,t}-I)\Delta_jf=\frac12\int_s^ta^{kl}(r,z)\partial^2_{kl}P_{s,r}\Delta_jf\,dr$$ and Bernstein's inequality, which give $\|(P_{s,t}-I)\Delta_jf\|_{L_p}\le N\min\{1,4^j(t-s)\}\|\Delta_jf\|_{L_p}\le N(t-s)^{\frac\mu2}2^{j\mu}\|\Delta_jf\|_{L_p}$. (iii) follows from (i), (ii) and $P_{\iota,t}-P_{\iota,s}=(P_{s,t}-I)P_{\iota,s}$. For (iv), by \eqref{est:d-a}, $2^{j(\nu+\mu)}\|\Delta_j\int_s^tP_{s,r}h_rdr\|_{L_p}\le N2^{j\mu}\min\{t-s,4^{-j}\}\sup_r\|h_r\|_{B^\nu}\le N(t-s)^{1-\frac\mu2}\sup_r\|h_r\|_{B^\nu}$. For (v), by \eqref{est:d-a} and H\"older's inequality with the conjugate exponent $q''$ of $q'$,
$$2^{j\mu}\Big\|\Delta_j\int_s^tP_{s,r}h_rdr\Big\|_{L_p}\le N2^{j\mu}\min\{(t-s)^{\frac1{q''}},4^{-\frac j{q''}}\}\|h\|_{L^{q'}([s,t];L_p)},$$
and $\min\{A,B\}\le A^{1-\vartheta}B^\vartheta$ with $\vartheta=\frac{\mu q''}2\in[0,1]$ (this is where $\mu\le2-\frac2{q'}$ is used) gives the claim.
\end{proof}

\begin{proof}[Proof of \cref{lem.prodS} (i), (ii) and (iv)]
We use Bony's decomposition $mF=T_mF+T_Fm+R(m,F)$, with $T_mF:=\sum_kS_{k-1}m\,\Delta_kF$ and $R(m,F):=\sum_{|k-l|\le1}\Delta_km\,\Delta_lF$; the Fourier support of the $k$-th term of $T_mF$ (and of $T_Fm$) is contained in an annulus of size $2^k$, and that of the $k$-th term of $R$ in a ball of radius $N2^k$ \cite[Section 2.8]{BCD}. Recall also that $\|m\|_{C^\varrho}\approx\sup_j2^{j\varrho}\|\Delta_jm\|_{L_\infty}$ for $\varrho\in(0,1)$, and $\le$ holds for $\varrho=1$.

(i) Fix $z$ and write $m=m^z$, $F=F^z$. For $T_mF$: $\|\Delta_jT_mF\|_{L_p}\le N\|m\|_{L_\infty}\sum_{|k-j|\le4}\|\Delta_kF\|_{L_p}$. For $T_Fm$: $\|\Delta_jT_Fm\|_{L_p}\le N\sum_{|k-j|\le4}\|S_{k-1}F\|_{L_p}\|\Delta_km\|_{L_\infty}$, with $\|S_{k-1}F\|_{L_p}\le\sum_{l<k}\|\Delta_lF\|_{L_p}$. For $R$: $\|\Delta_jR(m,F)\|_{L_p}\le N\sum_{k\ge j-3}\|\Delta_km\|_{L_\infty}\|\Delta_kF\|_{L_p}$, up to neighbouring indices. Each right-hand side is linear in the blocks of $F$, so after taking the $L_p$-norm in $z$ (Minkowski's inequality) we may replace $\|\Delta_kF\|_{L_p}$ by $(\int\|\Delta_kF^z\|^p_{L_p}dz)^{1/p}\le2^{-k\nu}\|F\|_{\bar B^\nu}$, and $\|\Delta_km\|_{L_\infty}$ by $2^{-k\varrho}\sup_z\|m^z\|_{C^\varrho}$. Multiplying by $2^{j\nu}$, the three terms are bounded by a constant times $\sup_z\|m^z\|_{C^\varrho}\|F\|_{\bar B^\nu}$, since
$\sum_{l<j}2^{(j-l)\nu}2^{-j\varrho}\le N(1+j)2^{-j(\varrho-\max\{\nu,0\})}$ is bounded because $\nu<\varrho$, and $\sum_{k\ge j-3}2^{(j-k)\nu-k\varrho}\le N2^{-j\varrho}$ because $\nu+\varrho>0$.

(ii) The same decomposition with $m=g$ (independent of $z$). The terms $T_gF$ and $R(g,F)$ are bounded as in (i), using $\|g\|_{L_\infty}\lesssim\|g\|_{L_{1,p}}$ and, for $R$, $\|\Delta_kg\|_{L_\infty}\le N2^{k(\frac dp-1)}\|g\|_{L_{1,p}}$ (Bernstein), which is summable against $2^{-k\nu}$ since $\nu+1-\frac dp>0$. For $T_Fg$ we use $\|S_{k-1}F^z\Delta_kg\|_{L_p}\le\|S_{k-1}F^z\|_{L_\infty}\|\Delta_kg\|_{L_p}$, $\|\Delta_kg\|_{L_p}\le N2^{-k}\|g\|_{L_{1,p}}$, and, by Bernstein and Minkowski,
\begin{align*}
    \Big(\int\|S_{k-1}F^z\|_{L_\infty}^pdz\Big)^{1/p}\le N\sum_{l<k}2^{l(\frac dp-\nu)}\|F\|_{\bar B^\nu}\le N\|F\|_{\bar B^\nu},
\end{align*}
since $\nu>\frac dp$. Thus $2^{j\nu}\|\Delta_jT_Fg\|\le N2^{j(\nu-1)}\|g\|_{L_{1,p}}\|F\|_{\bar B^\nu}$, which is bounded since $\nu\le1$.

(iv) Write $[a,\varphi_k]G:=a(\varphi_k\ast G)-\varphi_k\ast(aG)$, where $aG$ is defined by (i); then $[a,\varphi_k]G=\sum_l[a,\varphi_k]\Delta_lG$, the series converging in $L_p$ by the bounds below, and $[a,\varphi_k]g(x)=\int\varphi_k(x-y)(a(x)-a(y))g(y)dy$ for $g\in L_p$. By homogeneity we may assume $\|a\|_{C^\alpha}=\|G\|_{B^\nu}=1$, so that $\|\Delta_lG\|_{L_p}\le2^{-l\nu}$. Let $L\ge0$ be such that $2^L\le k<2^{L+1}$. We use two facts. First, since $\hat\varphi$ is a Schwartz function, if $\hat f$ is supported in $\{c2^l\le|\xi|\le C2^l\}$ with $l\ge1$, then
\begin{align}\label{est:phik-high}
    \|\varphi_k\ast f\|_{L_p}\le N\min\{1,k2^{-l}\}\|f\|_{L_p},
\end{align}
with $N$ depending on $c,C,\varphi$. Second, split $a=a_0+a_1$ with $a_0:=S_La$. Then $|a_0(x)-a_0(y)|\le N|x-y|^\alpha$; $\|\partial^\beta a_0\|_{L_\infty}\le N2^{L(|\beta|-\alpha)}\le Nk^{|\beta|-\alpha}$ for $1\le|\beta|\le2$ (by Bernstein's inequality, $\|\partial^\beta a_0\|_{L_\infty}\le N\sum_{l<L}2^{l(|\beta|-\alpha)}$, which gives the claim unless $|\beta|=\alpha=1$; in that case $\|\nabla a_0\|_{L_\infty}=\|S_L\nabla a\|_{L_\infty}\le N\|\nabla a\|_{L_\infty}\le N$, since the kernel of $S_L$ is bounded in $L_1$ uniformly in $L$); and $\|a_1\|_{L_\infty}\le\sum_{m\ge L}\|\Delta_ma\|_{L_\infty}\le N2^{-L\alpha}\le Nk^{-\alpha}$.

\textit{The part $a_0$.} For every $l$, $\|[a_0,\varphi_k]\Delta_lG\|_{L_p}\le N\int|\varphi_k(w)||w|^\alpha dw\,\|\Delta_lG\|_{L_p}\le Nk^{-\alpha}2^{-l\nu}$. For $l\ge1$ we can do better: $\Delta_lG=\Delta g_l$ with $g_l:=\Delta^{-1}\Delta_lG$, $\|g_l\|_{L_p}\le N4^{-l}2^{-l\nu}$, and integrating by parts,
\begin{align*}
    [a_0,\varphi_k]\Delta_lG(x)=\int\Delta_y\big[\varphi_k(x-y)(a_0(x)-a_0(y))\big]g_l(y)dy.
\end{align*}
The kernel is a sum of the terms $(\Delta\varphi_k)(x-y)(a_0(x)-a_0(y))$, $\nabla\varphi_k(x-y)\cdot\nabla a_0(y)$ and $\varphi_k(x-y)\Delta a_0(y)$ (up to constants), whose $L_1$-norms in $x$ and in $y$ are bounded by $Nk^{2-\alpha}$. By Schur's test, $\|[a_0,\varphi_k]\Delta_lG\|_{L_p}\le Nk^{2-\alpha}4^{-l}2^{-l\nu}$. Using the first bound for $l\le L$ and the second for $l>L$,
\begin{align*}
    \sum_l\|[a_0,\varphi_k]\Delta_lG\|_{L_p}\le Nk^{-\alpha}(1+L)2^{-L\nu}+Nk^{2-\alpha}\sum_{l>L}2^{-l(2+\nu)}\le Nk^{-(\alpha+\nu)}\log(1+k),
\end{align*}
since $\nu\le0$ and $2+\nu>0$.

\textit{The part $a_1$.} By \eqref{est:phik-high}, $\|a_1(\varphi_k\ast\Delta_lG)\|_{L_p}\le Nk^{-\alpha}\min\{1,k2^{-l}\}2^{-l\nu}$, and summing over $l$ gives at most $Nk^{-(\alpha+\nu)}\log(1+k)$, as for $a_0$ (now using $1+\nu>0$). Next, $a_1\Delta_lG=\sum_{m\ge L}\Delta_ma\,\Delta_lG$. There is an integer $N_0$ such that, for $m\le l-N_0$, the Fourier transform of $\Delta_ma\,\Delta_lG$ is supported in an annulus $\{c2^l\le|\xi|\le C2^l\}$; for these terms \eqref{est:phik-high} gives $\|\varphi_k\ast(\Delta_ma\,\Delta_lG)\|_{L_p}\le Nk2^{-l}2^{-m\alpha}2^{-l\nu}$, while for $m>l-N_0$ we simply use $\|\varphi_k\|_{L_1}\le N$. Hence
\begin{align*}
    \|\varphi_k\ast(a_1\Delta_lG)\|_{L_p}\le N\mathbf 1_{\{l\ge L+N_0\}}k^{1-\alpha}2^{-l(1+\nu)}+N2^{-\max\{L,l-N_0\}\alpha}2^{-l\nu}.
\end{align*}
Summing over $l$, the first term gives $Nk^{-(\alpha+\nu)}$ because $1+\nu>0$; the second gives $Nk^{-\alpha}(1+L)2^{-L\nu}$ for $l\le L+N_0$ and $N\sum_{l>L+N_0}2^{-l(\alpha+\nu)}\le Nk^{-(\alpha+\nu)}$ for $l>L+N_0$, because $\alpha+\nu>0$. Altogether, $\sum_l\|[a,\varphi_k]\Delta_lG\|_{L_p}\le Nk^{-(\alpha+\nu)}\log(1+k)$, which is \eqref{est:commutator}.
\end{proof}

\section{Auxiliary lemmas}
\begin{lemma}\label{lem.localization}
        Let $\nu\in\R$, $p\in(1,\infty)$, $r\in\{p,\infty\}$ and $\xi\in C^\infty_c(\Rd)$. Write $\bar B^\nu_{p,r}$ for the space $B^\nu_{p,r}(\R^d_x;L_p(\R^d_z))$, so that $\bar B^\nu_{p,\infty}=\bar B^\nu$. For a tempered distribution $v$ define $(\Xi v)^z:=\xi(\cdot-z)v$. Then
    \begin{align*}
        \|\Xi v\|_{\bar B^\nu_{p,r}}\le C\|v\|_{B^\nu_{p,r}},\qquad\|\Xi v\|_{\Lb_p}\le C\|v\|_{L_p},
    \end{align*}
    and if $\xi\not\equiv0$, also $\|v\|_{B^\nu_{p,r}}\le C\|\Xi v\|_{\bar B^\nu_{p,r}}$, where $C$ depends only on $d,\nu,p,\xi$. Consequently, $\|\Xi u\|_{\solb^\epsilon_{k,p}}\approx\|u\|_{\sol^\epsilon_{k,p}}$ for $k=1,2$ and $\|\Xi f\|_{\bar{\mathbb D}^\epsilon}\approx\|f\|_{\mathbb D^\epsilon}$.
\end{lemma}
\begin{proof}
The statement for $L_p$ is Fubini's theorem. Write $\xi^z:=\xi(\cdot-z)$; then $\Delta_m\xi^z=(\Delta_m\xi)(\cdot-z)$, and, since $\xi$ is a Schwartz function, $\|\Delta_m\xi\|_{L_p}+\|\Delta_m\xi\|_{L_{p'}}\le N_M2^{-mM}$ for every $M$. Fix $M>|\nu|$. There is an integer $N_0$ such that, for any functions $\eta,w$, $\Delta_j(\Delta_m\eta\,\Delta_lw)=0$ unless $j\le\max\{m,l\}+N_0$ and, in case $|m-l|>N_0$, also $|j-\max\{m,l\}|\le N_0$; let $I_j$ be the set of pairs $(m,l)$ allowed by these conditions. By Minkowski's inequality and Fubini's theorem,
\begin{align*}
    \Big(\int\|\Delta_j(\xi^zv)\|_{L_p}^pdz\Big)^{\frac1p}\le N\sum_{(m,l)\in I_j}\Big(\int\|\Delta_m\xi^z\,\Delta_lv\|_{L_p}^pdz\Big)^{\frac1p}=N\sum_{(m,l)\in I_j}\|\Delta_m\xi\|_{L_p}\|\Delta_lv\|_{L_p}.
\end{align*}
Hence $2^{j\nu}(\int\|\Delta_j(\xi^zv)\|_{L_p}^pdz)^{1/p}\le\sum_lK_{jl}2^{l\nu}\|\Delta_lv\|_{L_p}$ with $K_{jl}:=N2^{(j-l)\nu}\sum_{m:(m,l)\in I_j}2^{-mM}$. Checking the three cases ($l>m+N_0$, so that $|j-l|\le N_0$; $m>l+N_0$, so that $|j-m|\le N_0$; $|m-l|\le N_0$, so that $j\le l+2N_0$) one finds $K_{jl}\le N2^{-(M-|\nu|)|j-l|}$. Since $\sup_j\sum_lK_{jl}+\sup_l\sum_jK_{jl}<\infty$, the matrix $(K_{jl})$ is bounded on $\ell^r$, and $\|\Xi v\|_{\bar B^\nu_{p,r}}\le C\|v\|_{B^\nu_{p,r}}$. Conversely, let $R\bar F:=\int\xi^z\bar F^zdz$. By H\"older's inequality in $z$, $\|\int\Delta_m\xi^z\,\Delta_l\bar F^zdz\|_{L_p}\le\|\Delta_m\xi\|_{L_{p'}}(\int\|\Delta_l\bar F^z\|_{L_p}^pdz)^{1/p}$, and $\Delta_jR\bar F=\sum_{(m,l)\in I_j}\Delta_j\int\Delta_m\xi^z\,\Delta_l\bar F^zdz$; the same summation gives $\|R\bar F\|_{B^\nu_{p,r}}\le C\|\bar F\|_{\bar B^\nu_{p,r}}$. Since $R\Xi v=\|\xi\|_{L_2}^2v$, the lower bound follows when $\xi\not\equiv0$. The remaining equivalences follow by applying these estimates at each fixed time and to time increments.
\end{proof}
\section*{Acknowledgment}
KL has been funded by the Engineering \& Physical Sciences Research Council (EPSRC) Grant EP/Y016955/1.  CL is supported by Deutsche Forschungsgemeinschaft (DFG) - Projektnummer 563883019.
We are grateful to Peter Friz (TU Berlin) for fruitful discussions and suggestions.
The revision of this manuscript was carried out with the assistance of Claude, an AI model developed by Anthropic, which was used for proofreading, for checking arguments and literature references, and for drafting parts of several proofs and remarks. All such contributions were reviewed by the authors, who take full responsibility for the content of the paper.  The motivations, key ideas and
results  are the sole outcomes from  authors.
\bibliographystyle{alpha}
\bibliography{references}
\end{document}